\documentclass[10pt,twoside,reqno]{amsart}
\usepackage[top=1.4in, bottom=1.4in, left=1in, right=1in]{geometry}
\usepackage{graphicx}
\usepackage{amsmath,amscd,amsfonts,amsthm,amssymb}
\usepackage{mathtools,latexsym,bbm,esint,commath,enumerate}
\usepackage{stmaryrd,fancyhdr,indentfirst,hyperref}
\usepackage{mathrsfs,float,caption, subcaption}
\numberwithin{figure}{section}
\usepackage[bottom]{footmisc}
\usepackage{xcolor}

\numberwithin{equation}{section}
\newtheorem{theorem}{Theorem}[section]

\newtheorem{proposition}[theorem]{Proposition}

\theoremstyle{definition}

\newcommand{\Rm}[1]{
  \textup{\uppercase\expandafter{\romannumeral#1}}
}

\newcommand{\diff}{\,\mathrm{d}}
\newcommand{\Dissip}{\mathcal{D}}

\newcommand{\N}{\mathbb{N}}
\newcommand{\nk}{{\nu, \kappa}}
\renewcommand{\P}{\mathbb{P}}
\newcommand{\pt}{\partial_t}
\newcommand{\px}{\partial_x}
\newcommand{\R}{\mathbb{R}}
\newcommand{\Rc}{\mathcal{R}}
\newcommand{\T}{\mathbb{T}}
\newcommand{\ve}{\varepsilon}

\newcommand{\vr}{\varrho}
\newcommand{\vs}{\varsigma}
\newcommand{\Z}{\mathbb{Z}}
\newcommand{\be}{\begin{equation}}
\newcommand{\ee}{\end{equation}}
\newcommand{\pa}{\partial}
\newcommand{\la}{\label}
\newcommand{\fr}{\frac}
\newcommand{\beg}{\begin}
\newcommand{\na}{\nabla}
\let\div\relax
\DeclareMathOperator{\div}{\mathrm{div}}

\title[The Nernst-Planck-Euler System]{Global Solutions of the Nernst-Planck-Euler Equations}
\date{\today}

\author{Mihaela Ignatova}
\address{Department of Mathematics, Temple University, Philadelphia, PA 19122}
\email{ignatova@temple.edu}
\author{Jingyang Shu}
\address{Department of Mathematics, Temple University, Philadelphia, PA 19122}
\email{ jyshu@temple.edu}

\begin{document}

\begin{abstract}
We consider the initial value problem for the Nernst-Planck equations coupled to the incompressible Euler equations in $\mathbb T^2$. We prove global existence of weak solutions for vorticity in $L^p$. We also obtain global existence and uniqueness of smooth solutions. We show that smooth solutions of the Nernst-Planck-Navier-Stokes equations converge to solutions of the Nernst-Planck-Euler equations as viscosity tends to zero. All the results hold for large data. 

\end{abstract}

\maketitle

\section{Introduction}

In this paper, we consider the Nernst-Planck (NP) system
\begin{align}
& \pt c_i + \nabla \cdot \left(u c_i - D_i \nabla c_i - z_i D_i c_i \nabla \Phi\right) = 0, \qquad i = 1, \dotsc, N, \label{NP}\\
& - \ve \Delta \Phi = \rho, \label{poisson0}\\
& \rho = \sum_{i = 1}^N z_i c_i, \label{charge}
\end{align}
coupled with the incompressible Euler equations
\begin{align}
& \pt u + u \cdot \nabla u + \nabla p = - (k_B T_K) \rho \nabla \Phi, \label{euler0}\\
& \nabla \cdot u = 0, \label{incomp0}
\end{align}
where $c_i \colon \T^d \times [0, T] \to \R^+$ are the $i$-th ionic species concentrations, $z_i \in \Z$ are corresponding valences, $D_i > 0$ are constant diffusivities, $\Phi \colon \T^d \times [0, T] \to \R$ is the nondimensional electrical potential, $\rho \colon \T^d \times [0, T] \to \R$ is the nondimensional charge density, $u \colon \T^d \times [0, T] \to \T^d$ is the fluid velocity field, $p \colon \T^d \times [0, T] \to \R$ is the fluid pressure, $k_B > 0$ is the Boltzmann's constant, $T_K > 0$ is the (absolute) temperature, and $\ve > 0$ is a constant proportional to the square of the Debye length \cite{CI19,Rub90}. Here, $\T^d$ is a $d$-dimensional torus. 

We refer to the system \eqref{NP}--\eqref{incomp0} as the Nernst-Planck-Euler (NPE) equations. This system is obtained from the Nernst-Planck-Navier-Stokes (NPNS) equations by neglecting the kinematic viscosity term $\nu \Delta u$ in the momentum equations \eqref{euler0}. The NPNS equations are a classical model of the transport and the electrodiffusion of ions in homogeneous Newtonian fluids, with wide applications in physical chemistry, biology and engineering \cite{Rub90,Sch09}.

The initial value problem for the NPNS equations in the whole space $\R^d$ ($d = 2, 3$) is locally well posed \cite{Jer02}, and weak solutions are global \cite{LW20}. The equations have been studied in bounded domains. There are various boundary conditions for the ionic concentrations and for the electrical potential which influence the dynamics. Most known results concern vanishing flux (or blocking) boundary conditions for the ionic species.  

In two and three dimensions, with blocking boundary conditions and when the electric potential satisfies homogeneous Neumann boundary conditions, the NPNS equations have global weak solutions \cite{Sch09}. With the same boundary conditions for the ionic concentrations and homogeneous Dirichlet boundary conditions on the potential, in two dimensions, weak solutions exist globally for large data, and in three dimensions, global weak solutions are obtained for small data \cite{Ryh09p}. Global existence of weak solutions in three dimensional bounded domains with blocking boundary conditions is shown in \cite{FS17, JS09}. 

Global strong solutions are considered in two dimensions, with blocking boundary conditions for the concentrations and with Robin boundary conditions for the electric potential \cite{BFS14}. Also in two space dimensions, with blocking boundary conditions or uniform selective (special Dirichlet) boundary conditions for the concentrations and arbitrary inhomogeneous Dirichlet boundary conditions for the electrical potential, smooth solutions are proved to exist globally \cite{CI19} for arbitrary large data, and to converge to unique Boltzmann states. In three dimensions, under the same boundary conditions, nonlinear stability of Boltzmann states is obtained in \cite{CIL20pa}. In these two dimensional and three dimensional stable cases, interior electroneutrality ($\rho\to 0$) is established in the limit of vanishing Debye length in \cite{CIL20pc}.

The global existence of smooth solutions of three dimensional Nernst-Planck equations is not known in general, even if the equations are not coupled to fluids, or if they are coupled to the Stokes equation. Recent work \cite{CIL20pb} shows global regularity of solutions of the Nernst-Planck equations coupled to Stokes equations (or with Navier-Stokes equations, assuming the latter do not blow up) in three dimensions with general selective boundary conditions for large data.  

There are very few results on the NPE system \eqref{NP}--\eqref{incomp0}. A recent work \cite{QZ20p}  asserts that H\"{o}lder solutions may lose their regularity in finite time for equations which are purely inviscid (no ionic diffusivities). This work is based on the method of \cite{tarek} which is applicable to $C^{1,\alpha}$ solutions of the Euler equations that are not Lipschitz.

The Euler equations in two dimensions have well known global weak solutions \cite{Yudovich} if the vorticity is in $L^p$, $1<p<\infty$. If $p=\infty$ the solutions are unique. These results rely heavily on the 2D transport structure of the vorticity equation. If a force is driving the Euler equations, then the global existence of solutions depends on the nature of the forcing. The NPE system provides a natural example of physical forcing of the Euler equations, with the force being the Lorentz force, in the absence of a magnetic field.

In this paper we prove global existence for the NPE equations in $\T^2$ and show that they are the zero viscosity limit of NPNS equations. We focus on the initial value problem for the NPE equations \eqref{NP}--\eqref{incomp0} in two space dimensions with two ionic species ($N = 2$) with opposite valences ($z_1 = - z_2 = 1$) and with equal diffusivities ($D_1 = D_2 = D$). The initial data of the system is
\begin{align}
& c_i(\cdot , 0) = c_i(0), \quad i = 1, 2,\label{initc}\\
& u(\cdot ,0) = u_0, \label{initu}
\end{align}
where the ionic concentrations are nonnegative, $c_i(0) \geq 0$, and the electric charge obeys
\begin{align}
\label{neut}
\int_{\T^2} \rho(x, 0) \diff{x} = \sum_{i = 1}^2 \int_{\T^2} z_i c_i(x, 0) \diff{x} = \int_{\T^2} c_1(x, 0) - c_2(x, 0) \diff{x} = 0.
\end{align}
It follows from \eqref{NP} that the property \eqref{neut} is preserved in time.
For regular solutions of NPNS it is shown in \cite{CI19, CIL20pb} that if $c_i(0) \geq 0$, then $c_i(x, t)$ remains nonnegative for $t > 0$. This property follows from \eqref{NP} if $c_i$ are known to be sufficiently regular, and the same proof and result holds for the NPE equations.

In Section~\ref{sec:weak}, we address weak solutions of the NPE system with $\omega\in L^r$, $r\ge 2$. The NP equations are rewritten in terms of the charge density $\rho = c_1-c_2$ and the total concentration $\sigma = c_1 + c_2$. The parabolic nature of the NP equations and the transport structure of the Euler equations are crucially used, but they are not the sole reason for the global existence. In fact, it is doubtful that these would suffice in general. The additional important element of the proof is a coercivity in $L^2$ which is due to the positivity of the concentrations. This coercivity \eqref{L2estin} is the starting point for $L^p$ a~priori estimates for the charge density and total concentrations which do not involve the velocity in a quantitative fashion. Estimates in $W^{1,r}$ with $r\ge 2$ are then obtained for the ionic concentrations and the velocity. The a~priori estimates are performed on a regularized system \eqref{rhol}--\eqref{divul} given by a generalized vortex method coupled with a NP system driven by smoothed velocities of the vortex method. These equations have global existence and uniqueness. The choice of the generalized vortex method approximation is motivated by the fact that it preserves the transport nature of the vorticity equation.  The proof of existence and uniqueness is given in Appendix~\ref {sec:lwpreg} and itself is non-standard. There we use an iteration which preserves the nonnegativity of concentrations but loses the coercivity property. This iteration is only capable of producing short time existence and uniqueness of solutions of the vortex method NP system. Then we use the a~priori estimates and the local existence to globally extend the solutions of the vortex method NP system. The existence of weak solutions of the NPE equations is then obtained by removing the approximation. Uniqueness is lost, as it is for the Euler equations.

In Section~\ref{sec:strong}, we prove Theorem~\ref{thm:strongglobal}, which states the global existence and uniqueness of strong solutions which have
\[
\int_0^T \|\nabla \Delta \rho(\tau)\|_{L^2}^2 + \|\nabla \Delta \sigma(\tau)\|_{L^2}^2 \diff{\tau}
\]
finite. This bound is used to prove that the fluid vorticity $\omega = \nabla^\perp \cdot u$ is uniformly bounded. 

The inviscid limit is proved  in Section~\ref{sec:vanish}. In Theorem~\ref{thm:vanish} we show that strong solutions of the NPNS equations converge to strong solutions of the NPE equations on any time interval $[0, T]$. The inviscid limits for lower Sobolev norms are obtained as in \cite{Con86, CF88}. For the higher Sobolev norms, we make use of a regularization method \cite{BS75, CW96, Mas07}. The idea is to regularize the initial data and show that the regularized solution converges to the non regularized NPE solution when the regularization parameter tends to zero, while the difference between the regularized solution and the NPNS solution tends to zero as the fluid kinematic viscosity vanishes. We choose carefully the regularization parameter in terms of the viscosity to offset the loss of derivatives.

\section{Global existence of weak solutions}
\label{sec:weak}

In this section, we prove the global existence of weak solutions to the initial value problem \eqref{NP}--\eqref{initu}.

In the sequel we omit the integration domain $\T^2$ and write $\int f = \int_{\T^2} f(x) \diff{x}$. We denote 
\[
\bar{f} = \frac{1}{|\T^2|} \int f,
\]
the average of a function $f$ over the torus $\T^2$. In inequalities, $C$ and $C'$ denote constants which may change from line to line.

We denote $\rho = c_1 - c_2$, $\sigma = c_1 + c_2$.  Using \eqref{incomp0}, the Nernst-Planck system \eqref{NP} is equivalent to the equations
\begin{align}
\begin{split}
\pt \rho & = - u \cdot \nabla \rho + D (\Delta \rho + \nabla \sigma \cdot \nabla \Phi + \sigma \Delta \Phi),
\end{split} \label{rhot}\\
\begin{split}
\pt \sigma & = - u \cdot \nabla \sigma + D (\Delta \sigma + \nabla \rho \cdot \nabla \Phi + \rho \Delta \Phi).
\end{split} \label{sigmat0}
\end{align}
We have from \eqref{poisson0} that
\begin{align}
\label{poisson}
- \ve \Delta \Phi = \rho,
\end{align}
and from \eqref{euler0}--\eqref{incomp0} that
\begin{align}
& \pt u + u \cdot \nabla u + \nabla p = - (k_B T_K) \rho \nabla \Phi, \label{euler}\\
& \nabla \cdot u = 0. \label{incomp}
\end{align}
The system \eqref{rhot}--\eqref{incomp} has initial data from \eqref{initc}--\eqref{initu},
\begin{align}
\label{initial}
\begin{split}
& \rho(\cdot, 0) = \rho(0) = c_1(0) - c_2(0),\\
& \sigma (\cdot, 0) = \sigma(0) = c_1(0) + c_2(0),\\
& u(\cdot, 0) = u_0,
\end{split}
\end{align}
and 
\[
c_1 = \frac{\sigma + \rho}{2} \quad \text{and} \quad c_2 = \frac{\sigma - \rho}{2}
\]
solve the original Nernst-Planck-Euler system \eqref{NP}--\eqref{incomp0}.

The main theorem of this section is as follows.
\begin{theorem}
\label{thm:weakglobal}
Let $\ve > 0$, $D > 0$, and $r \geq 2$. 
Let $c_1(0), c_2(0) \in W^{1, r}(\T^2)$ be nonnegative functions satisfying \eqref{neut}, and $u_0 \in W^{1, r}(\T^2)$ be divergence free. Then for any $T > 0$, there exist $c_1(x,t)\ge 0$, $c_2(x,t)\ge 0$ and $u(x,t)$ divergence free, such that $c_1-c_2=\rho$ and $c_1+c_2= \sigma$ obey $\rho, \sigma \in L^{\infty}([0, T]; W^{1, r}(\T^2)) \cap L^2(0, T; H^2(\T^2))$,  $u(x,t)$ obeys $u \in L^{\infty}([0, T]; W^{1, r}(\T^2))$, and $(\rho, \sigma, u)$ solve the initial value problem \eqref{rhot}--\eqref{initial} in the sense of distributions. The charge density $\rho$ and total concentration $\sigma$ satisfy the following bounds
\begin{align}
\label{solest1}
\begin{split}
& (i) \quad \|\rho(t)\|_{L^p} + \left\|\sigma(t) - \bar{\sigma}\right\|_{L^p} \leq C e^{- C' t}, \quad \quad \forall p\ge 2,\\
& (ii) \quad \|\nabla \Phi(t)\|_{L^\infty} \leq C e^{- C' t},\\
& (iii) \quad \left\|\nabla \rho(t)\right\|_{L^r} + \left\|\nabla \sigma(t)\right\|_{L^r} \leq C e^{C t},\\
& (iv) \quad \|\nabla \rho(t)\|_{L^2}^2 + \|\nabla \sigma(t)\|_{L^2}^2 + \int_0^t \|\Delta \rho(\tau)\|_{L^2}^2 + \|\Delta \sigma(\tau)\|_{L^2}^2 \diff{\tau} \leq C,
\end{split}
\end{align}
with  constants $C, C' > 0$ depending on $D$, $\ve$, $p$, $r$, and the initial data $\|\rho(0)\|_{L^p}$, $\|\sigma(0)\|_{L^p}$, and $\|u_0\|_{W^{1, r}}$. Moreover, the fluid velocity satisfies the bound
\begin{align}
\label{solest2}
\|\nabla u(t)\|_{L^r} \leq \|u_0\|_{W^{1, r}} + C e^{C'' t},
\end{align}
where $C'' \in \R$ is a constant depending on $\|\rho(0)\|_{L^p}$, $\|\sigma(0)\|_{L^p}$, and $\|u_0\|_{W^{1, r}}$.
\end{theorem}

The strategy of proof consists in proving global existence {\em{and uniqueness}} of solutions of an approximate system. This system preserves positivity of the concentrations and has uniformly the a~priori bounds listed above in \eqref{solest1}, \eqref{solest2}. The weak solutions are obtained by removing the approximation.

In the approximation we regularize the velocity field $u$. Let $J_\ell$ be convolution with a nonnegative mollifier with scale $\ell > 0$ and denote
\[
[u]_\ell = J_\ell * u.
\]
The approximation then consists of the full Nernst-Planck system with regularized velocity coupled with a vortex method approximation of the incompressible Euler equation 
\begin{align}
& \pt \rho^\ell = - [u]_\ell \cdot \nabla \rho^\ell + D (\Delta \rho^\ell + \nabla \sigma^\ell \cdot \nabla \Phi^\ell + \sigma^\ell \Delta \Phi^\ell), \label{rhol}\\
& \pt \sigma^\ell = - [u]_\ell \cdot \nabla \sigma^\ell + D (\Delta \sigma^\ell + \nabla \rho^\ell \cdot \nabla \Phi^\ell + \rho^\ell \Delta \Phi^\ell),
\la{sigel}\\
& - \ve \Delta \Phi^\ell = \rho^\ell, \label{poisel}\\
& \pt u^\ell + [u]_\ell \cdot \nabla u^\ell + [\na u]_{\ell}^*u^\ell + \nabla p^\ell = - (k_B T_K) \rho^\ell \nabla \Phi^\ell, \la{ul}\\
& \nabla \cdot u^\ell = 0, \label{divul}
\end{align}
with initial conditions
\begin{align}
\label{approxinitell}
\begin{split}
& \rho^\ell(\cdot, 0) = \rho(0) \in W^{1, r}(\T^2),\\
& \sigma^\ell(\cdot, 0) = \sigma(0) \in W^{1, r}(\T^2),\\
& u^\ell(\cdot, 0) = u_0 \in W^{1,r}(\T^2) \quad \text{and} \quad \div u_0 = 0.
\end{split}
\end{align}

The system \eqref{rhol}--\eqref{poisel} preserves the positivity $c_i^{\ell}\ge 0$,  its solutions exist globally, are unique and obey the a~priori estimates \eqref{solest1}, \eqref{solest2}. The existence and uniqueness of solutions of \eqref{rhol}--\eqref{approxinitell} is proved in Appendix~\ref{sec:lwpreg}. The proof uses a local existence via an iterative scheme, the a~priori bounds (\ref{solest1}), (\ref{solest2}), and a unique extension result to show global existence. Below we prove the uniform in $\ell$ a~priori bounds.

\beg{theorem}\la{unifellb} Let $\ell > 0$. Let $r\ge 2$, $p\ge2$, $\ve>0$, $T>0$.  The system \eqref{rhol}--\eqref{divul} with initial data \eqref{approxinitell} has unique solutions $(\rho^\ell, \sigma^\ell, u^\ell) \in L^{\infty}([0, T]; W^{1, r}(\T^2)) \times L^{\infty}([0, T]; W^{1, r}(\T^2)) \times L^{\infty}([0, T]; W^{1,r}(\T^2))$ which obey \eqref{solest1}--\eqref{solest2} with constants independent of $\ell$.
\end{theorem}
\beg{proof} 
We show here the a~priori estimates on solutions  using the nonnegativity of $c_i^{\ell}$. For simplicity of notation we drop the $\ell$ superscript and write $[u]$ for the approximation of velocity. We consider first estimates in which the velocity $u$ does not participate in a quantitative fashion.

\vspace{1ex}\noindent{\bf Elliptic estimates.}
From the Poisson equation \eqref{poisel} and Sobolev embeddings, we obtain that
\begin{align}
\label{DPhiest}
\|\nabla \Phi\|_{L^\infty} \leq M_1 \|\rho\|_{L^3},
\end{align}
where $M_1 > 0$ is a constant depending only on $\ve$.

\vspace{1ex}\noindent{\bf $L^p$-bounds on ionic concentrations.} Because of the embedding $H^1\hookrightarrow L^p$, the initial concentrations are in $L^p$ for any $p$. We observe that \eqref{sigel} is equivalent to
\begin{align}
\label{sigmat}
\pt (\sigma - \bar{\sigma}) = - [u] \cdot \nabla (\sigma - \bar{\sigma}) + D \left(\Delta (\sigma - \bar{\sigma}) + \nabla \rho \cdot \nabla \Phi + \rho \Delta \Phi\right).
\end{align}
Note that the average $\bar{\sigma} \geq 0$ since $c_1, c_2 \geq 0$ and $\bar{\sigma}$ is conserved in time due to \eqref{sigel}.
Let $p \ge 2$. We multiply \eqref{rhol} by $\frac{1}{p - 1} \rho |\rho|^{p - 2}$ and \eqref{sigel} by $\frac{1}{p - 1} (\sigma - \bar{\sigma}) |\sigma - \bar{\sigma}|^{p - 2}$, and then integrate by parts,
\begin{align}
\begin{split}
& \frac{1}{p (p - 1)}\frac{\diff}{\diff{t}} \|\rho\|_{L^p}^p = - D \int |\rho|^{p - 2} |\nabla \rho|^2 - D \int |\rho|^{p - 2} (\sigma - \bar{\sigma}) \nabla \rho \cdot \nabla \Phi - D \int |\rho|^{p - 2} \bar{\sigma} \nabla \rho \cdot \nabla \Phi,
\end{split} \label{rholpinq}\\
& \frac{1}{p (p - 1)} \frac{\diff}{\diff{t}} \|\sigma - \bar{\sigma}\|_{L^p}^p = - D \int |\sigma - \bar{\sigma}|^{p - 2} |\nabla (\sigma - \bar{\sigma})|^2 - D \int |\sigma - \bar{\sigma}|^{p - 2} \rho \nabla(\sigma - \bar{\sigma}) \cdot \nabla \Phi. \label{sigmalpinq}
\end{align}
Taking $p = 2$, summing \eqref{rholpinq} and \eqref{sigmalpinq}, and using \eqref{poisson}, we obtain
\begin{align}
\label{L2esteq}
\frac{1}{2} \frac{\diff}{\diff{t}} \left(\|\rho\|_{L^2}^2 + \|\sigma - \bar{\sigma}\|_{L^2}^2\right) + D \left(\|\nabla \rho\|_{L^2}^2 + \|\nabla (\sigma - \bar{\sigma})\|_{L^2}^2\right) + \frac{D}{\ve} \int \sigma \rho^2 = 0.
\end{align}
Recall that the ionic concentrations $c_1, c_2 \geq 0$, so that $\sigma = c_1 + c_2 \geq 0$, and thus, the last term on the left hand side is nonnegative. Furthermore, since $|\rho| = |c_1 - c_2| \leq c_1 + c_2 = \sigma$, we have that
\begin{align}
\label{L2estin}
\frac{1}{2} \frac{\diff}{\diff{t}} \left(\|\rho\|_{L^2}^2 + \|\sigma - \bar{\sigma}\|_{L^2}^2\right) + D \left(\|\nabla \rho\|_{L^2}^2 + \|\nabla (\sigma - \bar{\sigma})\|_{L^2}^2\right) + \frac{D}{\ve} \|\rho\|_{L^3}^3 \leq 0.
\end{align}
Let $C_P$ denote the Poincar\'{e} constant in
\[
\left\|f - \bar{f}\right\|_{L^2} \leq C_P \|\nabla f\|_{L^2},
\]
then by the Poincar\'{e} inequality and Gr\"{o}nwall's inequality, we deduce the following exponential pointwise decay
\begin{align}
\label{l2decay}
\|\rho(t)\|_{L^2}^2 + \|\sigma(t) - \bar{\sigma}\|_{L^2}^2 \leq \left(\|\rho(0)\|_{L^2}^2 + \|\sigma(0) - \bar{\sigma}\|_{L^2}^2\right) e^{- \frac{2 D}{C_P^2} t},
\end{align}
and the bounds
\begin{align}
\label{l2time}
2 D \int_0^t \|\nabla \rho(\tau)\|_{L^2}^2 + \|\nabla \sigma(\tau)\|_{L^2}^2 \diff{\tau} + \frac{2 D}{\ve} \int_0^t \|\rho(\tau)\|_{L^3}^3 \diff{\tau} \leq \|\rho(0)\|_{L^2}^2 + \|\sigma(0) - \bar{\sigma}\|_{L^2}^2.
\end{align}

Now we use these bounds to obtain $L^p$ estimates. From \eqref{rholpinq} and \eqref{poisson} we have
\[
\frac{1}{p (p - 1)}\frac{\diff}{\diff{t}} \|\rho\|_{L^p}^p = - D \int |\rho|^{p - 2} |\nabla \rho|^2 - D \int |\rho|^{p - 2} (\sigma - \bar{\sigma}) \nabla \rho \cdot \nabla \Phi - \frac{D \bar{\sigma}}{(p - 1) \ve} \int |\rho|^p,
\]
then by H\"{o}lder's inequality and Young's inequality, we obtain
\begin{align*}
\frac{1}{p} \frac{\diff}{\diff{t}} \|\rho\|_{L^p}^p + \frac{D}{\ve} \bar{\sigma} \|\rho\|_{L^p}^p + \frac{D (p - 1)}{2} \int |\rho|^{p - 2} |\nabla \rho|^2 \leq \frac{D (p - 1)}{2} \|\rho\|_{L^p}^{p - 2} \|\sigma - \bar{\sigma}\|_{L^p}^2 \|\nabla \Phi\|_{L^\infty}^2.
\end{align*}
Discarding the third term on the left hand side of above inequality, dividing by $\|\rho\|_{L^p}^{p - 2}$ from both sides, and using the Gagliardo-Nirenberg interpolation inequality, we have that
\begin{align}
\label{rholpinq2}
\frac{\diff}{\diff{t}} \|\rho\|_{L^p}^2 + \frac{2 D}{\ve} \bar{\sigma} \|\rho\|_{L^p}^2 \leq D (p - 1) \|\sigma - \bar{\sigma}\|_{L^p}^2 \|\nabla \Phi\|_{L^\infty}^2 \leq D M_2 (p - 1) \|\nabla \sigma\|_{L^2}^{2 - \frac{4}{p}} \|\sigma - \bar{\sigma}\|_{L^2}^{\frac{4}{p}} \|\nabla \Phi\|_{L^\infty}^2,
\end{align}
where $M_2 > 0$ is an interpolation constant.

If $p = 3$, integrating \eqref{rholpinq2} in time and using \eqref{l2decay} and \eqref{DPhiest} we have
\begin{align}
&e^{\frac{2 D \bar{\sigma}}{\ve} t} \|\rho(t)\|_{L^3}^2
 \leq \|\rho(0)\|_{L^3}^2 \notag\\
&\quad\quad+ 2 D M_2 \left(\|\rho(0)\|_{L^2}^2 
 + \|\sigma(0) - \bar{\sigma}\|_{L^2}^2\right)^{\frac{2}{3}} \int_0^t e^{\left(\frac{2 D \bar{\sigma}}{\ve} - \frac{4 D}{3 C_P^2}\right) \tau} \|\nabla \sigma(\tau)\|_{L^2}^{\frac{2}{3}} \|\nabla \Phi(\tau)\|_{L^\infty}^2 \diff{\tau}.
\end{align}
By H\"{o}lder's inequality and \eqref{l2time}, we obtain
\begin{align}
&e^{\frac{2 D \bar{\sigma}}{\ve} t} \|\rho(t)\|_{L^3}^2
 \leq \|\rho(0)\|_{L^3}^2 \notag\\
&\quad 
+ 2 D M_1^2 M_2 \left(\|\rho(0)\|_{L^2}^2 + \|\sigma(0) - \bar{\sigma}\|_{L^2}^2\right)^{\frac{2}{3}} 
\sup_{\tau \in [0, t]} e^{\left(\frac{2 D \bar{\sigma}}{\ve} - \frac{4 D}{3 C_P^2}\right) \tau} 
\bigg(\int_0^t \|\nabla \sigma(\tau)\|_{L^2}^2 \diff{\tau}\bigg)^{\frac{1}{3}} 
\bigg(\int_0^t \|\rho(\tau)\|_{L^3}^3 \diff{\tau}\bigg)^{\frac{2}{3}}\notag\\
&\leq \|\rho(0)\|_{L^3}^2 + \ve^{\frac{2}{3}} M_1^2 M_2 \left(\|\rho(0)\|_{L^2}^2 + \|\sigma(0) - \bar{\sigma}\|_{L^2}^2\right)^2 \max\bigg\{1, e^{\left(\frac{2 D \bar{\sigma}}{\ve} - \frac{4 D}{3 C_P^2}\right) t}\bigg\}. \label{rhotL3}
\end{align}
Combining \eqref{rhotL3} with \eqref{DPhiest}, we deduce the exponential pointwise decay of $\|\nabla \Phi(t)\|_{L^\infty}$,
\begin{align}
\|\nabla \Phi(t)\|_{L^\infty}^2 & \leq M_1^2 \|\rho(t)\|_{L^3}^2\notag\\
& \leq M_1^2 \|\rho(0)\|_{L^3}^2 e^{- \frac{2 D \bar{\sigma}}{\ve} t} + \ve^{\frac{2}{3}} M_1^4 M_2 \left(\|\rho(0)\|_{L^2}^2 + \|\sigma(0) - \bar{\sigma}\|_{L^2}^2\right)^2 \max\bigg\{e^{- \frac{2 D \bar{\sigma}}{\ve} t}, e^{- \frac{4 D}{3 C_P^2} t}\bigg\}.
\label{Phidecay}
\end{align}

Next we use \eqref{Phidecay} back in \eqref{rholpinq2}. By similar arguments as in \eqref{rhotL3}, we obtain pointwise decay estimates for $\|\rho\|_{L^p}$,
\begin{align}
\|\rho(t)\|_{L^p}^2 & \leq \|\rho(0)\|_{L^p}^2 e^{- \frac{2 D \bar{\sigma}}{\ve} t} + C e^{- \frac{2 D \bar{\sigma}}{\ve} t} \bigg(\int_0^t \|\nabla \sigma(\tau)\|_{L^2}^2 \diff{\tau}\bigg)^{1 - \frac{2}{p}} \bigg(\int_0^t \max\bigg\{e^{- \frac{2 D}{C_P^2} p \tau}, e^{\left(\frac{D \bar{\sigma}}{\ve} - \frac{4 D}{3 C_P^2}\right) p \tau}\bigg\} \diff{\tau}\bigg)^{\frac{2}{p}} \notag\\
& \leq \|\rho(0)\|_{L^p}^2 e^{- \frac{2 D \bar{\sigma}}{\ve} t} + C \max\bigg\{e^{- \frac{2 D \bar{\sigma}}{\ve} t}, e^{- \frac{8 D}{3 C_P^2} t}\bigg\}, \label{rholpdecay}
\end{align}
for any $p \geq 2$, uniformly in time $t$. Here, the constant $C > 0$ changing from line to line depends on $\|\rho(0)\|_{L^3}$, $\|\sigma(0) - \bar{\sigma}\|_{L^2}$, $D$, $\ve$,  and $p$.

We now implement a Moser-type iteration argument (see e.g. \cite{BFS14, CL95, CIL20pa}) to obtain pointwise decays of $\|\sigma - \bar{\sigma}\|_{L^p}$ for $p > 2$. By \eqref{sigmalpinq}, we have
\begin{align*}
\frac{1}{p} \frac{\diff}{\diff{t}} \left\||\sigma - \bar{\sigma}|^{\frac{p}{2}}\right\|_{L^2}^2 + D (p - 1) \int |\sigma - \bar{\sigma}|^{p - 2} |\nabla (\sigma - \bar{\sigma})|^2 = - D (p - 1) \int |\sigma - \bar{\sigma}|^{p - 2} \rho \nabla (\sigma - \bar{\sigma}) \cdot \nabla \Phi.
\end{align*}
We use the bounds
\[
D (p - 1) \int |\sigma - \bar{\sigma}|^{p - 2} |\nabla (\sigma - \bar{\sigma})|^2 = \frac{4 D (p - 1)}{p^2} \int \left|\nabla |\sigma - \bar{\sigma}|^{\frac{p}{2}}\right|^2 \geq \frac{2 D}{p} \int \left|\nabla |\sigma - \bar{\sigma}|^{\frac{p}{2}}\right|^2
\]
and
\[
D  (p - 1) \int |\sigma - \bar{\sigma}|^{p - 2} \rho \nabla (\sigma - \bar{\sigma}) \cdot \nabla \Phi \leq 2 D \|\rho\|_{L^p} \|\nabla \Phi\|_{L^\infty} \left\||\sigma - \bar{\sigma}|^{\frac{p}{2}}\right\|_{L^2}^{\frac{p - 2}{p}} \left\|\nabla |\sigma - \bar{\sigma}|^{\frac{p}{2}}\right\|_{L^2}
\]
to deduce 
\[
\frac{\diff}{\diff{t}} \left\||\sigma - \bar{\sigma}|^{\frac{p}{2}}\right\|_{L^2}^2 + 2 D \left\|\nabla |\sigma - \bar{\sigma}|^{\frac{p}{2}}\right\|_{L^2}^2 \leq 2 D p \|\rho\|_{L^p} \|\nabla \Phi\|_{L^\infty} \left\||\sigma - \bar{\sigma}|^{\frac{p}{2}}\right\|_{L^2}^{\frac{p - 2}{p}} \left\|\nabla |\sigma - \bar{\sigma}|^{\frac{p}{2}}\right\|_{L^2}.
\]
By Young's inequality 
\begin{align}
\label{sigmaest0}
\frac{\diff}{\diff{t}} \left\||\sigma - \bar{\sigma}|^{\frac{p}{2}}\right\|_{L^2}^2 + D \left\|\nabla |\sigma - \bar{\sigma}|^{\frac{p}{2}}\right\|_{L^2}^2 \leq D p^2 \|\rho\|_{L^p}^2 \|\nabla \Phi\|_{L^\infty}^2 \left\||\sigma - \bar{\sigma}|^{\frac{p}{2}}\right\|_{L^2}^{2 - \frac{4}{p}}.
\end{align}
The Gagliardo-Nirenberg interpolation inequality and Young's inequality imply that
\begin{align}
\left\||\sigma - \bar{\sigma}|^{\frac{p}{2}}\right\|_{L^2}^2 & \leq M_3 \left[\left\|\nabla |\sigma - \bar{\sigma}|^{\frac{p}{2}}\right\|_{L^2} + \left\||\sigma - \bar{\sigma}|^{\frac{p}{2}}\right\|_{L^1}\right] \left\||\sigma - \bar{\sigma}|^{\frac{p}{2}}\right\|_{L^1}\notag\\
& \leq \delta \left\|\nabla |\sigma - \bar{\sigma}|^{\frac{p}{2}}\right\|_{L^2}^2 + \frac{M_3^2 + 4 M_3 \delta}{4 \delta} \left\||\sigma - \bar{\sigma}|^{\frac{p}{2}}\right\|_{L^1}^2,
\label{secondtermest}
\end{align}
where $M_3 > 0$ is the constant from the interpolation inequality and $\delta$ is given by
\begin{align}
\label{deltadef}
\delta = \left[p^2 \left(\|\rho(0)\|_{L^p}^2 + C (p - 1)\right) \left(M_1^2 \|\rho(0)\|_{L^3}^3 + \ve^{\frac{2}{3}} M_1^2 M_2 (\|\rho(0)\|_{L^2}^2 + \|\sigma(0) - \bar{\sigma}\|_{L^2}^2)^2\right) + 1\right]^{-1}.
\end{align}
Here, $C$ is the same as in \eqref{rholpdecay}, $M_1$ and $M_2$ are the same as in \eqref{Phidecay}. We note that $\delta < 1$.

Multiplying \eqref{secondtermest} by $\frac{D}{\delta}$, we get 
\begin{align}
\label{secondtermest2}
D \left\|\nabla |\sigma - \bar{\sigma}|^{\frac{p}{2}}\right\|_{L^2}^2  \geq \frac{D}{\delta} \left\||\sigma - \bar{\sigma}|^{\frac{p}{2}}\right\|_{L^2}^2 - \frac{D (M_3^2 + 4 M_3 \delta)}{4 \delta^2} \left\||\sigma - \bar{\sigma}|^{\frac{p}{2}}\right\|_{L^1}^2.
\end{align}
Thus, using \eqref{secondtermest2} in \eqref{sigmaest0} yields
\begin{align}
\label{beforeabsorb}
& \frac{\diff}{\diff{t}} \left\||\sigma - \bar{\sigma}|^{\frac{p}{2}}\right\|_{L^2}^2 + \frac{D}{\delta} \left\||\sigma - \bar{\sigma}|^{\frac{p}{2}}\right\|_{L^2}^2\notag\\
&\leq D p^2 \|\rho\|_{L^p}^2 \|\nabla \Phi\|_{L^\infty}^2 \left\||\sigma - \bar{\sigma}|^{\frac{p}{2}}\right\|_{L^2}^{2 - \frac{4}{p}} + \frac{D (M_3^2 + 4)}{4} \left\||\sigma - \bar{\sigma}|^{\frac{p}{2}}\right\|_{L^1}^2\notag\\
&\leq D p^2 \|\rho\|_{L^p}^2 \|\nabla \Phi\|_{L^\infty}^2 + D p^2 \|\rho\|_{L^p}^2 \|\nabla \Phi\|_{L^\infty}^2 \left\||\sigma - \bar{\sigma}|^{\frac{p}{2}}\right\|_{L^2}^2 + \frac{D (M_3^2 + 4 M_3 \delta)}{4 \delta^2} \left\||\sigma - \bar{\sigma}|^{\frac{p}{2}}\right\|_{L^1}^2,
\end{align}
where the last inequality follows from Young's inequality because $2 - \frac{4}{p} < 2$. The choice of $\delta$ in \eqref{deltadef} guarantees that the second term in the last line can be bounded as follows
\begin{align}
& D p^2 \|\rho\|_{L^p}^2 \|\nabla \Phi\|_{L^\infty}^2 \left\||\sigma - \bar{\sigma}|^{\frac{p}{2}}\right\|_{L^2}^2 \notag\\
&\leq  D p^2 \left(\|\rho(0)\|_{L^p}^2 + M_0 (p - 1)\right) \left(M_1^2 \|\rho(0)\|_{L^3}^3 + \ve^{\frac{2}{3}} M_1^2 M_2 (\|\rho(0)\|_{L^2}^2 + \|\sigma(0) - \bar{\sigma}\|_{L^2}^2)^2\right) \left\||\sigma - \bar{\sigma}|^{\frac{p}{2}}\right\|_{L^2}^2 \notag\\
&= \bigg(\frac{D}{\delta} - D\bigg) \left\||\sigma - \bar{\sigma}|^{\frac{p}{2}}\right\|_{L^2}^2, \label{absorb}
\end{align}
where we used the estimates \eqref{Phidecay} and \eqref{rholpdecay} and we bounded the exponential decay by $1$.
Therefore, \eqref{absorb} can be absorbed by the second term in the first line of \eqref{beforeabsorb}, which results in
\[
\frac{\diff}{\diff{t}} \left\|\sigma - \bar{\sigma}\right\|_{L^p}^p + D \left\|\sigma - \bar{\sigma}\right\|_{L^p}^p \leq D p^2 \|\rho\|_{L^p}^2 \|\nabla \Phi\|_{L^\infty}^2 + \frac{D (M_3^2 + 4 M_3 \delta)}{4 \delta^2} \left\|\sigma - \bar{\sigma}\right\|_{L^\frac{p}{2}}^p.
\]
Applying Gr\"{o}nwall's inequality then leads to
\begin{align}
\left\|\sigma(t) - \bar{\sigma}\right\|_{L^p}^p
 &\leq e^{- D t} \bigg[\left\|\sigma(0) - \bar{\sigma}\right\|_{L^p}^p 
 + D p^2 \int_0^t e^{D \tau} \|\rho(\tau)\|_{L^p}^p \|\nabla \Phi(\tau)\|_{L^\infty}^p \diff{\tau}\notag\\
 &\quad\quad\quad\quad
 + \frac{D (M_3^2 + 4 M_3 \delta)}{4 \delta^2} 
 \int_0^t e^{D \tau} \left\|\sigma(\tau) - \bar{\sigma}\right\|_{L^{\frac{p}{2}}}^p \diff{\tau}\bigg].
\label{sigmagronwall}
\end{align}
From \eqref{Phidecay} and \eqref{rholpdecay} it follows that
\begin{align*}
& e^{- \frac{D}{p} t} \bigg(D p^2 \int_0^t e^{D \tau} \|\rho(\tau)\|_{L^p}^p \|\nabla \Phi(\tau)\|_{L^\infty}^p \diff{\tau}\bigg)^{\frac{1}{p}}
\leq C e^{- \frac{D}{p} t} \bigg(\int_0^t e^{D \tau} \max\bigg\{e^{- \frac{2 D \bar{\sigma}}{\ve} p \tau}, e^{- \frac{2 D}{3 C_P^2} p \tau}\bigg\} \diff{\tau}\bigg)^{\frac{1}{p}}\\
&\leq C e^{- \frac{D}{p} t} \max \bigg\{\bigg(\int_0^t e ^{\left(D - \frac{2 D \bar{\sigma} p}{\ve}\right) \tau} \diff{\tau}\bigg)^{\frac{1}{p}}, \bigg(\int_0^t e ^{\left(D - \frac{2 D p}{3 C_P^2}\right) \tau} \diff{\tau}\bigg)^{\frac{1}{p}}\bigg\}
\leq C \max\bigg\{e^{- \frac{D}{p} t}, e^{- \frac{2 D \bar{\sigma}}{\ve} t}, e^{- \frac{2 D}{3 C_P^2} t}\bigg\},
\end{align*}
where $C > 0$ changes from line to line and is a constant depending on $\|\rho(0)\|_{L^p}$, $\|\sigma(0) - \bar{\sigma}\|_{L^2}$, $D$, $\ve$, and $p$.

We estimate the last integral in \eqref{sigmagronwall} by induction. We first recall that $\|\sigma(t) - \bar{\sigma}\|_{L^2}$ decays exponentially in time (cf.~\eqref{l2decay}). Then we assume $p = 2^{j + 1}$ for $j \in \N$ and deduce that $\|\sigma(t) - \bar{\sigma}\|_{L^{2^{j + 1}}}$ decays exponentially from the inductive assumption that $\|\sigma(t) - \bar{\sigma}\|_{L^{2^j}}$ decays exponentially. In fact, for each fixed $j \in \N$, $p = 2^{j + 1}$,
\[
\bigg(\frac{D (M_3^2 + 4 M_3 \delta)}{\delta^2}\bigg)^{\frac{1}{p}}
\]
is some finite positive constant that depends on $p$. Thus, since the integrand of the following integral grows slower than $e^{D p \tau}$, we have that
\[
e^{- \frac{D}{p} t} \bigg(\int_0^t e^{D \tau} \|\sigma(\tau) - \bar{\sigma}\|_{L^\frac{p}{2}}^p \diff{\tau}\bigg)^{\frac{1}{p}} \leq C e^{- C' t},
\]
for some $C, C' > 0$. Therefore, using this information, we deduce from \eqref{sigmagronwall} that $\|\sigma(t) - \bar{\sigma}\|_{L^p}$ decays exponentially for each fixed $p \geq 2$ of the form $p = 2^j$ ($j \in \N$). Then by interpolation, we obtain that $\|\sigma(t) - \bar{\sigma}\|_{L^p}$ decays exponentially for all $p \geq 2$,
\begin{align}
\label{sigmalpdecay}
\left\|\sigma(t) - \bar{\sigma}\right\|_{L^p} \leq C e^{- C' t},
\end{align}
for some constants $C, C' > 0$ depending on $\|\rho(0)\|_{L^3 \cap L^p}$, $\|\sigma(0) - \bar{\sigma}\|_{L^p}$, $D$, $\ve$, and $p$.
This concludes the proof of inequalities $(i)$ and $(ii)$ of \eqref{solest1}. 

\vspace{1ex}\noindent{\bf $L^2$-bounds on fluid vorticity.} We take the curl of \eqref{ul}. We denote $\omega = \nabla^\perp \cdot u$ with $\nabla^\perp = (- \partial_y, \px)$. The vortex approximation respects exactly the vorticity equation.  The vorticity equation is
\begin{align}
& \pt \omega + [u] \cdot \nabla \omega = - (k_B T_K) \nabla^\perp \rho \cdot \nabla \Phi \la{vort}\\
& u = \nabla^\perp \Delta^{-1} \omega. \label{BS}
\end{align}
It is well-known that 
\[
\|\nabla u\|_{L^r} \leq C \|\omega\|_{L^r}
\]
for $1 \leq r < \infty$.

Now we turn to the a~priori estimates for the vorticity equation \eqref{vort}.
Let $r \ge 2$. We multiply \eqref{vort} by $r \omega |\omega|^{r - 2}$, integrate by parts, and use H\"{o}lder's inequality
\begin{align*}
\frac{\diff}{\diff{t}} \|\omega\|_{L^r}^r & = - r k_B T_K \int \omega |\omega|^{r - 2} \nabla^\perp \rho \cdot \nabla \Phi\\
& \leq C \|\omega\|_{L^r}^{r - 1} \|\nabla \rho\|_{L^r} \|\nabla \Phi\|_{L^\infty},
\end{align*}
By Gr\"{o}nwall's inequality, we deduce 
\begin{align}
\label{omegalr0}
\|\omega(t)\|_{L^r} \leq \|\omega(0)\|_{L^r} + C \int_0^t \|\nabla \rho(\tau)\|_{L^r} \|\nabla \Phi(\tau)\|_{L^\infty} \diff{\tau}.
\end{align}
In particular, if $r = 2$, using \eqref{l2time} and \eqref{Phidecay}, we obtain
\begin{align}
\label{omegal2}
\begin{split}
\|\omega(t)\|_{L^2} & \leq \|\omega(0)\|_{L^2} + C \bigg(\int_0^t \|\nabla \rho(\tau)\|_{L^2}^2 \diff{\tau}\bigg)^{\frac{1}{2}} \bigg(\int_0^t \|\nabla \Phi(\tau)\|_{L^\infty}^2 \diff{\tau}\bigg)^{\frac{1}{2}}\\
& \leq \|\omega(0)\|_{L^2} + C,
\end{split}
\end{align}
with $C$ depending on $\|\rho(0)\|_{L^3}$, $\|\sigma(0) - \bar{\sigma}\|_{L^2}$, $\ve$, and $D$.

\vspace{1ex}\noindent {\bf $W^{1, r}$-bounds on ionic concentrations.} We now estimate $L^r$-norms of $\nabla \rho$ and $\nabla \sigma$. Taking gradients of \eqref{rhot} and \eqref{sigmat0}, we have 
\begin{align}
\label{gradeqn}
\begin{split}
\pt \nabla \rho & = - (\nabla [u])^* \nabla \rho - [u] \cdot \nabla \nabla \rho + D \nabla (\Delta \rho + \nabla \sigma \cdot \nabla \Phi + \sigma \Delta \Phi),\\
\pt \nabla \sigma & = - (\nabla [u])^* \nabla \sigma - [u] \cdot \nabla \nabla \sigma + D \nabla (\Delta \sigma + \nabla \rho \cdot \nabla \Phi + \rho \Delta \Phi),
\end{split}
\end{align}

At the $L^2$-level, we take the scalar product of the two equations in \eqref{gradeqn} with $\nabla \rho$ and $\nabla \sigma$ respectively, integrate by parts, and use \eqref{poisson} to get
\begin{align*}
& \frac{1}{2} \frac{\diff}{\diff{t}} \left(\|\nabla \rho\|_{L^2}^2 + \|\nabla \sigma\|_{L^2}^2\right) + D \left(\|\Delta \rho\|_{L^2}^2 + \|\Delta \sigma\|_{L^2}^2\right) + \frac{D}{\ve} \int \sigma |\nabla \rho|^2\\
&\quad= - \int \nabla \rho \cdot (\nabla [u])^* \nabla \rho - \int \nabla \sigma \cdot (\nabla [u])^* \nabla \sigma - D \int \Delta \rho (\nabla \sigma \cdot \nabla \Phi)\\
&\qquad\qquad - D \int \Delta \sigma (\nabla \rho \cdot \nabla \Phi) - \frac{D}{\ve} \int \rho \nabla \rho \cdot \nabla \sigma - \frac{2 D}{\ve} \int \rho |\nabla \rho|^2.
\end{align*}

Using H\"{o}lder's inequalities for $L^2$-$L^2$-$L^\infty$ or $L^2$-$L^4$-$L^4$, the elliptic estimates, the Gagliardo-Nirenberg interpolation inequality, and Young's inequality, we have
\begin{align}
\label{DrsL2}
\begin{split}
& \frac{\diff}{\diff{t}} \left(\|\nabla \rho\|_{L^2}^2 + \|\nabla \sigma\|_{L^2}^2\right) + \frac{D}{2} \left(\|\Delta \rho\|_{L^2}^2 + \|\Delta \sigma\|_{L^2}^2\right)\\
&\quad\leq C \left(\|\nabla u\|_{L^2} + \|\rho\|_{L^2}\right) \left(\|\nabla \rho\|_{L^4}^2 + \|\nabla \sigma\|_{L^4}^2\right) +  C \|\nabla \Phi\|_{L^\infty}^2 \left(\|\nabla \rho\|_{L^2}^2 + \|\nabla \sigma\|_{L^2}^2\right)\\
&\quad\leq C \left(\|\nabla u\|_{L^2} + \|\rho\|_{L^2}\right) \left(\|\nabla \rho\|_{L^2} \|\Delta \rho\|_{L^2} + \|\nabla \rho\|_{L^2}^2 + \|\nabla \sigma\|_{L^2} \|\Delta \sigma\|_{L^2} + \|\nabla \sigma\|_{L^2}^2\right)\\
&\quad \quad\quad + C \|\nabla \Phi\|_{L^\infty}^2 \left(\|\nabla \rho\|_{L^2}^2 + \|\nabla \sigma\|_{L^2}^2\right)\\
&\quad \leq  \frac{D}{4} \left(\|\Delta \rho\|_{L^2}^2 + \|\Delta \sigma\|_{L^2}^2\right) + C \left(1 + \|\omega\|_{L^2}^2 + \|\rho\|_{L^2}^2 + \|\nabla \Phi\|_{L^\infty}^2\right) \left(\|\nabla \rho\|_{L^2}^2 + \|\nabla \sigma\|_{L^2}^2\right).
\end{split}
\end{align}
We can absorb the first term in the last line of \eqref{DrsL2} into the dissipation term. Then we integrate \eqref{DrsL2} in time and apply \eqref{l2decay}, \eqref{l2time}, \eqref{Phidecay}, and \eqref{omegal2} to deduce that
\begin{align}
\la{nabl2}
\begin{split}
& \|\nabla \rho(t)\|_{L^2}^2 + \|\nabla \sigma(t)\|_{L^2}^2 + \frac{D}{4} \int_0^t \|\Delta \rho(\tau)\|_{L^2}^2 + \|\Delta \sigma(\tau)\|_{L^2}^2 \diff{\tau}\\
&\quad\leq \|\nabla \rho(0)\|_{L^2}^2 + \|\nabla \sigma(0)\|_{L^2}^2\\
&\qquad\qquad + C \sup_{\tau \in [0, t]} \left(1 + \|\omega(\tau)\|_{L^2}^2 + \|\rho(\tau)\|_{L^2}^2 + \|\nabla \Phi(\tau)\|_{L^\infty}^2\right) \int_0^t \|\nabla \rho(\tau)\|_{L^2}^2 + \|\nabla \sigma(\tau)\|_{L^2}^2 \diff{\tau}\\
\end{split}
\end{align}
and thus we have
\be
\label{nablal2bdd}
\|\nabla \rho(t)\|_{L^2}^2 + \|\nabla \sigma(t)\|_{L^2}^2 + \frac{D}{4} \int_0^t \|\Delta \rho(\tau)\|_{L^2}^2 + \|\Delta \sigma(\tau)\|_{L^2}^2 \diff{\tau}\le C
\ee
where $C > 0$ is a constant depending on the initial data.

If $r > 2$, we take the scalar product of the two equations in \eqref{gradeqn} with $\nabla \rho |\nabla \rho|^{r - 2}$ and $\nabla \sigma |\nabla \sigma|^{r - 2}$ respectively, integrate over $\T^2$, and integrate by parts to obtain
\begin{align}
\begin{split}
\frac{1}{r} \frac{\diff}{\diff{t}} \|\nabla \rho\|_{L^r}^r & = - \int |\nabla \rho|^{r - 2} \nabla \rho \cdot (\nabla [u])^* \nabla \rho - D \int |\nabla \rho|^{r - 2} |\nabla \nabla \rho|^2 - \frac{4 D (r - 2)}{r^2} \int \left|\nabla |\nabla \rho|^{\frac{r}{2}}\right|^2\\
& \qquad - D \int |\nabla \rho|^{r - 2} \Delta \rho \nabla \sigma \cdot \nabla \Phi - D \int |\nabla \rho|^{r - 2} \Delta \rho \sigma \Delta \Phi - D \int \nabla \rho \cdot \nabla |\nabla \rho|^{r - 2} \sigma \Delta \Phi\\
& \qquad  - D (r - 2) \int \nabla \rho \cdot (\nabla \nabla \rho) \cdot \nabla \rho |\nabla \rho|^{r - 4} \nabla \sigma \cdot \nabla \Phi
\end{split} \label{DrhoLp}
\end{align}
and
\begin{align}
\begin{split}
\frac{1}{r} \frac{\diff}{\diff{t}} \|\nabla \sigma\|_{L^r}^r & = - \int |\nabla \sigma|^{r - 2} \nabla \sigma \cdot (\nabla [u])^* \nabla \sigma - D \int |\nabla \sigma|^{r - 2} |\nabla \nabla \sigma|^2 - \frac{4 D (r - 2)}{r^2} \int \left|\nabla |\nabla \sigma|^{\frac{r}{2}}\right|^2\\
& \qquad - D \int |\nabla \sigma|^{r - 2} \Delta \sigma \nabla \rho \cdot \nabla \Phi - D \int |\nabla \sigma|^{r - 2} \Delta \sigma \rho \Delta \Phi - D \int \nabla \sigma \cdot \nabla |\nabla \sigma|^{r - 2} \rho \Delta \Phi\\
& \qquad  - D (r - 2) \int \nabla \sigma \cdot (\nabla \nabla \sigma) \cdot \nabla \sigma |\nabla \sigma|^{r - 4} \nabla \rho \cdot \nabla \Phi.
\end{split} \label{DsigmaLp}
\end{align}
For simplicity, we denote
\[
Y = \|\nabla \rho\|_{L^r}^r + \|\nabla \sigma\|_{L^r}^r = \|R\|_{L^2}^2 + \|S\|_{L^2}^2, \quad R = |\nabla \rho|^{\frac{r}{2}}, \quad S = |\nabla \sigma|^{\frac{r}{2}}.
\]
Adding \eqref{DrhoLp} to \eqref{DsigmaLp} and using \eqref{poisson}, H\"{o}lder's inequality, and Young's inequality, we obtain
\begin{align*}
 \frac{\diff}{\diff{t}} Y + \Dissip_1 \leq ~& r \int |\nabla [u]| (R^2 + S^2) + \frac{D r}{2} \int |\nabla \rho|^{r - 2} |\nabla \nabla \rho|^2 + \frac{D r}{2} \int |\nabla \sigma|^{r - 2} |\nabla \nabla \sigma|^2\\
& + 2 D r ((r - 2)^2 + 1) \|\nabla \Phi\|_{L^\infty}^2 \left(\|\nabla \rho\|_{L^r}^{r - 2} \|\nabla \sigma\|_{L^r}^2 + \|\nabla \sigma\|_{L^r}^{r - 2} \|\nabla \rho\|_{L^r}^2\right)\\
& + \frac{2 D r}{\ve^2} ((r - 2)^2 + 1) \int \rho^2 \left(|\nabla \rho|^{r - 2} \sigma^2 + |\nabla \sigma|^{r - 2} \rho^2\right),
\end{align*}
where $\Dissip_1$ is the dissipation term
\begin{align*}
\Dissip_1 = D r \int |\nabla \rho|^{r - 2} |\nabla \nabla \rho|^2  + D r \int |\nabla \sigma|^{r - 2} |\nabla \nabla \sigma|^2 + \frac{4 D (r - 2)}{r} \left(\|\nabla R\|_{L^2}^2 + \|\nabla S\|_{L^2}^2\right).
\end{align*}
Therefore, we have 
\begin{align}
\label{Yest1}
\begin{split}
\frac{\diff}{\diff{t}} Y + \Dissip_2 & \leq r \bigg(\int |\nabla [u]| (R^2 + S^2)\bigg) + 2 D r ((r - 2)^2 + 1) \|\nabla \Phi\|_{L^\infty}^2 Y\\
& \quad + \frac{2 D r}{\ve^2} ((r - 2)^2 + 1) \int \rho^2 (\rho^2 + \sigma^2) \left(R^{\frac{2 r - 4}{r}} + S^{\frac{2 r - 4}{r}}\right),
\end{split}
\end{align}
where
\begin{align*}
\Dissip_2 & = \frac{D r}{2} \int |\nabla \rho|^{r - 2} |\nabla \nabla \rho|^2  + \frac{D r}{2} \int |\nabla \sigma|^{r - 2} |\nabla \nabla \sigma|^2 + \frac{4 D (r - 2)}{r} \left(\|\nabla R\|_{L^2}^2 + \|\nabla S\|_{L^2}^2\right)\\
& \geq \frac{4 D (r - 2)}{r} \left(\|R\|_{H^1}^2 - \|R\|_{L^2}^2 + \|S\|_{H^1}^2 - \|S\|_{L^2}^2\right).
\end{align*}
We first note that from the Biot-Savart law \eqref{BS}, the Gagliardo-Nirenberg interpolation inequality, and Young's inequality
\begin{align}
\label{Yest1-1}
\begin{split}
\int |\nabla [u]| (R^2 + S^2) & \leq \|\nabla [u]\|_{L^2} \left(\|R\|_{L^4}^2 + \|S\|_{L^4}^2\right)\\
& \leq \|\nabla [u]\|_{L^2} \left(\|R\|_{L^2} \|\nabla R\|_{L^2} + \|R\|_{L^2}^2 + \|S\|_{L^2} \|\nabla S\|_{L^2} + \|S\|_{L^2}^2\right)\\
& \leq \frac{D (r - 2)}{r^2} \left(\|R\|_{H^1}^2 + \|S\|_{H^1}^2\right) + C \|\omega\|_{L^2}^2 \left(\|R\|_{L^2}^2 + \|S\|_{L^2}^2\right).
\end{split}
\end{align}
By H\"{o}lder's inequality with exponents $\fr{r}{2}$ and $\fr{2}{2-\fr{4}{r}}$, we have
\begin{align}
\label{Yest1-2}
\begin{split}
& \frac{2 D r}{\ve^2} ((r - 2)^2 + 1) \int \rho^2 (\rho^2 + \sigma^2) \left(R^{\frac{2 r - 4}{r}} + S^{\frac{2 r - 4}{r}}\right)\\
&\quad\leq C \Big(\|R\|_{L^2}^{2} + \|S\|_{L^2}^{2}\Big)^{\frac{r - 2}{r}} \Big(\|\rho\|_{L^{2r}} + \|\sigma\|_{L^{2r}}\Big)^{4}\\
&\quad\leq \|R\|_{L^2}^2 + \|S\|_{L^2}^2 + C \Big(\|\rho\|_{L^{2r}} + \|\sigma\|_{L^{2r}}\Big)^{2r}.
\end{split}
\end{align}
Using the inequalities \eqref{Yest1-1}, \eqref{Yest1-2} in \eqref{Yest1}, we get
\be
\frac{\diff}{\diff{t}} Y \leq C \left(1 + \|\nabla \Phi\|_{L^\infty}^2   + \|\omega\|_{L^2}^2\right) Y + \Big(\|\rho\|_{L^{2r}} + \|\sigma\|_{L^{2r}}\Big)^{2r}.
\la{odeY}
\ee
By the bounds \eqref{omegal2}, \eqref{rholpdecay}, and \eqref{sigmalpdecay} and Gr\"{o}nwall's inequality, we then deduce that $Y(t)$ has at most exponential growth in time $t > 0$,
\begin{align}
\begin{split}
&Y(t) =  \|\nabla \rho(t)\|_{L^r}^r + \|\nabla \sigma(t)\|_{L^r}^r\\
&\leq \exp\bigg(C \int_0^t 1 + \|\nabla \Phi(\tau)\|_{L^\infty}^2 + \|\omega(\tau)\|_{L^2}^2\diff{\tau}\bigg)\left[\|\nabla \rho(0)\|_{L^r}^r + \|\nabla \sigma(0)\|_{L^r}^r + \int_0^t\Big(\|\rho(\tau)\|_{L^{2r}} + \|\sigma(\tau)\|_{L^{2r}}\Big)^{2r}\diff{\tau}\right ]
\end{split} \notag\\
&\leq C e^{C' t},\label{Ygrowth}
\end{align}
where the constants $C, C' > 0$ depend on $\|\rho(0)\|_{L^3 \cap L^{2r}}$, $\|\sigma(0) - \bar{\sigma}\|_{L^{2r}}$, $D$, $\ve$, and $r$.

\vspace{1ex}\noindent {\bf $L^r$-bounds on fluid vorticity.} Using the bounds \eqref{Phidecay} and \eqref{Ygrowth} in \eqref{omegalr0} allows us to improve the global estimates for the vorticity $\omega$ from $L^2$ to $L^r$ for $r \geq 2$ and all $t > 0$,
\begin{align*}
\|\omega(t)\|_{L^r} \leq \|\omega(0)\|_{L^r} + C \int_0^t \|\nabla \rho(\tau)\|_{L^r} \|\nabla \Phi(\tau)\|_{L^\infty} \diff{\tau} \leq \|\omega(0)\|_{L^r} + C e^{C'' t},
\end{align*}
for some constant $C'' \in \R$ depending on the initial data.
\end{proof}

The proof of Theorem \ref{thm:weakglobal} follows from Theorem \ref{unifellb}.
Because the a~priori estimates \eqref{solest1}--\eqref{solest2} are uniform in $\ell$ we can find a sequence $\ell_m$ such that, as $m \to \infty$
\begin{align*}
& \rho^{\ell_m} \rightharpoonup \rho \quad \text{in} \ L^2(\T^2),\\
& \sigma^{\ell_m} \rightharpoonup \sigma \quad \text{in} \ L^2(\T^2),\\
& u^{\ell_m} \rightharpoonup u \quad \text{in} \ L^2(\T^2),
\end{align*}
for $t \geq 0$. Then using lower semi-continuity of the norms and the Aubin--Lions lemma, we complete the proof the theorem.

\beg{remark} Although the approximate system has unique solutions for each $\ell$, we do not know uniqueness of the weak solutions. In fact, these solutions correspond to weak solutions of 2D Euler with vorticity in $L^{r}$, $r<\infty$.
\end{remark}

\section{Global existence of strong solutions}
\label{sec:strong}

In this section, we establish the global in time existence and uniqueness of strong solutions of the Nernst-Planck-Euler system \eqref{rhot}--\eqref{incomp} with initial data \eqref{initial}. Our result, given in Theorem \ref{thm:strongglobal} below, is stated in the case when $\rho(0), \sigma(0), u_0 \in H^3$, but it also holds for any $H^s$ with $s > 2$, using a similar argument.

\begin{theorem}
\label{thm:strongglobal}
Let $\ve > 0$ and $D > 0$. Let $c_1(0), c_2(0) \in H^3(\T^2)$ be nonnegative functions satisfying \eqref{neut}, and $u_0 \in H^3(\T^2)$ be divergence free. Then for any $T > 0$, there exists a unique strong solution $\rho, \sigma \in L^{\infty}([0, T]; H^3(\T^2)) \cap L^2(0, T; H^4(\T^2))$ and $u \in L^{\infty}([0, T]; H^3(\T^2))$ to the initial value problem \eqref{rhot}--\eqref{initial}. In addition to the bounds \eqref{solest1} for the ionic concentrations, we also have for any $t > 0$
\begin{align}
\label{Hsbdd1}
\begin{split}
& \|\Delta \rho(t)\|_{L^2} + \|\Delta \sigma(t)\|_{L^2} + \int_0^t \|\nabla \Delta \rho(\tau)\|_{L^2}^2 + \|\nabla \Delta \sigma(\tau)\|_{L^2}^2 \leq C,\\
& \|\nabla \Delta \rho(t)\|_{L^2} + \|\nabla \Delta \sigma(t)\|_{L^2} + \int_0^t \|\Delta^2 \rho(\tau)\|_{L^2}^2 + \|\Delta^2 \sigma(\tau)\|_{L^2}^2 \leq C e^{C e^t},\\
\end{split}
\end{align}
where $C > 0$ depending only on $\ve$, $D$, and the initial data. For the fluid vorticity $\omega = \nabla^\perp \cdot u$ and velocity $u$, in addition to the estimates \eqref{solest2}, we also have for any $t > 0$
\begin{align}
\label{Hsbdd2}
\|\omega(t)\|_{L^\infty} \leq C \qquad \text{and} \qquad \|u(t)\|_{H^3} \leq e^{C e^{C t}}.
\end{align}
\end{theorem}

Before we proceed with the proof, we first recall several inequalities that will be used repeatedly. The Gagliardo-Nirenberg interpolation inequality for $L^4(\T^2)$ between $L^2(\T^2)$ and $\dot{H}^1(\T^2)$ is
\begin{align}
\label{L4interp}
\|f\|_{L^4} \leq C \|\nabla f\|_{L^2}^{\frac{1}{2}} \|f\|_{L^2}^{\frac{1}{2}} + C \|f\|_{L^2}.
\end{align}
For $s > 2$, the calculus Sobolev inequality for the $H^s(\T^2)$-norm ($s > 2$) of a product,
\begin{align}
\label{fgHm}
\|f g\|_{H^s} \leq \|f\|_{L^\infty} \|g\|_{H^s} + \|g\|_{L^\infty} \|f\|_{H^s},
\end{align}
and the Calder\'{o}n-Zygmund estimates \cite{BKM84, Kat86},
\begin{align}
\label{BKM}
\|\nabla u\|_{L^\infty} \leq C \|\omega\|_{L^\infty} \left(1 + \log\left(1 + \|u\|_{H^s}\right)\right).
\end{align}
Recall also the inequality due to Brezis-Gallouet \cite{BG80} and Brezis-Wainger \cite{BW80}
\begin{align}
\label{BG}
\|\nabla f\|_{L^\infty} \leq C \|\nabla f\|_{H^1} \bigg(1 +  \left[\log(1 + \|\nabla f\|_{H^2})\right]^{\frac{1}{2}}\bigg).
\end{align}

\begin{proof}[Proof of Theorem~\ref{thm:strongglobal}]
The embedding $H^3(\T^2) \hookrightarrow W^{1, p}(\T^2)$ (for $p \geq 1$) and Theorem~\ref{thm:weakglobal} imply the global existence of weak solutions together with the bounds \eqref{solest1}--\eqref{solest2} in the interval $[0, T]$ for any $T > 0$. We only need to show the propagation of $H^3$-regularity and the uniqueness of the solutions.  The construction of solutions is similar to the construction of weak solutions, see the proof of Theorem~\ref{thm:weakglobal}.

\vspace{1ex}\noindent {\bf $H^2$-bounds on ionic concentrations.} We multiply \eqref{rhot} and \eqref{sigmat0} by $\Delta^2 \rho$ and $\Delta^2 \sigma$, respectively, and integrate over $\T^2$. Integration-by-parts and \eqref{poisson} give
\begin{align}
\label{DtH2}
\begin{split}
& \frac{\diff}{\diff{t}} \left(\|\Delta \rho\|_{L^2}^2 + \|\Delta \sigma\|_{L^2}^2\right) + D \left(\|\nabla \Delta \rho\|_{L^2}^2 + \|\nabla \Delta \sigma\|_{L^2}^2\right) + \frac{D}{\ve} \int \sigma |\Delta \rho|^2\\
&\quad=  \int \nabla \rho \cdot (\nabla u \nabla \Delta \rho) + \int \nabla \sigma \cdot (\nabla u \nabla \Delta \sigma) + \int u \cdot (\nabla \nabla \rho \nabla \Delta \rho) + \int u \cdot (\nabla \nabla \sigma \nabla \Delta \sigma)\\
&\quad\quad + D \int \nabla \Delta \sigma \cdot \nabla \Phi \Delta \rho + 2 D \int \nabla \nabla \sigma : \nabla \nabla \Phi \Delta \rho - \frac{3 D}{\ve} \int \nabla \sigma \cdot \nabla \rho \Delta \rho\\
&\quad\quad + D \int \nabla \Delta \rho \cdot \nabla \Phi \Delta \sigma + 2 D \int \nabla \nabla \rho : \nabla \nabla \Phi \Delta \sigma - \frac{3 D}{\ve} \int |\nabla \rho|^2 \Delta \sigma - \frac{3 D}{\ve} \int \rho \Delta \sigma \Delta \rho\\
&\quad= I_{1, 1} + I_{1, 2} + I_{1, 3} + I_{1, 4} + I_{1, 5},
\end{split}
\end{align}
where
\begin{align*}
I_{1, 1} & = \int \nabla \rho \cdot (\nabla u \nabla \Delta \rho) + \int \nabla \sigma \cdot (\nabla u \nabla \Delta \sigma),\\
I_{1, 2} & = \int u \cdot (\nabla \nabla \rho \nabla \Delta \rho) + \int u \cdot (\nabla \nabla \sigma \nabla \Delta \sigma),\\
I_{1, 3} & = D \int \nabla \Delta \sigma \cdot \nabla \Phi \Delta \rho  + D \int \nabla \Delta \rho \cdot \nabla \Phi \Delta \sigma,\\
I_{1, 4} & = - \frac{3 D}{\ve} \int \nabla \sigma \cdot \nabla \rho \Delta \rho - \frac{3 D}{\ve} \int |\nabla \rho|^2 \Delta \sigma,\\
I_{1, 5} & = 2 D \int \nabla \nabla \sigma : \nabla \nabla \Phi \Delta \rho + 2 D \int \nabla \nabla \rho : \nabla \nabla \Phi \Delta \sigma - \frac{3 D}{\ve} \int \rho \Delta \sigma \Delta \rho.\\
\end{align*}

For the term $I_{1, 1}$, we apply H\"{o}lder's inequality, Young's inequality, and \eqref{BG} to obtain
\begin{align}
\begin{split}
I_{1, 1} & \leq \|\nabla u\|_{L^2} \left(\|\nabla \rho\|_{L^\infty} \|\nabla \Delta \rho\|_{L^2} + \|\nabla \sigma\|_{L^\infty} \|\nabla \Delta \sigma\|_{L^2}\right)\\
& \leq \frac{5}{2 D} \|\omega\|_{L^2}^2 \left(\|\nabla \rho\|_{L^\infty}^2 + \|\nabla \sigma\|_{L^\infty}^2\right) + \frac{D}{10} \left(\|\nabla \Delta \rho\|_{L^2}^2 + \|\nabla \Delta \sigma\|_{L^2}^2\right)\\
& \leq C \|\omega\|_{L^2}^2 \|\nabla \rho\|_{H^1}^2 \bigg(1 +  \left[\log(1 + \|\nabla \rho\|_{H^2})\right]\bigg) + C \|\omega\|_{L^2}^2 \|\nabla \sigma\|_{H^1}^2 \bigg(1 +  \left[\log(1 + \|\nabla \sigma\|_{H^2})\right]\bigg)\\
& \qquad  + \frac{D}{10} \left(\|\nabla \Delta \rho\|_{L^2}^2 + \|\nabla \Delta \sigma\|_{L^2}^2\right),
\end{split}
\end{align}
which implies
\begin{align}
\label{DtH2est1}
\begin{split}
I_{1,1}& \leq C \|\omega\|_{L^2}^2 \left(\|\nabla \rho\|_{H^1}^2 + \|\nabla \sigma\|_{H^1}^2\right) + C \|\omega\|_{L^2}^2 \left(\|\nabla \rho\|_{H^1}^2 \|\nabla \rho\|_{H^2} + \|\nabla \sigma\|_{H^2}^2 \|\nabla \sigma\|_{H^2}\right)\\
& \qquad  + \frac{D}{10} \left(\|\nabla \Delta \rho\|_{L^2}^2 + \|\nabla \Delta \sigma\|_{L^2}^2\right)\\
& \leq C \|\omega\|_{L^2}^2 \left(\|\nabla \rho\|_{H^1}^2 + \|\nabla \sigma\|_{H^1}^2\right) + C \|\omega\|_{L^2}^4 \left(\|\nabla \rho\|_{H^1}^4 + \|\nabla \sigma\|_{H^1}^4\right)\\
&\qquad + \frac{D}{5} \left(\|\nabla \Delta \rho\|_{L^2}^2 + \|\nabla \Delta \sigma\|_{L^2}^2\right).
\end{split}
\end{align}
To estimate the term $I_{1, 2}$, we use H\"{o}lder's inequality, Young's inequality, and the Gagliardo-Nirenberg interpolation inequality \eqref{L4interp},
\begin{align}
\label{DtH2est2}
\begin{split}
I_{1, 2} & \leq \|u\|_{L^4} \left(\|\nabla \nabla \rho\|_{L^4} \|\nabla \Delta \rho\|_{L^2} + \|\nabla \nabla \sigma\|_{L^4} \|\nabla \Delta \sigma\|_{L^2}\right)\\
& \leq \frac{5}{2 D} \|u\|_{L^4}^2 \left(\|\nabla \nabla \rho\|_{L^4}^2 + \|\nabla \nabla \sigma\|_{L^4}^2\right) + \frac{D}{10} \left(\|\nabla \Delta \rho\|_{L^2}^2 + \|\nabla \Delta \sigma\|_{L^2}^2\right)\\
& \leq C \|u\|_{L^4}^2 \left(\|\nabla \nabla \rho\|_{L^2} + \|\nabla \nabla \sigma\|_{L^2}\right) \left(\|\nabla \Delta \rho\|_{L^2} + \|\nabla \Delta \sigma\|_{L^2}\right) + \frac{D}{10} \left(\|\nabla \Delta \rho\|_{L^2}^2 + \|\nabla \Delta \sigma\|_{L^2}^2\right)\\
& \leq C \|u\|_{L^4}^2 \left(\|\nabla \nabla \rho\|_{L^2}^2 + \|\nabla \nabla \sigma\|_{L^2}^2\right) + \frac{D}{5} \left(\|\nabla \Delta \rho\|_{L^2}^2 + \|\nabla \Delta \sigma\|_{L^2}^2\right)\\
& \leq C \|u\|_{L^2} \left(\|u\|_{L^2} + \|\nabla u\|_{L^2}\right) \left(\|\nabla \nabla \rho\|_{L^2}^2 + \|\nabla \nabla \sigma\|_{L^2}^2\right) + \frac{D}{5} \left(\|\nabla \Delta \rho\|_{L^2}^2 + \|\nabla \Delta \sigma\|_{L^2}^2\right).
\end{split}
\end{align}
The estimates for the other terms in \eqref{DtH2} are similar. By H\"{o}lder's inequality, Young's inequality, the elliptic estimates, and interpolation \eqref{L4interp}, we obtain
\begin{align}
\label{DtH2est3}
\begin{split}
I_{1, 3} & \leq D \|\nabla \Phi\|_{L^\infty} \left(\|\nabla \Delta \sigma\|_{L^2} \|\Delta \rho\|_{L^2} + \|\nabla \Delta \rho\|_{L^2} \|\Delta \sigma\|_{L^2}\right)\\
& \leq \frac{5}{4} \|\nabla \Phi\|_{L^\infty}^2 \left(\|\Delta \rho\|_{L^2}^2 + \|\Delta \sigma\|_{L^2}^2\right) + \frac{D}{5} \left(\|\nabla \Delta \rho\|_{L^2}^2 + \|\nabla \Delta \sigma\|_{L^2}^2\right),
\end{split}
\end{align}
\begin{align}
\label{DtH2est4}
\begin{split}
I_{1, 4} & \leq \frac{3 D}{\ve} \|\nabla \rho\|_{L^4} \|\nabla \sigma\|_{L^4} \|\Delta \rho\|_{L^2} + \frac{3 D}{\ve} \|\nabla \rho\|_{L^4}^2 \|\Delta \sigma\|_{L^2}\\
& \leq C \left(\|\nabla \rho\|_{L^2} + \|\nabla \sigma\|_{L^2}\right) \left(\|\Delta \rho\|_{L^2}^2 + \|\Delta \sigma\|_{L^2}^2\right) \\
&\qquad+ C \left(\|\nabla \rho\|_{L^2}^2 + \|\nabla \sigma\|_{L^2}^2\right) \left(\|\Delta \rho\|_{L^2} + \|\Delta \sigma\|_{L^2}\right)\\
& \leq C \left(\|\nabla \rho\|_{L^2} + \|\nabla \sigma\|_{L^2}\right) \left(\|\Delta \rho\|_{L^2}^2 + \|\Delta \sigma\|_{L^2}^2\right)\\
&\qquad + C \left(\|\nabla \rho\|_{L^2}^4 + \|\nabla \sigma\|_{L^2}^4\right) + C \left(\|\Delta \rho\|_{L^2}^2 + \|\Delta \sigma\|_{L^2}^2\right),
\end{split}
\end{align}
and
\begin{align}
\label{DtH2est5}
\begin{split}
I_{1, 5} & \leq C \|\nabla \nabla \Phi\|_{L^2} \|\Delta \rho\|_{L^4} \|\Delta \sigma\|_{L^4} + C \|\rho\|_{L^2} \|\Delta \rho\|_{L^4} \|\Delta \sigma\|_{L^4}\\
& \leq C \|\rho\|_{L^2} \left(\|\Delta \rho\|_{L^2} + \|\Delta \sigma\|_{L^2}\right) \left(\|\nabla \Delta \rho\|_{L^2} + \|\nabla \Delta \sigma\|_{L^2}\right)\\
& \leq C \|\rho\|_{L^2}^2 \left(\|\Delta \rho\|_{L^2}^2 + \|\Delta \sigma\|_{L^2}^2\right) + \frac{D}{5} \left(\|\nabla \Delta \rho\|_{L^2}^2 + \|\nabla \Delta \sigma\|_{L^2}^2\right).
\end{split}
\end{align}
Using the estimates \eqref{DtH2est1} and \eqref{DtH2est2}--\eqref{DtH2est5} in \eqref{DtH2}, we conclude 
\begin{align}
\label{DtH2est02}
\begin{split}
& \frac{1}{2} \frac{\diff}{\diff{t}} \left(\|\Delta \rho\|_{L^2}^2 + \|\Delta \sigma\|_{L^2}^2\right) + \frac{D}{5} \left(\|\nabla \Delta \rho\|_{L^2}^2 + \|\nabla \Delta \sigma\|_{L^2}^2\right) + \frac{D}{\ve} \int \sigma |\Delta \rho|^2\\
&\quad\leq C \left(1 + \|\nabla \rho\|_{L^2} + \|\nabla \sigma\|_{L^2} + \|\omega\|_{L^2}^2 + \|u\|_{L^2}^2 + \|\nabla \Phi\|_{L^\infty}^2 + \|\rho\|_{L^2}^2\right) \left(\|\Delta \rho\|_{L^2}^2 + \|\Delta \sigma\|_{L^2}^2\right)\\
&\quad\qquad + C \left(\|\nabla \rho\|_{L^2}^4 + \|\nabla \sigma\|_{L^2}^4\right) + C \|\omega\|_{L^2}^4 \left(\|\nabla \rho\|_{H^1}^4 + \|\nabla \sigma\|_{H^1}^4\right).
\end{split}
\end{align}
For simplicity, we denote
\begin{align}
\label{ZW}
\begin{split}
Z & = \|\Delta \rho\|_{L^2}^2 + \|\Delta \sigma\|_{L^2}^2,\\
W_1 & = 1 + \|\nabla \rho\|_{L^2} + \|\nabla \sigma\|_{L^2} + \|\omega\|_{L^2}^2 + \|u\|_{L^2}^2 + \|\nabla \Phi\|_{L^\infty}^2 + \|\rho\|_{L^2}^2,\\
W_2 & = \|\nabla \rho\|_{L^2}^4 + \|\nabla \sigma\|_{L^2}^4,\\
W_3 & = \|\omega\|_{L^2}^4.
\end{split}
\end{align}
From \eqref{nablal2bdd}, it follows that
\begin{align}
\label{Ztimeintbdd}
\int_0^t Z(\tau) \diff{\tau} \leq C.
\end{align}
Using \eqref{l2decay}, \eqref{Phidecay}, \eqref{omegal2}, \eqref{nablal2bdd}, and the $L^2$-estimate for the velocity field $u$,
\[
\|u(t)\|_{L^2}^2 \leq \|u_0\|_{L^2}^2 + C \int_0^t \|\rho(\tau)\|_{L^2} \|\nabla \Phi(\tau)\|_{L^\infty} \diff{\tau} \leq C \int_0^t e^{- C \tau} \diff{\tau} \leq C,
\]
we obtain
\begin{align}
\sup_{\tau \in [0, t]} \left(W_1(\tau) + W_3(\tau)\right) \leq C.
\end{align}
The bounds \eqref{l2time} and \eqref{nablal2bdd} imply that
\begin{align}
\label{W2bdd}
\int_0^t W_2(\tau) \diff{\tau} \leq \sup_{\tau \in [0, t]} \left(\|\nabla \rho(\tau)\|_{L^2}^2 + \|\nabla \sigma(\tau)\|_{L^2}^2\right) \cdot \int_0^t \|\nabla \rho(\tau)\|_{L^2}^2 + \|\nabla \sigma(\tau)\|_{L^2}^2 \diff{\tau} \leq C.
\end{align}
Using the notation \eqref{ZW} and the fact that $\sigma \geq 0$, the inequality \eqref{DtH2est02} leads to
\[
\frac{\diff\left(\log(1 + Z)\right)}{\diff{t}} = \frac{1}{1 + Z} \frac{\diff{Z}}{\diff{t}} \leq C \frac{1}{1 + Z} \left(W_1 Z + W_2 + W_3 Z^2\right) \leq C \left(W_1 Z + W_2 + W_3 Z\right)
\]
for any $t > 0$.
Integrating this inequality in time, and applying the bounds \eqref{Ztimeintbdd}--\eqref{W2bdd}, we obtain
\begin{align*}
& \log\left(1 + Z(t)\right)\\
&\quad\leq \log\left(1 + Z(0)\right) + C \sup_{\tau \in [0, t]} W_1(\tau) \cdot \int_0^t Z(\tau) \diff{\tau} + C \int_0^t W_2(\tau) \diff{\tau} + C \sup_{\tau \in [0, t]} W_3(\tau) \cdot \int_0^t Z(\tau) \diff{\tau}\\
&\quad\leq C,
\end{align*}
which implies
\begin{align}
\label{deltal2bdd}
\|\Delta \rho(t)\|_{L^2}^2 + \|\Delta \sigma(t)\|_{L^2}^2 \leq C.
\end{align}
Going back to \eqref{DtH2est02}, we also deduce
\begin{align}
\label{L2H3bdd}
\int_0^t \|\nabla \Delta \rho(\tau)\|_{L^2}^2 + \|\nabla \Delta \sigma(\tau)\|_{L^2}^2 \leq C.
\end{align}

\vspace{1ex}\noindent {\bf $L^\infty$-bounds on fluid vorticity.} We take the limit $r \to \infty$ in \eqref{omegalr0} to get
\begin{align}
\label{omegaLinf}
\|\omega(t)\|_{L^\infty} \leq \|\omega_0\|_{L^\infty} + C \int_0^t \|\nabla \rho(\tau)\|_{L^\infty} \|\nabla \Phi(\tau)\|_{L^\infty} \diff{\tau}.
\end{align}
Now, we use the Sobolev embedding $H^2(\T^2) \hookrightarrow L^\infty(\T^2)$, the Cauchy-Schwarz inequality, \eqref{Phidecay}, \eqref{nablal2bdd}, and \eqref{L2H3bdd} to obtain
\begin{align*}
& \int_0^t \|\nabla \rho(\tau)\|_{L^\infty} \|\nabla \Phi(\tau)\|_{L^\infty} \diff{\tau}
\leq C \int_0^t \|\nabla \rho(\tau)\|_{H^2} \|\nabla \Phi(\tau)\|_{L^\infty} \diff{\tau}\\
&\quad\leq C \sup_{\tau \in [0, t]} \|\nabla \rho(\tau)\|_{L^2} \cdot \int_0^t \|\nabla \Phi(\tau)\|_{L^\infty} \diff{\tau} + C \bigg(\int_0^t \|\nabla \Delta \rho(\tau)\|_{L^2}^2 \diff{\tau}\bigg)^{\frac{1}{2}} \bigg(\int_0^t \|\nabla \Phi(\tau)\|_{L^\infty}^{2} \diff{\tau}\bigg)^{\frac{1}{2}}\leq C,
\end{align*}
which, along with \eqref{omegaLinf}, implies 
\begin{align}
\label{omegabdd}
\|\omega(t)\|_{L^\infty} \leq C.
\end{align}

\vspace{1ex}\noindent {\bf $H^3$-bounds on fluid velocity.} We test \eqref{euler} with $- \Delta^3 u$ and use the standard estimates for the incompressible Euler equations (see e.g. \cite{MB02}) as well as the calculus Sobolev inequality \eqref{fgHm} to obtain
\begin{align}
\label{Dalphau}
\begin{split}
\frac{\diff}{\diff{t}} \|\nabla \Delta u\|_{L^2}^2 & \leq C \|\nabla u\|_{L^\infty} \|u\|_{H^3}^2 + (k_B T_K) \int \partial^\alpha (\rho \nabla \Phi) \cdot \partial^\alpha u\\
& \leq C \|\nabla u\|_{L^\infty} \|u\|_{H^3}^2 + C \|u\|_{H^3} \left(\|\rho\|_{L^\infty} \|\nabla \Phi\|_{H^3} + \|\nabla \Phi\|_{L^\infty} \|\rho\|_{H^3}\right).
\end{split}
\end{align}
From the elliptic estimate $\|\nabla \Phi\|_{H^3} \leq C \left(\|\nabla \Phi\|_{L^2} + \|\Delta \rho\|_{L^2}\right)$,  Sobolev embeddings, and \eqref{BKM} it follows that
\begin{align*}
\frac{\diff}{\diff{t}} \|u\|_{H^3} & \leq C \|\nabla u\|_{L^\infty} \|u\|_{H^3} + C \left(\|\rho\|_{L^\infty} \|\nabla \Phi\|_{H^3} + \|\nabla \Phi\|_{L^\infty} \|\rho\|_{H^3}\right)\\
& \leq C \|\nabla u\|_{L^\infty} \|u\|_{H^3} + C \left(\|\rho\|_{L^\infty} \|\nabla \Phi\|_{L^2} + \|\rho\|_{L^\infty} \|\Delta \rho\|_{L^2} + \|\nabla \Phi\|_{L^\infty} \|\rho\|_{H^3}\right)\\
& \leq C \|\omega\|_{L^\infty} \left(1 + \log\left(1 + \|u\|_{H^3}\right)\right) \|u\|_{H^3} + C \left(\|\rho\|_{L^\infty} \|\Delta \rho\|_{L^2} + \|\nabla \Phi\|_{L^\infty} \|\rho\|_{H^3}\right).
\end{align*}
Therefore, we obtain
\begin{align}
\label{uH3ineq}
\frac{\diff \log\left(1 + \log(1 + \|u\|_{H^3})\right)}{\diff{t}} & = \frac{1}{1 + \log(1 + \|u\|_{H^3})} \frac{1}{1 + \|u\|_{H^3}} \frac{\diff}{\diff{t}} \|u\|_{H^3}\notag\\
& \leq C \|\omega\|_{L^\infty} + C \left(\|\rho\|_{L^\infty} \|\Delta \rho\|_{L^2} + \|\nabla \Phi\|_{L^\infty} \|\rho\|_{H^3}\right).
\end{align}
We integrate \eqref{uH3ineq} in time and use the bounds \eqref{Phidecay}, \eqref{nablal2bdd}, \eqref{L2H3bdd}, and \eqref{omegabdd} to deduce
\begin{align*}
& \log\left(1 + \log(1 + \|u(t)\|_{H^3})\right)\\
&\quad\leq \log\left(1 + \log(1 + \|u_0\|_{H^3})\right) + C \int_0^t \|\omega(\tau)\|_{L^\infty} \diff{\tau} + C \int_0^t \|\rho(\tau)\|_{L^\infty} \|\Delta \rho(\tau)\|_{L^2} + \|\nabla \Phi(\tau)\|_{L^\infty} \|\rho(\tau)\|_{H^3} \diff{\tau}\\
&\quad\leq \log\left(1 + \log(1 + \|u_0\|_{H^3})\right) + C \int_0^t \|\omega(\tau)\|_{L^\infty} \diff{\tau} + C \int_0^t \|\rho(\tau)\|_{H^2}^2 \diff{\tau}\\
& \qquad + C \bigg(\int_0^t \|\nabla \Phi(\tau)\|_{L^\infty}^2 \diff{\tau}\bigg)^{\frac{1}{2}} \bigg(\int_0^t \|\rho(\tau)\|_{H^3}^2 \diff{\tau}\bigg)^{\frac{1}{2}}\\
&\quad\leq C (1 + t).
\end{align*}
Hence, we conclude the double exponential bound of $\|u(t)\|_{H^3}$,
\begin{align}
\label{uH3bdd}
\|u(t)\|_{H^3} \leq e^{C e^{C t}}.
\end{align}

\vspace{1ex}\noindent {\bf $H^3$-bounds on ionic concentrations.}
We multiply \eqref{rhot} and \eqref{sigmat0} by $- \Delta^3 \rho$ and $- \Delta^3 \sigma$ respectively, integrate over $\T^2$. We integrate by parts and use \eqref{poisson} to obtain
\begin{align}
\label{DtH3}
\begin{split}
& \frac{1}{2} \frac{\diff}{\diff{t}} \left(\|\nabla \Delta \rho\|_{L^2}^2 + \|\nabla \Delta \sigma\|_{L^2}^2\right) + D \left(\|\Delta^2 \rho\|_{L^2}^2 + \|\Delta^2 \sigma\|_{L^2}^2\right) + \frac{D}{\ve} \int \sigma |\nabla \Delta \rho|^2\\
&\quad= I_{2, 1} + I_{2, 2} + I_{2, 3} + I_{2, 4} + I_{2, 5} + I_{2, 6},
\end{split}
\end{align}
where
\begin{align*}
I_{2, 1} & = \int \Delta u \cdot (\nabla \rho \Delta^2 \rho + \nabla \sigma \Delta^2 \sigma) + \int u \cdot (\nabla \Delta \rho \Delta^2 \rho + \nabla \Delta \sigma \Delta^2 \sigma) + 2 \int \nabla u : (\nabla \nabla \rho \Delta^2 \rho + \nabla \nabla \sigma \Delta^2 \sigma),\\
I_{2, 2} & = D \int \nabla \Delta \rho \cdot \nabla \nabla \Delta \sigma \nabla \Phi + D \int \nabla \Delta \sigma \cdot \nabla \nabla \Delta \rho \nabla \Phi,\\
I_{2, 3} & = - \frac{2 D}{\ve} \int \nabla \Delta \rho \cdot (\nabla \nabla \rho \nabla \sigma) - \frac{5 D}{\ve} \int \nabla \Delta \sigma \cdot (\nabla \nabla \rho \nabla \rho) - \frac{3 D}{\ve} \int \nabla \Delta \rho \cdot (\nabla \nabla \sigma \nabla \rho)\\
& \quad - \frac{D}{\ve} \int \nabla \Delta \rho \cdot \nabla \rho \Delta \sigma - \frac{3 D}{\ve} \int \nabla \Delta \rho \cdot \nabla \sigma \Delta \rho - \frac{4 D}{\ve} \int \nabla \Delta \sigma \cdot \nabla \rho \Delta \rho,\\
I_{2, 4} & = 2 D \int \nabla \rho \cdot (\nabla \nabla \nabla \sigma : \nabla \nabla \Phi) + 2 D \int \nabla \sigma \cdot (\nabla \nabla \nabla \rho : \nabla \nabla \Phi),\\
I_{2, 5} & = - \frac{3 D}{\ve} \int \nabla \Delta \rho \cdot \nabla \Delta \sigma \rho,\\
I_{2, 6} & = 2 D \int \nabla \rho \cdot (\nabla \nabla \nabla \Phi : \nabla \nabla \sigma) + 2 D \int \nabla \sigma \cdot (\nabla \nabla \nabla \Phi : \nabla \nabla \rho).
\end{align*}
The estimates of these terms are similar to the estimates of the terms in \eqref{DtH2}. For the terms involving velocity $u$, we use H\"{o}lder's inequalities for $L^2$-$L^2$-$L^\infty$ or $L^2$-$L^4$-$L^4$, the Sobolev embedding $H^2(\T^2) \hookrightarrow L^\infty(\T^2)$, the interpolation inequality, and Young's inequality,
\begin{align}
\begin{split}
I_{2, 1} & \leq \|\Delta u\|_{L^2} \left(\|\nabla \rho\|_{L^\infty} \|\Delta^2 \rho\|_{L^2} + \|\nabla \sigma\|_{L^\infty} \|\Delta^2 \sigma\|_{L^2}\right) + \|u\|_{L^\infty} \left(\|\Delta^2 \rho\|_{L^2} \|\nabla \Delta \rho\|_{L^2} + \|\Delta^2 \sigma\|_{L^2} \|\nabla \Delta \sigma\|_{L^2}\right)\\
& \quad+ 2 \|\nabla u\|_{L^4} \left(\|\nabla \nabla \rho\|_{L^4} \|\Delta^2 \rho\|_{L^2} + \|\nabla \nabla \sigma\|_{L^4} \|\Delta^2 \sigma\|_{L^2}\right)\\
& \leq \frac{5}{4 D} \|\Delta u\|_{L^2}^2 \left(\|\nabla \rho\|_{L^\infty}^2 + \|\nabla \sigma\|_{L^\infty}^2\right) + C \|u\|_{H^2}^2 \left(\|\nabla \Delta \rho\|_{L^2}^2 + \|\nabla \Delta \sigma\|_{L^2}^2\right)\end{split} \notag\\
& \quad + \frac{3 D}{5} \left(\|\Delta^2 \rho\|_{L^2}^2 + \|\Delta^2 \sigma\|_{L^2}^2\right) + C \left(\|\nabla \nabla u\|_{L^2} \|\nabla u\|_{L^2} + \|\nabla u\|_{L^2}^2\right) \left(\|\nabla \nabla \rho\|_{L^2}^2 + \|\nabla \nabla \sigma\|_{L^2}^2\right). \label{DtH3est1}
\end{align}
For the term $I_{2, 2}$, we use H\"{o}lder's inequality and Young's inequality to get
\begin{align}
\begin{split}
I_{2, 2} & \leq D \|\nabla \Phi\|_{L^\infty} \left(\|\nabla \Delta \rho\|_{L^2} \|\nabla \nabla \Delta \sigma\|_{L^2} + \|\nabla \Delta \sigma\|_{L^2} \|\nabla \nabla \Delta \rho\|_{L^2}\right)\\
& \leq C \|\nabla \Phi\|_{L^\infty}^2 \left(\|\nabla \Delta \rho\|_{L^2}^2 + \|\nabla \Delta \sigma\|_{L^2}^2\right) + \frac{D}{5} \left(\|\Delta^2 \rho\|_{L^2}^2 + \|\Delta^2 \sigma\|_{L^2}^2\right).
\end{split}
\end{align}
By H\"{o}lder's inequality for $L^2$-$L^4$-$L^4$ and the Gagliardo-Nirenberg inequalities
\begin{align*}
\|\Delta f\|_{L^4} & \leq C \|\nabla \Delta f\|_{L^2}^{\frac{3}{4}} \|\nabla f\|_{L^2}^{\frac{1}{4}} + C \|\nabla f\|_{L^2},\\
\|\nabla f\|_{L^4} & \leq C \|\nabla \Delta f\|_{L^2}^{\frac{1}{4}} \|\nabla f\|_{L^2}^{\frac{3}{4}} + C \|\nabla f\|_{L^2},
\end{align*}
we obtain
\begin{align}
\begin{split}
I_{2, 3} & \leq C \|\nabla \Delta \rho\|_{L^2} \|\Delta \rho\|_{L^4} \|\nabla \sigma\|_{L^4} + C \|\nabla \Delta \sigma\|_{L^2} \|\Delta \rho\|_{L^4} \|\nabla \rho\|_{L^4} + C \|\nabla \Delta \rho\|_{L^2} \|\Delta \sigma\|_{L^4} \|\nabla \rho\|_{L^4}\\
& \leq C \|\nabla \Delta \rho\|_{L^2} (\|\nabla \Delta \rho\|_{L^2}^{3/4} \|\nabla \rho\|_{L^2}^{1/4} + \|\nabla \rho\|_{L^2}) (\|\nabla \Delta \sigma\|_{L^2}^{1/4} \|\nabla \sigma\|_{L^2}^{3/4} + \|\nabla \sigma\|_{L^2})\\
& \quad + C \|\nabla \Delta \sigma\|_{L^2} (\|\nabla \Delta \rho\|_{L^2}^{3/4} \|\nabla \rho\|_{L^2}^{1/4} + \|\nabla \rho\|_{L^2}) (\|\nabla \Delta \rho\|_{L^2}^{1/4} \|\nabla \rho\|_{L^2}^{3/4} + \|\nabla \rho\|_{L^2})\\
& \quad + C \|\nabla \Delta \rho\|_{L^2} (\|\nabla \Delta \sigma\|_{L^2}^{3/4} \|\nabla \sigma\|_{L^2}^{1/4} + \|\nabla \sigma\|_{L^2}) (\|\nabla \Delta \rho\|_{L^2}^{1/4} \|\nabla \rho\|_{L^2}^{3/4} + \|\nabla \rho\|_{L^2}),
\end{split}
\end{align}
which implies
\begin{align}
\label{DtH3est3}
\begin{split}
I_{2, 3}
& \leq C \left(\|\nabla \rho\|_{L^2} + \|\nabla \sigma\|_{L^2}\right) \left(\|\nabla \Delta \rho\|_{L^2}^2 + \|\nabla \Delta \sigma\|_{L^2}^2\right) \\
&\qquad\qquad+ C \left(\|\nabla \rho\|_{L^2}^2 + \|\nabla \sigma\|_{L^2}^2\right) \left(\|\nabla \Delta \rho\|_{L^2} + \|\nabla \Delta \sigma\|_{L^2}\right).
\end{split}
\end{align}
The estimates for $I_{2, 4}$ and $I_{2, 5}$ follow from H\"{o}lder's inequality and the Sobolev embedding $H^1(\T^2) \hookrightarrow L^4(\T^2)$,
\begin{align}
\label{DtH3est4}
\begin{split}
I_{2, 4} & \leq C \|\rho\|_{L^4} \left(\|\nabla \rho\|_{L^4} \|\nabla \Delta \sigma\|_{L^2} + \|\nabla \sigma\|_{L^4} \|\nabla \Delta \rho\|_{L^2}\right)\\
& \leq C \|\rho\|_{L^4} \left(\|\rho\|_{H^2} + \|\sigma - \bar{\sigma}\|_{H^2}\right) \left(\|\nabla \Delta \rho\|_{L^2} + \|\nabla \Delta \sigma\|_{L^2}\right),
\end{split}
\end{align}
and
\begin{align}
\label{DtH3est5}
\begin{split}
I_{2, 5} & \leq C \|\rho\|_{L^\infty} \|\nabla \Delta \rho\|_{L^2} \|\nabla \Delta \sigma\|_{L^2} \leq C \|\rho\|_{H^2} \|\nabla \Delta \rho\|_{L^2} \|\nabla \Delta \sigma\|_{L^2}.
\end{split}
\end{align}
Finally, we use H\"{o}lder's inequality and the interpolation inequality \eqref{L4interp} to obtain
\begin{align}
\label{DtH3est6}
\begin{split}
I_{2, 6} & \leq C \|\nabla \rho\|_{L^4}^2 \|\Delta \sigma\|_{L^2} + C \|\nabla \sigma\|_{L^4} \|\nabla \rho\|_{L^4} \|\Delta \rho\|_{L^2}\\
& \leq C \left(\|\nabla \rho\|_{L^2}^2 + \|\nabla \sigma\|_{L^2}^2\right) \left(\|\Delta \rho\|_{L^2} + \|\Delta \sigma\|_{L^2}\right) + C \left(\|\nabla \rho\|_{L^2} + \|\nabla \sigma\|_{L^2}\right) \left(\|\Delta \rho\|_{L^2}^2 + \|\Delta \sigma\|_{L^2}^2\right).
\end{split}
\end{align}
Gathering the estimates \eqref{DtH3est1}--\eqref{DtH3est6} into \eqref{DtH3}, we arrive at
\begin{align}
\label{DtH3estlast}
\begin{split}
& \frac{1}{2} \frac{\diff}{\diff{t}} \left(\|\nabla \Delta \rho\|_{L^2}^2 + \|\nabla \Delta \sigma\|_{L^2}^2\right) + \frac{D}{5} \left(\|\Delta^2 \rho\|_{L^2}^2 + \|\Delta^2 \sigma\|_{L^2}^2\right) + \frac{D}{\ve} \int \sigma |\nabla \Delta \rho|^2\\
&\quad\leq C \left(\|u\|_{H^2}^2 + \|\nabla \Phi\|_{L^\infty}^2 + \|\rho\|_{H^2} + \|\sigma - \bar{\sigma}\|_{H^2}\right) \left(\|\nabla \Delta \rho\|_{L^2}^2 + \|\nabla \Delta \sigma\|_{L^2}^2\right)\\
&\qquad + C \left(\|\nabla \rho\|_{L^2}^2 + \|\nabla \sigma\|_{L^2}^2 + \|\rho\|_{L^4} \|\rho\|_{H^2} + \|\rho\|_{L^4} \|\sigma - \bar{\sigma}\|_{H^2}\right) \left(\|\nabla \Delta \rho\|_{L^2} + \|\nabla \Delta \sigma\|_{L^2}\right)\\
&\qquad + C \left(\|\Delta u\|_{L^2} \|\nabla u\|_{L^2} + \|\omega\|_{L^2}^2 + \|\nabla \rho\|_{L^2} + \|\nabla \sigma\|_{L^2}\right) \left(\|\Delta \rho\|_{L^2}^2 + \|\Delta \sigma\|_{L^2}^2\right)\\
&\qquad + C \left(\|\nabla \rho\|_{L^2}^2 + \|\nabla \sigma\|_{L^2}^2\right) \left(\|\Delta \rho\|_{L^2} + \|\Delta \sigma\|_{L^2}\right).
\end{split}
\end{align}
We drop the dissipation terms, integrate in time, and use the bounds \eqref{l2time}, \eqref{Phidecay}, \eqref{rholpdecay}, \eqref{omegal2}, \eqref{deltal2bdd}, \eqref{L2H3bdd}, \eqref{omegabdd}, and \eqref{uH3bdd} to obtain
\[
\|\nabla \Delta \rho(t)\|_{L^2}^2 + \|\nabla \Delta \sigma(t)\|_{L^2}^2 \leq C + e^{C e^{C t}}.
\]
Going back to \eqref{DtH3estlast}, we conclude the second inequality in \eqref{Hsbdd1}.

\vspace{1ex}\noindent {\bf Uniqueness.}
Let $(\rho_1, \sigma_1, u_1)$, and $(\rho_2, \sigma_2, u_2)$ be two solutions of the initial value problem \eqref{rhot}--\eqref{initial} satisfying $\rho_1, \rho_2, \sigma_1, \sigma_2 \in C([0, T]; H^s(\T^2))$ and $u_1, u_2 \in C([0, T]; H^s(\T^2))$ for $s > 2$. We denote their differences by $(\tilde{\rho}, \tilde{\sigma}, \tilde{u}) = (\rho_1 - \rho_2, \sigma_1 - \sigma_2, u_1 - u_2)$. Then the equations for $(\tilde{\rho}, \tilde{\sigma}, \tilde{u})$ are
\begin{align*}
& \pt \tilde{\rho} = - u_1 \cdot \nabla \tilde{\rho} - \tilde{u} \cdot \nabla \rho_2 + D \Delta \tilde{\rho} + D (\nabla \sigma_1 \cdot \nabla \tilde{\Phi} + \nabla \tilde{\sigma} \cdot \nabla \Phi_2) + D (\sigma_1 \Delta \tilde{\Phi} + \tilde{\sigma} \Delta \Phi_2),\\
& \pt \tilde{\sigma} = - u_1 \cdot \nabla \tilde{\sigma} - \tilde{u} \cdot \nabla \sigma_2 + D \Delta \tilde{\sigma} + D (\nabla \rho_1 \cdot \nabla \tilde{\Phi} + \nabla \tilde{\rho} \cdot \nabla \Phi_2) + D (\rho_1 \Delta \tilde{\Phi} + \tilde{\rho} \Delta \Phi_2),\\
& - \ve \Delta \tilde{\Phi} = \tilde{\rho},\\
& \pt \tilde{u} + u_1 \cdot \nabla \tilde{u} + \tilde{u} \cdot \nabla u_2 + \nabla (p_1 - p_2) = - (k_B T_K) (\rho_1 \nabla \tilde{\Phi} + \tilde{\rho} \nabla \Phi_2),\\
& \nabla \cdot \tilde{u} = 0,
\end{align*}
with initial data
\[
\tilde{\rho}(\cdot, 0) = \tilde{\sigma}(\cdot, 0) = 0, \qquad \tilde{u}(\cdot, 0) = 0.
\]
The $L^2$-estimates for $(\tilde{\rho}, \tilde{\sigma}, \tilde{u})$ lead to
\begin{align}
\label{enuniq}
\begin{split}
& \frac{1}{2} \frac{\diff}{\diff{t}} \left(\|\tilde{\rho}\|_{L^2}^2 + \|\tilde{\sigma}\|_{L^2}^2 + \|\tilde{u}\|_{L^2}^2\right) + D \left(\|\nabla \tilde{\rho}\|_{L^2}^2 + \|\nabla \tilde{\sigma}\|_{L^2}^2\right) + \frac{D}{\ve} \int \tilde{\rho}^2 \sigma_1\\
&\quad= - \int \tilde{\rho} \tilde{u} \cdot \nabla \rho_2 - \int \tilde{\sigma} \tilde{u} \cdot \nabla \sigma_2 + D \int \tilde{\rho} \nabla \sigma_1 \cdot \nabla \tilde{\Phi} + D \int \tilde{\sigma} \nabla \rho_1 \cdot \nabla \tilde{\Phi} - \frac{D}{\ve} \int \tilde{\rho} \tilde{\sigma} (\rho_1 + \rho_2)\\
& \quad\quad + \int u_2 \cdot (\tilde{u} \cdot \tilde{u}) - (k_B T_K) \int \rho_1 \tilde{u} \cdot \nabla \tilde{\Phi} - (k_B T_K) \int \tilde{\rho} \tilde{u} \cdot \nabla \Phi_2 \\
&\quad\leq C \left(\|\nabla \rho_2\|_{L^\infty} + \|\nabla \sigma_2\|_{L^\infty} + \|\rho_1\|_{L^\infty} + \|\rho_2\|_{L^\infty} + \|\nabla u_2\|_{L^\infty}\right) \left(\|\tilde{\rho}\|_{L^2}^2 + \|\tilde{\sigma}\|_{L^2}^2 + \|\tilde{u}\|_{L^2}^2\right).
\end{split}
\end{align}
Because $\sigma_1 \geq 0$, we can drop the last two terms in the first line of \eqref{enuniq}. Finally, in view of Sobolev embeddings and Gr\"{o}nwall's inequality, we obtain $(\tilde{\rho}, \tilde{\sigma}, \tilde{u}) = (0, 0, 0)$, which proves the uniqueness.
\end{proof}

\section{Vanishing viscosity limits}
\label{sec:vanish}

In this section, we consider the approximation of inviscid flows by high Reynolds number viscous flows.

We denote by $\rho$, $\sigma$, and $u$ the solution to the initial value problem for the Nernst-Planck-Euler system \eqref{rhot}--\eqref{initial}. Let $\rho^\nu$, $\sigma^\nu$, and $u^\nu$ be the solutions to the Nernst-Planck-Navier--Stokes system
\begin{align}
\label{npns}
\begin{split}
& \pt \rho^\nu = - u^\nu \cdot \nabla \rho^\nu + D \left(\Delta \rho^\nu + \nabla \sigma^\nu \cdot \nabla \Phi^\nu + \sigma^\nu \Delta \Phi^\nu\right),\\
& \pt \sigma^\nu = - u^\nu \cdot \nabla \sigma^\nu + D \left(\Delta \sigma^\nu + \nabla \rho^\nu \cdot \nabla \Phi^\nu + \rho^\nu \Delta \Phi^\nu\right),\\
& - \ve \Delta \Phi^\nu = \rho^\nu,\\
& \pt u^\nu + u^\nu \cdot \nabla u^\nu + \nabla p^\nu = \nu \Delta u^\nu - (k_B T_K) \rho^\nu \nabla \Phi^\nu,\\
& \nabla \cdot u^\nu = 0,
\end{split}
\end{align}
where $0 < \nu \ll 1$ is the kinematic viscosity of the fluid, with initial data
\begin{align}
\label{initialnu}
\begin{split}
& \rho^\nu(\cdot, 0) = \rho^\nu(0) = c_1^\nu(0) - c_2^\nu(0),\\
& \sigma^\nu(\cdot, 0) = \sigma^\nu(0) = c_1^\nu(0) + c_2^\nu(0),\\
& u^\nu(\cdot, 0) = u_0^\nu,
\end{split}
\end{align}
which also satisfies the property
\begin{align}
\label{neutnu}
\int \rho^\nu(x, 0) \diff{x} = \int c_1^\nu(x, 0) - c_2^\nu(x, 0) \diff{x} = 0.
\end{align}

To study the inviscid limit, we assume the asymptotically matching initial condition
\begin{align}
\label{initconv}
\lim_{\nu \to 0^+} \left(\|c_1^\nu(0) - c_1(0)\|_{H^3} + \|c_2^\nu(0) - c_2(0)\|_{H^3} + \|u_0^\nu - u_0\|_{H^3}\right) = 0.
\end{align}

Notice that the nonnegativity of the ionic concentrations $c_1^\nu$ and $c_2^\nu$ is also preserved in time and that the energy estimates of the system \eqref{npns}--\eqref{initialnu} only incur an extra $\nu$-dependent dissipation term. Hence, we are able to show, by the same arguments as in Sections~\ref{sec:weak} and \ref{sec:strong}, that the solutions to \eqref{npns}--\eqref{initialnu} exist for all time and are unique. Moreover, for any $t > 0$, they satisfy the bounds (as in Theorem~\ref{thm:weakglobal} and Theorem~\ref{thm:strongglobal})
\begin{align}
\label{nusolest}
\begin{split}
& \|\rho^\nu(t)\|_{L^p} + \|\sigma^\nu(t) - \bar{\sigma}^\nu\|_{L^p} + \|\nabla \Phi^\nu(t)\|_{L^\infty} \leq C e^{- C' t} \quad \text{for all}\ 1 \leq p < \infty,\\
& \|\nabla \rho^\nu(t)\|_{L^2}^2 + \|\nabla \sigma^\nu(t)\|_{L^2}^2 + \int_0^t \|\Delta \rho^\nu(\tau)\|_{L^2}^2 + \|\Delta \sigma^\nu(\tau)\|_{L^2}^2 \diff{\tau} \leq C,\\
& \|\Delta \rho^\nu(t)\|_{L^2} + \|\Delta \sigma^\nu(t)\|_{L^2} + \int_0^t \|\nabla \Delta \rho^\nu(\tau)\|_{L^2}^2 + \|\nabla \Delta \sigma^\nu(\tau)\|_{L^2}^2 \diff{\tau} \leq C,\\
& \|\nabla \Delta \rho^\nu(t)\|_{L^2} + \|\nabla \Delta \sigma^\nu(t)\|_{L^2} + \int_0^t \|\Delta^2 \rho^\nu(\tau)\|_{L^2}^2 + \|\Delta^2 \sigma^\nu(\tau)\|_{L^2}^2 \diff{\tau} \leq C e^{C e^t},\\
& \|u^\nu(t)\|_{L^2} \leq C,\\
& \|u^\nu(t)\|_{H^3} \leq e^{C e^{C t}},
\end{split}
\end{align}
for some constants $C, C' > 0$ depending only on $D$, $\ve$, and the initial data. We note that when $\nu$ is sufficiently close to $0$, the $H^3$-norms of the initial data are uniformly bounded, so we are allowed to assume the constants $C$ and $C'$ to be independent of $\nu$.

The main result of this section is given in Theorem \ref{thm:vanish} below, which can be improved (by a similar argument) to hold for any strong solutions in $H^s(\T^2)$ with $s > 2$. Namely, the $L^\infty(0, T; H^{s- 2}(\T^2))$-norms of the solutions satisfy \eqref{H1conv}, the $L^\infty(0, T; H^{s- 1}(\T^2))$-norms of the solutions satisfy \eqref{H2conv}, and the $L^\infty(0, T; H^s(\T^2))$-norms of the solutions satisfy \eqref{H3conv}.

\begin{theorem}
\label{thm:vanish}
Let $\ve > 0$ and $D > 0$. Let $c_1(0), c_2(0), c_1^\nu(0), c_2^\nu(0) \in H^3(\T^2)$ be nonnegative and satisfy \eqref{neut} and \eqref{neutnu}. Let $u_0, u_0^\nu \in H^3(\T^2)$ be divergence free. Suppose that $(c_1(0), c_2(0), u_0)$ and $(c_1^\nu(0), c_2^\nu(0), u_0^\nu)$ satisfy the asymptotically matching condition \eqref{initconv}. Let $(\rho^\nu, \sigma^\nu, u^\nu)$ denote the unique solution to the initial value problem for the Nernst-Planck-Navier-Stokes system \eqref{npns}--\eqref{initialnu}, and $(\rho, \sigma, u)$ denote the unique solution to the initial value problem for the Nernst-Planck-Euler system \eqref{rhot}--\eqref{initial}. Then for any $T > 0$ and $t \in [0, T]$, we have the pointwise convergence
\begin{align}
\begin{split}
& \|\rho^\nu(t) - \rho(t)\|_{H^1(\T^2)} + \|\sigma^\nu(t) - \sigma(t)\|_{H^1(\T^2)} + \|u^\nu(t) - u(t)\|_{H^1(\T^2)}\\
&\quad\leq C \left(\nu t + \|c_1^\nu(0) - c_1(0)\|_{H^1(\T^2)} + \|c_2^\nu(0) - c_2(0)\|_{H^1(\T^2)} + \|u_0^\nu - u_0\|_{H^1(\T^2)}\right),
\end{split}\label{H1conv}
\end{align}
and
\begin{align}
\begin{split}
& \|\rho^\nu(t) - \rho(t)\|_{H^2(\T^2)} + \|\sigma^\nu(t) - \sigma(t)\|_{H^2(\T^2)} + \|u^\nu(t) - u(t)\|_{H^2(\T^2)}\\
&\quad \leq C \left((\nu t)^{\frac{1}{2}} + \|c_1^\nu(0) - c_1(0)\|_{H^2(\T^2)} + \|c_2^\nu(0) - c_2(0)\|_{H^2(\T^2)} + \|u_0^\nu - u_0\|_{H^2(\T^2)}\right),
\end{split}\label{H2conv}
\end{align}
where $C$ depends only on $\ve$, $D$, $\rho$, $\sigma$, $u$, and $T$. At the $H^3$-level, we have
\begin{align}
\label{H3conv}
\lim_{\nu \to 0^+} \left(\|\rho^\nu - \rho\|_{L^\infty(0, T; H^3(\T^2))} + \|\sigma^\nu - \sigma\|_{L^\infty(0, T; H^3(\T^2))} + \|u^\nu - u\|_{L^\infty(0, T; H^3(\T^2))}\right) = 0.
\end{align}
\end{theorem}
\begin{proof}
Let $(\vr, \vs, v) = (\rho^\nu - \rho, \sigma^\nu - \sigma, u^\nu - u)$ denote the difference between the solutions to \eqref{npns}--\eqref{initialnu} and the solutions to \eqref{rhot}--\eqref{initial}. Then the system obeyed by $(\vr, \vs, v)$ is 
\begin{align}
& \pt \vr = - v \cdot \nabla \vr - u \cdot \nabla \vr - v \cdot \nabla \rho + D (\Delta \vr + \nabla \vs \cdot \nabla \Psi + \nabla \sigma \cdot \nabla \Psi + \nabla \vs \cdot \nabla \Phi + \vs \Delta \Psi + \sigma \Delta \Psi + \vs \Delta \Phi), \label{vrt}\\
& \pt \vs = - v \cdot \nabla \vs - u \cdot \nabla \vs - v \cdot \nabla \sigma + D (\Delta \vs + \nabla \vr \cdot \nabla \Psi + \nabla \rho \cdot \nabla \Psi + \nabla \vr \cdot \nabla \Phi + \vr \Delta \Psi + \rho \Delta \Psi + \vr \Delta \Phi), \label{vst}\\
& - \ve \Delta \Psi = \varrho, \label{Psieqn}\\
& \pt v + v \cdot \nabla v + u \cdot \nabla v + v \cdot \nabla u + \nabla (p^\nu - p) = \nu \Delta v + \nu \Delta u - (k_B T_K) (\vr \nabla \Psi + \rho \nabla \Psi + \vr \nabla \Phi), \label{vt}\\
& \nabla \cdot v = 0. \label{vincomp}
\end{align}

\vspace{1ex}\noindent {\bf Inviscid limit for $H^1$-norms.}
We multiply \eqref{vrt} and \eqref{vst} by $\vr$ and $\vs$ respectively, integrate over $\T^2$, integrate by parts, and use \eqref{poisson}, \eqref{incomp}, \eqref{Psieqn}, and \eqref{vincomp} to obtain
\begin{align*}
& \frac{1}{2} \frac{\diff}{\diff{t}} \left(\|\vr\|_{L^2}^2 + \|\vs\|_{L^2}^2\right) + D \left(\|\nabla \vr\|_{L^2}^2 + \|\nabla \vs\|_{L^2}^2\right) + \frac{D}{\ve} \int (\sigma + \vs) \vr^2\\
&\quad= \int \rho v \cdot \nabla \vr + \int (\sigma - \bar{\sigma}) v \cdot \nabla \vs + D \int \vr \nabla \sigma \cdot \nabla \Psi - D \int \vs \nabla \rho \cdot \nabla \Psi - \frac{2 D}{\ve} \int \vs \rho \vr\\
&\quad\leq C \left(\|\rho\|_{H^2}^2 + \|\sigma - \bar{\sigma}\|_{H^2}^2\right) \|v\|_{L^2}^2 + \frac{D}{2} \left(\|\nabla \vr\|_{L^2}^2 + \|\nabla \vr\|_{L^2}^2\right) + C \|\nabla \Psi\|_{L^\infty} \|\nabla \sigma\|_{L^2} \|\vr\|_{L^2}\\
& \qquad\qquad + C \|\nabla \Psi\|_{L^\infty} \|\nabla \rho\|_{L^2} \|\vs\|_{L^2} + C \|\rho\|_{L^\infty} \|\vr\|_{L^2} \|\vs\|_{L^2},
\end{align*}
where we used H\"{o}lder's inequality and Young's inequality. Absorbing the term $\frac{D}{2} \left(\|\nabla \vr\|_{L^2}^2 + \|\nabla \vs\|_{L^2}^2\right)$ into the dissipation term in the first line and using the bounds \eqref{solest1}--\eqref{solest2} as well as a consequence of the elliptic estimate and Sobolev embeddings
\begin{align}
\label{Psiemb}
\|\nabla \Psi\|_{L^\infty} \leq C \|\vr\|_{H^1},
\end{align}
we get
\begin{align}
\label{vrvsL2}
\begin{split}
& \frac{\diff}{\diff{t}} \left(\|\vr\|_{L^2}^2 + \|\vs\|_{L^2}^2\right) + D \left(\|\nabla \vr\|_{L^2}^2 + \|\nabla \vs\|_{L^2}^2\right) + \frac{2 D}{\ve} \int (\sigma + \vs) \vr^2\\
&\quad\leq C \|v\|_{L^2}^2 + C \|\nabla \Psi\|_{L^\infty} \|\vr\|_{L^2} + C \|\nabla \Psi\|_{L^\infty} \|\vs\|_{L^2} + C \|\vr\|_{L^2} \|\vs\|_{L^2}\\
&\quad\leq C \|v\|_{L^2}^2 + C \left(\|\vr\|_{L^2} + \|\vs\|_{L^2}\right) \|\vr\|_{H^1}.
\end{split}
\end{align}

We take the scalar product of \eqref{vt} with $v$ and integrate by parts,
\[
\frac{1}{2} \frac{\diff}{\diff{t}} \|v\|_{L^2}^2 + \nu \|\nabla v\|_{L^2}^2\\
= \nu \int v \cdot \Delta u - \int v \cdot (\nabla u v) - (k_B T_K) \int v \cdot (\vr \nabla \Psi + \rho \nabla \Psi + \vr \nabla \Phi).
\]
Using H\"{o}lder's inequality and the bounds \eqref{solest1}--\eqref{solest2}, \eqref{Hsbdd1}--\eqref{Hsbdd2} together with \eqref{Psiemb}, we have
\begin{align}
\label{vL2}
\begin{split}
 \frac{1}{2} \frac{\diff}{\diff{t}} \|v\|_{L^2}^2 + \nu \|\nabla v\|_{L^2}^2
&\leq \nu \|\Delta u\|_{L^2} \|v\|_{L^2} + \|\nabla u\|_{L^\infty} \|v\|_{L^2}^2\\
&\qquad+ C \left(\|\nabla \Psi\|_{L^\infty} \|\vr\|_{L^2} + \|\nabla \Psi\|_{L^\infty} \|\rho\|_{L^2} + \|\nabla \Phi\|_{L^\infty} \|\vr\|_{L^2}\right) \|v\|_{L^2}\\
&\leq C \nu \|v\|_{L^2} + C \left(\|v\|_{L^2} + \|\vr\|_{H^1}\right) \|v\|_{L^2}.
\end{split}
\end{align}
At $H^1$-level, from \eqref{vrt}--\eqref{Psieqn} and \eqref{poisson}, we integrate by parts and use H\"{o}lder's inequality to obtain
\begin{align*}
& \frac{1}{2} \frac{\diff}{\diff{t}} \left(\|\nabla \vr\|_{L^2}^2 + \|\nabla \vs\|_{L^2}^2\right) + D \left(\|\Delta \vr\|_{L^2}^2 + \|\Delta \vs\|_{L^2}^2\right) + \frac{D}{\ve} \int (\vs + \sigma) |\nabla \vr|^2\\
&\quad\leq \|\nabla v\|_{L^2} \left(\|\nabla \vr\|_{L^4}^2 + \|\nabla \vs\|_{L^4}^2\right) + \|\nabla u\|_{L^\infty} \left(\|\nabla \vr\|_{L^2}^2 + \|\nabla \vs\|_{L^2}^2\right)\\
& \qquad+ \left(\|\nabla \rho\|_{L^4} \|\nabla v\|_{L^2} + \|\nabla \nabla \rho\|_{L^2} \|v\|_{L^4}\right) \|\nabla \vr\|_{L^4} + \left(\|\nabla \sigma\|_{L^4} \|\nabla v\|_{L^2} + \|\nabla \nabla \sigma\|_{L^2} \|v\|_{L^4}\right) \|\nabla \vs\|_{L^4}\\
&\qquad + D \|\nabla \Psi\|_{L^\infty} \|\nabla \vr\|_{L^2} \|\nabla \nabla \vs\|_{L^2} + D \|\nabla \Psi\|_{L^\infty} \|\nabla \vs\|_{L^2} \|\nabla \nabla \vr\|_{L^2} + 2 D \|\nabla \nabla \Psi\|_{L^2} \|\nabla \vr\|_{L^4} \|\nabla \vs\|_{L^4}\\
&\qquad + D \|\nabla \nabla \sigma\|_{L^2} \|\nabla \Psi\|_{L^\infty} \|\nabla \vr\|_{L^2} + D \|\nabla \nabla \rho\|_{L^2} \|\nabla \Psi\|_{L^\infty} \|\nabla \vs\|_{L^2} + D \|\nabla \sigma\|_{L^4} \|\nabla \nabla \Psi\|_{L^2} \|\nabla \vr\|_{L^4}\\
&\qquad + D \|\nabla \rho\|_{L^4} \|\nabla \nabla \Psi\|_{L^2} \|\nabla \vs\|_{L^4} + D \|\nabla \Phi\|_{L^\infty} \|\nabla \vr\|_{L^2} \|\nabla \nabla \vs\|_{L^2} + D \|\nabla \Phi\|_{L^\infty} \|\nabla \vs\|_{L^2} \|\nabla \nabla \vr\|_{L^2}\\
&\qquad + 2 D \|\nabla \nabla \Phi\|_{L^2} \|\nabla \vr\|_{L^4} \|\nabla \vs\|_{L^4} + \frac{3 D}{\ve} \|\vr\|_{L^2} \|\nabla \vs\|_{L^4} \|\nabla \vr\|_{L^4} + \frac{D}{\ve} \|\nabla \sigma\|_{L^\infty} \|\vr\|_{L^2} \|\nabla \vr\|_{L^2}\\
&\qquad + \frac{2 D}{\ve} \|\nabla \rho\|_{L^\infty} \|\vr\|_{L^2} \|\nabla \vs\|_{L^2} + \frac{D}{\ve} \|\nabla \rho\|_{L^\infty} \|\vs\|_{L^2} \|\nabla \vr\|_{L^2} + \frac{3 D}{\ve} \|\rho\|_{L^\infty} \|\nabla \vs\|_{L^2} \|\nabla \vr\|_{L^2}.
\end{align*}
Then by elliptic estimates, the Sobolev embedding $H^1(\T^2) \hookrightarrow L^4(\T^2)$, the Gagliardo-Nirenberg interpolation inequality \eqref{L4interp}, Young's inequality, and the bounds \eqref{solest1}--\eqref{solest2}, \eqref{Hsbdd1}--\eqref{Hsbdd2} as well as \eqref{Psiemb} we have
\begin{align*}
& \frac{1}{2} \frac{\diff}{\diff{t}} \left(\|\nabla \vr\|_{L^2}^2 + \|\nabla \vs\|_{L^2}^2\right) + D \left(\|\Delta \vr\|_{L^2}^2 + \|\Delta \vs\|_{L^2}^2\right) + \frac{D}{\ve} \int (\vs + \sigma) |\nabla \vr|^2\\
&\quad\leq \frac{D}{2} \left(\|\Delta \vr\|_{L^2}^2 + \|\Delta \vs\|_{L^2}^2\right) + C \|\nabla v\|_{L^2}^2 \left(\|\nabla \vr\|_{H^1} + \|\nabla \vs\|_{H^1}\right) \left(\|\nabla \vr\|_{L^2} + \|\nabla \vs\|_{L^2}\right)\\
&\qquad + C \|v\|_{L^2}^{\frac{2}{3}} \|\nabla v\|_{L^2}^{\frac{2}{3}} \left(\|\nabla \vr\|_{L^2}^{\frac{2}{3}} + \|\nabla \vs\|_{L^2}^{\frac{2}{3}}\right) + C \left(\|\vr\|_{H^1}^2 + \|\vs\|_{H^1}^2\right) + C \|\vr\|_{H^1}^4 + C \|\vr\|_{H^1}^2 \|\vs\|_{H^1}^2\\
&\qquad + C \|\vr\|_{H^1}^3 \|\vs\|_{H^1} + C \|\vr\|_{H^1}^3 + C \|\vr\|_{H^1}^2 \|\vs\|_{H^1} + C \|\vr\|_{H^1} \|\vs\|_{H^1},
\end{align*}
so that
\begin{align}
\label{DvrvsL2}
\begin{split}
& \frac{\diff}{\diff{t}} \left(\|\nabla \vr\|_{L^2}^2 + \|\nabla \vs\|_{L^2}^2\right) + D \left(\|\Delta \vr\|_{L^2}^2 + \|\Delta \vs\|_{L^2}^2\right) + \frac{2 D}{\ve} \int (\vs + \sigma) |\nabla \vr|^2\\
&\quad\leq C \|v\|_{H^1}^2 \left(\|\nabla \vr\|_{L^2}^2 + \|\nabla \vs\|_{L^2}^2 + \|\nabla \vr\|_{L^2} + \|\nabla \vs\|_{L^2}\right) + C \left(\|\vr\|_{H^1}^4 + \|\vs\|_{H^1}^4 + \|\vr\|_{H^1}^2 + \|\vs\|_{H^1}^2\right).
\end{split}
\end{align}
For the $H^1$-norm of $v$, we use the standard estimates for the cubic velocity terms (see e.g. \cite{MB02}), and apply the H\"{o}lder inequality with $L^2$-$L^2$-$L^\infty$, the bounds \eqref{Psiemb}, and the Sobolev calculus inequality \eqref{fgHm} to the electric volume force term to obtain
\begin{align}
\label{vH1}
\begin{split}
& \frac{1}{2} \frac{\diff}{\diff{t}} \|\nabla v\|_{L^2}^2 + \nu \|\Delta v\|_{L^2}^2\\
&\quad\leq C \left(\|u\|_{H^3} + \|v\|_{H^3}\right) \|v\|_{H^1}^2 + \nu \|\Delta u\|_{H^1} \|v\|_{H^1} + C \left(\|\rho\|_{H^1} \|\nabla \Psi\|_{L^\infty} + \|\rho\|_{L^\infty} \|\nabla \Psi\|_{H^1}\right) \|v\|_{H^1}\\
& \qquad + C \left(\|\nabla \Psi\|_{L^\infty} + \|\nabla \Phi\|_{L^\infty} + \|\vr\|_{H^1} + \|\rho\|_{H^1}\right) \|\vr\|_{H^1} \|v\|_{H^1}\\
&\quad \leq C \left(\nu + \|v\|_{H^1} + \|v\|_{H^3} \|v\|_{H^1} + \|\vr\|_{H^1} + \|\vr\|_{H^1}^2\right) \|v\|_{H^1}.
\end{split}
\end{align}
Summing the equations \eqref{vrvsL2}, \eqref{vL2} and \eqref{DvrvsL2}--\eqref{vH1}, dropping the dissipation terms on the left hand side (where we use the fact that $\sigma + \vs = \sigma^\nu \geq 0$), and dividing by $\|\vr\|_{H^1} + \|\vs\|_{H^1} + \|v\|_{H^1}$, we get
\begin{align}
\label{allH1}
\begin{split}
& \frac{\diff}{\diff{t}} \left(\|\vr\|_{H^1} + \|\vs\|_{H^1} + \|v\|_{H^1}\right)\\
&\quad\leq C \Big(\nu + \|v\|_{H^1} + \|v\|_{H^3} \|v\|_{H^1} + \|\vr\|_{H^1}^2 \|v\|_{H^1} + \|\vs\|_{H^1}^2 \|v\|_{H^1} + \|\vr\|_{H^1} \|v\|_{H^1}\\
& \qquad\qquad + \|\vs\|_{H^1} \|v\|_{H^1} + \|\vr\|_{H^1} + \|\vs\|_{H^1} + \|\vr\|_{H^1}^3 + \|\vs\|_{H^1}^3\Big).
\end{split}
\end{align}
Denote
\[
V_1 = \|\vr\|_{H^1} + \|\vs\|_{H^1} + \|v\|_{H^1}.
\]
Using the bounds \eqref{solest1}, \eqref{solest2}, \eqref{Hsbdd1}, \eqref{Hsbdd2}, and \eqref{nusolest} in \eqref{allH1}, we obtain
\[
\frac{\diff}{\diff{t}} V_1 \leq C \left(\nu + V_1 + V_1^2\right).
\]
As a consequence, following \cite{Con86} (in particular, multiplying by an integration factor and then using Lemma~1.3 in \cite{Con86}; see also \cite{CF88}, Chapter~11), we obtain \eqref{H1conv}.

\vspace{1ex}\noindent {\bf Inviscid limit for $H^2$-norms.}
We multiply \eqref{vrt} and \eqref{vst} by $\Delta^2 \vr$ and $\Delta^2 \vs$, respectively, and integrate over $\T^2$. We obtain
\begin{align}
\begin{split}
& \frac{\diff}{\diff{t}} \left(\|\Delta \vr\|_{L^2}^2 + \|\Delta \vs\|_{L^2}^2\right) + D \left(\|\nabla \Delta \vr\|_{L^2}^2 + \|\nabla \Delta \vs\|_{L^2}^2\right)\\
&\quad\leq C \left(\|v\|_{H^1}^2 + \|\vr\|_{H^2}^2 + \|\vs\|_{H^2}^2\right) \left(\|\vr\|_{H^2}^2 + \|\vs\|_{H^2}^2\right) + C \|v\|_{H^1}^4 \left(\|\vr\|_{H^2}^4 + \|\vs\|_{H^2}^4\right) + C \left(\|\Delta \vr\|_{L^2}^2 + \|\Delta \vs\|_{L^2}^2\right)\\
&\qquad + C \left(\|u\|_{H^2}^2 + \|\rho\|_{H^2}\right) \left(\|\Delta \vr\|_{L^2}^2 + \|\Delta \vs\|_{L^2}^2\right) + C \left(\|\rho\|_{H^3} + \|\sigma - \bar{\sigma}\|_{H^3}\right) \left(\|\Delta \vr\|_{L^2} + \|\Delta \vs\|_{L^2}\right) \|v\|_{H^2}\end{split} \notag\\
&\qquad + C \left(\|\rho\|_{H^3} + \|\sigma - \bar{\sigma}\|_{H^3}\right) \|\vr\|_{H^2} \left(\|\Delta \vr\|_{L^2} + \|\Delta \vs\|_{L^2}\right) + C \|\rho\|_{H^1}^2 \left(\|\Delta \vr\|_{L^2}^2 + \|\Delta \vs\|_{L^2}^2\right), 
\label{vrvsH2}
\end{align}
using similar estimates to \eqref{DtH2}--\eqref{DtH2est02}. We omit further details.
Similarly, for the energy estimates of $\|\Delta v\|_{L^2}$, we have
\begin{align}
\label{vH2}
\begin{split}
\frac{1}{2} \frac{\diff}{\diff{t}} \|\Delta v\|_{L^2}^2 + \nu \|\nabla \Delta v\|_{L^2}^2 & \leq \nu \|u\|_{H^3} \|v\|_{H^3} + C \|v\|_{H^3} \|v\|_{H^2}^2 + C \|\rho\|_{H^2}^2 \|v\|_{H^2} + C \|u\|_{H^3} \|v\|_{H^3} \|v\|_{H^2}\\
& \qquad + C \|\rho\|_{H^2} \|\vr\|_{H^2} \|v\|_{H^2}.
\end{split}
\end{align}
Summing the equations \eqref{vrvsH2} and \eqref{vH2}, dropping the dissipation terms on the left hand side, dividing by $\|\vr\|_{H^2} + \|\vs\|_{H^2} + \|v\|_{H^2}$, and using the definition of $\vr$, $\vs$, and $v$, together with the bounds \eqref{solest1}--\eqref{solest2}, \eqref{Hsbdd1}--\eqref{Hsbdd2}, and \eqref{nusolest}, we obtain
\[
\frac{\diff}{\diff{t}} V_2 \leq C \left(\nu + V_2 + V_2^2\right),
\]
where we denote
\[
V_2 = \|\vr\|_{H^2}^2 + \|\vs\|_{H^2}^2 + \|v\|_{H^2}^2.
\]
By using an integration factor and invoking Lemma~1.3 in \cite{Con86}, we conclude \eqref{H2conv}.

\vspace{1ex}\noindent {\bf Inviscid limit for $H^3$-norms.}
Now we consider the inviscid limit of the higher norms of the solutions. The idea is to regularize the initial data and then in the estimates to put higher derivatives on the regularized solution \cite{BS75, CW96, Mas07}.

First we regularize the initial data: for any $\kappa > 0$, let
\[
c_i^\kappa(0) = \kappa^{-2} \phi\bigg(\frac{\cdot}{\kappa}\bigg) * c_i(0), \quad i = 1, 2,
\]
where $\phi$ is a radial function satisfying
\[
\phi(|x|) \in C_0^\infty(\T^2), \quad \phi \geq 0, \quad \int \phi = 1.
\]
Notice that the nonnegativity of the function $\phi$ preserves the nonnegativity the ionic concentrations, i.e., $c_1^\kappa(0), c_2^\kappa(0) \geq 0$. 

We further denote
\[
\rho^\kappa(\cdot, 0) = c_1^\kappa(0) -  c_2^\kappa(0), \quad \sigma^\kappa(\cdot, 0) = c_1^\kappa(0) + c_2^\kappa(0), \quad u_0^\kappa = \P_{\lfloor 1 / \kappa \rfloor} u_0,
\]
where $\lfloor \cdot \rfloor$ is the greatest integer function (floor function) and $\P_m$ is the orthogonal projection onto the $m$-dimensional spaces spanned by the eigenfunctions of the Stokes operator on $\T^2$. We note that these definitions result in 
\begin{align}
\label{initk}
\begin{split}
& \|\rho^\kappa(0) - \rho(0)\|_{H^s} \leq C \kappa^{3 - s}, \quad \|\rho^\kappa(0)\|_{H^3} \leq \|\rho(0)\|_{H^3}, \quad \|\rho^\kappa(0)\|_{H^4} \leq \frac{C}{\kappa} ,\\
& \|\sigma^\kappa(0) - \sigma(0)\|_{H^s} \leq C \kappa^{3 - s}, \quad \|\sigma^\kappa(0)\|_{H^3} \leq \|\sigma(0)\|_{H^3}, \quad \|\sigma^\kappa(0)\|_{H^4} \leq \frac{C}{\kappa},\\
& \|u_0^\kappa- u_0\|_{H^s} \leq C \kappa^{3 - s}, \quad \|u_0^\kappa\|_{H^3} \leq \|u_0\|_{H^3}, \quad \|u_0^\kappa\|_{H^4} \leq \frac{C}{\kappa},\\
\end{split}
\end{align}
for any $1 < s < 2$ and some $C > 0$ depending only on $\|c_1(0)\|_{H^3}$, $\|c_2(0)\|_{H^3}$, and $\|u_0\|_{H^3}$.

We denote by $(\rho^\kappa, \sigma^\kappa, u^\kappa)$ the solutions of the initial value problem for the Nernst-Planck-Euler system \eqref{rhot}--\eqref{incomp} with the regularized data $(\rho^\kappa(0), \sigma^\kappa(0), u_0^\kappa)$. We denote $(\vr^\kappa, \vs^\kappa, v^\kappa) = (\rho^\kappa - \rho, \sigma^\kappa - \sigma, u^\kappa - u)$.

Because the solutions $(\rho^\kappa, \sigma^\kappa, u^\kappa)$ satisfy the same equations as $(\rho, \sigma, u)$ but the initial data instead satisfies \eqref{initk}, we see that $(\rho^\kappa, \sigma^\kappa, u^\kappa)$ and, therefore by the triangle inequality, $(\vr^\kappa, \vs^\kappa, v^\kappa)$ satisfy the same bounds as in the conclusions of Theorem~\ref{thm:weakglobal} and Theorem~\ref{thm:strongglobal}, uniformly in $\kappa$. Namely,
\begin{align}
\label{ksolest}
\begin{split}
& \|f(t)\|_{L^p} + \left\|g(t) - \bar{g}\right\|_{L^p} + \|\nabla H(t)\|_{L^\infty} \leq C e^{- C' t} \quad \text{for all}\ 1 \leq p < \infty,\\
& \|\nabla f(t)\|_{L^2} + \|\nabla g(t)\|_{L^2} + \int_0^t \|\Delta f(\tau)\|_{L^2}^2 + \|\Delta g(\tau)\|_{L^2}^2 \diff{\tau} \leq C,\\
& \|\Delta f(t)\|_{L^2} + \|\Delta g(t)\|_{L^2} + \int_0^t \|\nabla \Delta f(\tau)\|_{L^2}^2 + \|\nabla \Delta g(\tau)\|_{L^2}^2 \diff{\tau} \leq C,\\
& \|\nabla \Delta f(t)\|_{L^2} + \|\nabla \Delta g(t)\|_{L^2} + \int_0^t \|\Delta^2 f(\tau)\|_{L^2}^2 + \|\Delta^2 g(\tau)\|_{L^2}^2 \diff{\tau} \leq C e^{C e^t},\\
& \|u^\kappa(t)\|_{L^2} + \|v^\kappa(t)\|_{L^2} \leq C,\\
& \|u^\kappa(t)\|_{H^3} + \|v^\kappa(t)\|_{H^3} \leq e^{C e^{C t}},
\end{split}
\end{align}
for some $C, C' > 0$ and $(f,g) \in \{(\rho^\kappa, \sigma^\kappa), (\vr^\kappa, \vs^\kappa)\}$. 

First we establish a bound on $\|u^\kappa\|_{H^4}$. We differentiate $\|u^\kappa\|_{H^4}^2$ in time  and use the equation for $u^\kappa$ to get
\begin{align*}
\frac{1}{2} \frac{\diff}{\diff{t}} \|\Delta^2 u^\kappa\|_{L^2}^2 = - \int \Delta^2 u^\kappa \cdot \Delta^2 \left(u^\kappa \cdot \nabla u^\kappa\right) - (k_B T_K) \int \Delta^2 u^\kappa \cdot \Delta^2 \left(\rho^\kappa \nabla \Phi^\kappa\right).
\end{align*}
Using integration by parts and the Sobolev calculus inequality \eqref{fgHm}, we obtain
\[
- \int \Delta^2 u^\kappa \cdot \Delta^2 \left(u^\kappa \cdot \nabla u^\kappa\right) \leq C \|\nabla u^\kappa\|_{L^\infty} \|u^\kappa\|_{H^4}^2,
\]
and
\[
- (k_B T_K) \int \Delta^2 u^\kappa \cdot \Delta^2 \left(\rho^\kappa \nabla \Phi^\kappa\right) \leq C \left(\|\rho^\kappa\|_{H^4} \|\nabla \Phi^\kappa\|_{L^\infty} + \|\rho^\kappa\|_{L^\infty} \|\nabla \Phi^\kappa\|_{H^4}\right) \|u^\kappa\|_{H^4}.
\]
Applying the Calder\'{o}n-Zygmund inequality \eqref{BKM} and the elliptic estimates for $\Phi^\kappa$, we arrive at
\begin{align*}
\frac{\diff}{\diff{t}} \|u^\kappa\|_{H^4} \leq C \|\omega^\kappa\|_{L^\infty} \left(1 + \log(1 + \|u^\kappa\|_{H^3})\right) \|u^\kappa\|_{H^4} + C \left(\|\rho^\kappa\|_{H^4} \|\nabla \Phi^\kappa\|_{L^\infty} + \|\rho^\kappa\|_{L^\infty} \|\nabla \Phi^\kappa\|_{H^4}\right).
\end{align*}
Using Gr\"{o}nwall's lemma and the bounds \eqref{ksolest}, we have that for all $0 \leq t \leq T$
\begin{align*}
\|u^\kappa(t)\|_{H^4} & \leq \exp\left(\int_0^t \|\omega^\kappa(\tau)\|_{L^\infty} \left(1 + \log(1 + \|u^\kappa(\tau)\|_{H^3})\right) \diff{\tau}\right)\\
& \qquad \cdot \bigg[\|u_0^\kappa\|_{H^4} + C \int_0^t \|\rho^\kappa(\tau)\|_{H^4} \|\nabla \Phi^\kappa(\tau)\|_{L^\infty} + \|\rho^\kappa(\tau)\|_{L^\infty} \|\nabla \Phi^\kappa(\tau)\|_{H^4} \diff{\tau}\bigg]\\
& \leq C \|u_0^\kappa\|_{H^4} + C \int _0^t \|\rho^\kappa(\tau)\|_{H^4}^2 \diff{\tau}\\
& \leq \frac{C}{\kappa} + C.
\end{align*}
When $\kappa < 1$, we have that for any $0 \leq t \leq T$,
\begin{align}
\label{H4kest}
\|u^\kappa(t)\|_{H^4} \leq \frac{C}{\kappa},
\end{align}
for some $C > 0$ depending on $T$, but not on $\kappa$.

The rest of the proof consists of two steps. The first step is to show the convergence $\|\vr^\kappa\|_{H^3} + \|\vs^\kappa\|_{H^3} + \|v^\kappa\|_{H^3} = \|\rho^\kappa - \rho\|_{H^3} + \|\sigma^\kappa - \sigma\|_{H^3} + \|u^\kappa - u\|_{H^3} \to 0$ as $\kappa \to 0$. The second step is to prove that we can choose $\kappa$ in terms of $\nu$ such that as $\nu \to 0$, both $\kappa \to 0$ and $\|\rho^\nu - \rho^\kappa\|_{H^3} + \|\sigma^\nu - \sigma^\kappa\|_{H^3} + \|u^\nu - u^\kappa\|_{H^3} \to 0$. Then the triangle inequality completes the proof.

\vspace{.5ex}\noindent {\bf Step 1.}
From the systems for $(\rho, \sigma, u)$ and $(\rho^\kappa, \sigma^\kappa, u^\kappa)$, we obtain a system of equations for $(\vr^\kappa, \vs^\kappa, v^\kappa)$
\begin{align}
\begin{split}
& \pt \vr^\kappa - v^\kappa \cdot \nabla \vr^\kappa - D \left(\Delta \vr^\kappa - \nabla \vs^\kappa \cdot \nabla \Psi^\kappa - \vs \Delta \Psi^\kappa\right)\\
& \qquad \qquad = - u^\kappa \cdot \nabla \vr^\kappa - v^\kappa \cdot \nabla \rho^\kappa + D \left(\nabla \sigma^\kappa \cdot \nabla \Psi^\kappa + \nabla \vs^\kappa \cdot \nabla \Phi^\kappa + \sigma^\kappa \Delta \Psi^\kappa + \vs^\kappa \Delta \Phi^\kappa\right),
\end{split} \label{vrkt}\\
\begin{split}
& \pt \vs^\kappa - v^\kappa \cdot \nabla \vs^\kappa - D \left(\Delta \vs^\kappa - \nabla \vr^\kappa \cdot \nabla \Psi^\kappa - \vr \Delta \Psi^\kappa\right)\\
& \qquad \qquad = - u^\kappa \cdot \nabla \vs^\kappa - v^\kappa \cdot \nabla \sigma^\kappa + D \left(\nabla \rho^\kappa \cdot \nabla \Psi^\kappa + \nabla \vr^\kappa \cdot \nabla \Phi^\kappa + \rho^\kappa \Delta \Psi^\kappa + \vr^\kappa \Delta \Phi^\kappa\right),
\end{split} \label{vskt}\\
& - \ve \Delta \Psi^\kappa = \vr^\kappa, \\
\begin{split}
& \pt v^\kappa - v^\kappa \cdot \nabla v^\kappa + \nabla (p^\kappa - p) - (k_B T_K) \vr^\kappa \nabla \Psi^\kappa\\
& \qquad \qquad = - u^\kappa \cdot \nabla v^\kappa - v^\kappa \cdot \nabla u^\kappa - (k_B T_K) \left(\rho^\kappa \nabla \Psi^\kappa + \vr^\kappa \nabla \Phi^\kappa\right),
\end{split} \label{vkt}\\
& \nabla \cdot v^\kappa = 0. 
\end{align}
We test \eqref{vrkt} and \eqref{vskt} with $- \Delta^3 \vr^\kappa$ and $- \Delta^3 \vs^\kappa$ respectively to obtain
\begin{align}
\label{vrkvskH3}
\begin{split}
& \frac{1}{2} \frac{\diff}{\diff{t}} \left(\|\nabla \Delta \vr^\kappa\|_{L^2}^2 + \|\nabla \Delta \vs^\kappa\|_{L^2}^2\right) + D \left(\|\Delta^2 \vr^\kappa\|_{L^2} + \|\Delta^2 \vs^\kappa\|_{L^2}\right)\\
&\quad = J_{1, 1} + J_{1, 2} + J_{1, 3} + J_{1, 4} + J_{1, 5} + J_{1, 6} + J_{1, 7}.
\end{split}
\end{align}
where
\begin{align*}
J_{1, 1} & = - \int \Delta^2 \vr^\kappa \Delta \left(v^\kappa \cdot \nabla \vr^\kappa\right) - \int \Delta^2 \vs^\kappa \Delta \left(v^\kappa \cdot \nabla \vs^\kappa\right),\\
J_{1, 2} & = - D \int \nabla \Delta \vr^\kappa \cdot \nabla \Delta \left(\nabla \vs^\kappa \cdot \nabla \Psi^\kappa + \vs \Delta \Psi^\kappa\right) - D \int \nabla \Delta \vs^\kappa \cdot \nabla \Delta \left(\nabla \vr^\kappa \cdot \nabla \Psi^\kappa + \vr \Delta \Psi^\kappa\right),\\
J_{1, 3} & = - \int \nabla \Delta \vr^\kappa \cdot \nabla \Delta (u^\kappa \cdot \nabla \vr^\kappa) - \int \nabla \Delta \vs^\kappa \cdot \nabla \Delta (u^\kappa \cdot \nabla \vs^\kappa)\\
& \qquad - \int \nabla \Delta \vr^\kappa \cdot \nabla \Delta (v^\kappa \cdot \nabla \rho^\kappa) - \int \nabla \Delta \vs^\kappa \cdot \nabla \Delta (v^\kappa \cdot \nabla \sigma^\kappa),\\
J_{1, 4} & =  - D \int \Delta^2 \vr^\kappa \Delta (\nabla \sigma^\kappa \cdot \nabla \Psi^\kappa + \sigma^\kappa \Delta \Psi^\kappa) - D \int \Delta^2 \vr^\kappa \Delta (\nabla \vs^\kappa \cdot \nabla \Phi + \vs^\kappa \Delta \Phi^\kappa)\\
& \qquad - D \int \Delta^2 \vs^\kappa \Delta (\nabla \rho^\kappa \cdot \nabla \Psi^\kappa + \rho^\kappa \Delta \Psi^\kappa) - D \int \Delta^2 \vs^\kappa \Delta (\nabla \vr^\kappa \cdot \nabla \Phi + \vr^\kappa \Delta \Phi^\kappa).\\
\end{align*}
We observe that the terms $J_{1, 1}$ and $J_{1, 2}$ come from the left hand sides of \eqref{vrkt}--\eqref{vskt}, which are the same as the right hand sides of \eqref{DtH3}--\eqref{DtH3estlast}. So if we repeat the estimates \eqref{DtH3est1}--\eqref{DtH3est6}, we get
\begin{align}
\begin{split}
J_{1, 1} + J_{1, 2} & \leq \frac{4 D}{5} \left(\|\Delta^2 \vr^\kappa\|_{L^2}^2 + \|\Delta^2 \vs^\kappa\|_{L^2}^2\right) - \frac{D}{\ve} \int \vs^\kappa |\nabla \Delta \vr^\kappa|^2\\
& \quad + C \left(\|v^\kappa\|_{H^2}^2 + \|\nabla \Psi^\kappa\|_{L^\infty}^2 + \|\vr^\kappa\|_{H^2} + \|\vs^\kappa - \bar{\vs}^\kappa\|_{H^2}\right) \left(\|\nabla \Delta \vr^\kappa\|_{L^2}^2 + \|\nabla \Delta \vs^\kappa\|_{L^2}^2\right)\\
& \quad + C \left(\|\nabla \vr^\kappa\|_{L^2}^2 + \|\nabla \vs^\kappa\|_{L^2}^2 + \|\vr^\kappa\|_{L^4} \|\vr^\kappa\|_{H^2} + \|\vr^\kappa\|_{L^4} \|\vs^\kappa - \bar{\vs}^\kappa\|_{H^2}\right) \left(\|\nabla \Delta \vr^\kappa\|_{L^2} + \|\nabla \Delta \vs^\kappa\|_{L^2}\right)\\
& \quad + C \left(\|\Delta v^\kappa\|_{L^2} \|\nabla v^\kappa\|_{L^2} + \|\nabla v^\kappa\|_{L^2}^2 + \|\nabla \vr^\kappa\|_{L^2} + \|\nabla \vs^\kappa\|_{L^2}\right) \left(\|\Delta \vr^\kappa\|_{L^2}^2 + \|\Delta \vs^\kappa\|_{L^2}^2\right)
\end{split} \notag\\
& \quad + C \left(\|\nabla \vr^\kappa\|_{L^2}^2 + \|\nabla \vs^\kappa\|_{L^2}^2\right) \left(\|\Delta \vr^\kappa\|_{L^2} + \|\Delta \vs^\kappa\|_{L^2}\right). \label{vrkvskH3-1}
\end{align}
Notice that in \eqref{vrkvskH3-1}, the sign of $\vs^\kappa = \sigma^\kappa - \sigma$ is unknown, so we cannot discard the term $- \frac{D}{\ve} \int \vs^\kappa |\nabla \Delta \vr^\kappa|^2$. Instead, we estimate this term directly
\begin{align}
- \frac{D}{\ve} \int \vs^\kappa |\nabla \Delta \vr^\kappa|^2 \leq \frac{D}{\ve} \|\vs^\kappa\|_{L^\infty} \|\nabla \Delta \vr^\kappa\|_{L^2}^2 \leq C \|\vs^\kappa\|_{H^2} \|\nabla \Delta \vr^\kappa\|_{L^2}^2.
\end{align}
The term $J_{1, 3}$, involving the velocity fields $u^\kappa$ or $v^\kappa$, we rewrite as follows
\begin{align}
\begin{split}
J_{1, 3} & = J_{1, 3, 1} + J_{1, 3, 2} + J_{1, 3, 3} + \Rc,
\end{split}
\end{align}
where 
\begin{align*}
J_{1, 3, 1} & = - \int \nabla \vr^\kappa \cdot (\nabla \Delta u^\kappa \nabla \Delta \vr^\kappa) - \int \nabla \vs^\kappa \cdot (\nabla \Delta u^\kappa \nabla \Delta \vs^\kappa)\\
& \qquad - \int \nabla \rho^\kappa \cdot (\nabla \Delta v^\kappa \nabla \Delta \vr^\kappa) - \int \nabla \sigma^\kappa \cdot (\nabla \Delta v^\kappa \nabla \Delta \vs^\kappa),\\
J_{1, 3, 2} & = - \int \nabla \Delta \vr^\kappa \cdot (\nabla \nabla \Delta \vr^\kappa u^\kappa) - \int \nabla \Delta \vs^\kappa \cdot (\nabla \nabla \Delta \vs^\kappa u^\kappa),\\
J_{1, 3, 3} & = - \int \nabla \Delta \vr^\kappa \cdot (\nabla \nabla \Delta \rho^\kappa v^\kappa) - \int \nabla \Delta \vs^\kappa \cdot (\nabla \nabla \Delta \sigma^\kappa v^\kappa).
\end{align*}
and $\Rc$ represents the lower order terms,
\begin{align}
|\Rc| \leq C \left(\|v^\kappa\|_{H^3}^2 + \|\vr^\kappa\|_{H^3}^2 + \|\vs^\kappa\|_{H^3}^2\right).
\end{align}
To estimate $J_{1, 3, 1}$, we use H\"{o}lder's inequality and \eqref{ksolest},
\begin{align}
\begin{split}
J_{1, 3, 1} & \leq \|\nabla \Delta u^\kappa\|_{L^2} \left(\|\nabla \vr^\kappa\|_{L^\infty} \|\nabla \Delta \vr^\kappa\|_{L^2} + \|\nabla \vs^\kappa\|_{L^\infty} \|\nabla \Delta \vs^\kappa\|_{L^2}\right)\\
& \qquad + \|\nabla \Delta v^\kappa\|_{L^2} \left(\|\nabla \rho^\kappa\|_{L^\infty} \|\nabla \Delta \vr^\kappa\|_{L^2} + \|\nabla \sigma^\kappa\|_{L^\infty} \|\nabla \Delta \vs^\kappa\|_{L^2}\right).
\end{split}
\end{align}
The estimate for $J_{1, 3, 2}$ follows from H\"{o}lder's inequality, Young's inequality, and \eqref{ksolest}
\begin{align}
\begin{split}
J_{1, 3, 2} & \leq C \|\nabla \Delta \vr^\kappa\|_{L^2} \|u^\kappa\|_{L^\infty} \|\Delta^2 \vr^\kappa\|_{L^2} + C \|\nabla \Delta \vs^\kappa\|_{L^2} \|u^\kappa\|_{L^\infty} \|\Delta^2 \vs^\kappa\|_{L^2}\\
& \leq \frac{D}{15} \left(\|\Delta^2 \vr^\kappa\|_{L^2}^2 + \|\Delta^2 \vs^\kappa\|_{L^2}^2\right) + C \left(\|\nabla \Delta \vr^\kappa\|_{L^2}^2  + \|\nabla \Delta \vs^\kappa\|_{L^2}^2\right).
\end{split}
\end{align}
For $J_{1, 3, 3}$, we first integrate by parts, then use H\"{o}lder's inequality and Young's inequality as well as \eqref{ksolest} to get
\begin{align}
\begin{split}
J_{1, 3, 3} & \leq \|\nabla v^\kappa\|_{L^\infty} \left(\|\nabla \Delta \rho^\kappa\|_{L^2} \|\nabla \Delta \vr^\kappa\|_{L^2} + \|\nabla \Delta \sigma^\kappa\|_{L^2} \|\nabla \Delta \vs^\kappa\|_{L^2}\right)\\
& \qquad + \|v^\kappa\|_{L^\infty} \left(\|\nabla \Delta \rho^\kappa\|_{L^2} \|\Delta^2 \vr^\kappa\|_{L^2} + \|\nabla \Delta \sigma^\kappa\|_{L^2} \|\Delta^2 \vs^\kappa\|_{L^2}\right)\\
& \leq C \|\nabla v^\kappa\|_{L^\infty} \left(\|\nabla \Delta \vr^\kappa\|_{L^2} + \|\nabla \Delta \vs^\kappa\|_{L^2}\right) + C \|v^\kappa\|_{L^\infty}^2 + \frac{D}{15} \left(\|\Delta^2 \vr^\kappa\|_{L^2}^2 + \|\Delta^2 \vs^\kappa\|_{L^2}^2\right).
\end{split}
\end{align}
Lastly, for the term $J_{1, 4}$, we use the Sobolev calculus inequality \eqref{fgHm}, Young's inequality, 
the elliptic estimates $\|\Delta \Psi^\kappa\|_{L^2} \leq C \|\vr^\kappa\|_{L^2}$ and the bound $\|\nabla \Psi^\kappa\|_{L^\infty} \leq C \|\vr^\kappa\|_{H^1}$ to obtain
\begin{align}
\label{vrkvskH3-7}
\begin{split}
J_{1, 4} & \leq C \|\Delta^2 \vr^\kappa\|_{L^2} \big(\|\sigma^\kappa\|_{H^3} \|\nabla \Psi^\kappa\|_{L^\infty} + \|\sigma^\kappa\|_{L^\infty} \|\nabla \Psi^\kappa\|_{H^3} + \|\vs^\kappa\|_{H^3} \|\nabla \Phi^\kappa\|_{L^\infty} + \|\vs^\kappa\|_{L^\infty} \|\nabla \Phi^\kappa\|_{H^3}\big)\\
& \qquad + C \|\Delta^2 \vs^\kappa\|_{L^2} \big(\|\rho^\kappa\|_{H^3} \|\nabla \Psi^\kappa\|_{L^\infty} + \|\rho^\kappa\|_{L^\infty} \|\nabla \Psi^\kappa\|_{H^3} + \|\vr^\kappa\|_{H^3} \|\nabla \Phi^\kappa\|_{L^\infty} + \|\vr^\kappa\|_{L^\infty} \|\nabla \Phi^\kappa\|_{H^3}\big)
\end{split} \notag\\
& \leq \frac{D}{15} \left(\|\Delta^2 \vr^\kappa\|_{L^2}^2 + \|\Delta^2 \vs^\kappa\|_{L^2}\right) + C \left(\|\vr^\kappa\|_{H^3}^2 + \|\vs^\kappa\|_{H^3}^2\right),
\end{align}
where in the last line, we also used the bounds \eqref{ksolest}.

Putting the estimates \eqref{vrkvskH3-1}--\eqref{vrkvskH3-7} into \eqref{vrkvskH3}, discarding the dissipation terms, and using the Sobolev embedding $H^3(\T^2) \hookrightarrow L^4(\T^2)$, the estimate $\|\nabla \Psi^\kappa\|_{L^\infty} \leq C \|\vr^\kappa\|_{H^1}$, and the bounds \eqref{ksolest} for $(\vr^\kappa, \vs^\kappa, v^\kappa)$, we have 
\begin{align}
\label{vrvskH3}
\begin{split}
& \frac{\diff}{\diff{t}} \left(\|\nabla \Delta \vr^\kappa\|_{L^2}^2 + \|\nabla \Delta \vs^\kappa\|_{L^2}^2\right)\\
&\quad\leq C \left(\|v^\kappa\|_{H^2}^2 + \|\nabla \Psi^\kappa\|_{L^\infty}^2 + \|\nabla \vr^\kappa\|_{L^2} + \|\nabla \vs^\kappa\|_{L^2}\right) \left(\|\nabla \Delta \vr^\kappa\|_{L^2}^2 + \|\nabla \Delta \vs^\kappa\|_{L^2}^2\right)\\
&\qquad + C \left(\|\nabla \vr^\kappa\|_{L^2}^2 + \|\nabla \vs^\kappa\|_{L^2}^2 + \|\vr^\kappa\|_{L^4} \|\vs^\kappa\|_{H^2} + \|\vr^\kappa\|_{L^4} \|\vs^\kappa\|_{H^2}\right) \left(\|\nabla \Delta \vr^\kappa\|_{L^2} + \|\nabla \Delta \vs^\kappa\|_{L^2}\right)\\
&\qquad + C \left(\|\Delta v^\kappa\|_{L^2} \|\nabla v^\kappa\|_{L^2} + \|\nabla v^\kappa\|_{L^2}^2 + \|\nabla \vr^\kappa\|_{L^2} + \|\nabla \vs^\kappa\|_{L^2}\right) \left(\|\Delta \vr^\kappa\|_{L^2}^2 + \|\Delta \vs^\kappa\|_{L^2}^2\right)\\
&\qquad + C \left(\|\nabla \vr^\kappa\|_{L^2}^2 + \|\nabla \vs^\kappa\|_{L^2}^2\right) \left(\|\Delta \vr^\kappa\|_{L^2} + \|\Delta \vs^\kappa\|_{L^2}\right) + C \left(\|\vr^\kappa\|_{H^3}^2 + \|\vs^\kappa - \bar{\vs}^\kappa\|_{H^3}^2\right)\\
&\qquad + C \|\nabla v^\kappa\|_{L^\infty} \left(\|\nabla \Delta \vr^\kappa\|_{L^2} + \|\nabla \Delta \vs^\kappa\|_{L^2}\right) + C \|v^\kappa\|_{L^\infty}^2\\
&\quad\leq C \left(\|v^\kappa\|_{H^3}^2 + \|\vr^\kappa\|_{H^3}^2 + \|\vs^\kappa - \bar{\vs}^\kappa\|_{H^3}^2\right).
\end{split}
\end{align}
We now turn to the $H^3$-estimates for the velocity difference $v^\kappa$. Testing \eqref{vkt} with $- \Delta^3 v^\kappa$, we obtain
\begin{align}
\label{vkH4rhs}
\begin{split}
\frac{1}{2} \frac{\diff}{\diff{t}} \|\nabla \Delta v^\kappa\|_{L^2}^2 = J_{2, 1} + J_{2, 2} + J_{2, 3},
\end{split}
\end{align}
where
\begin{align*}
J_{2, 1} & = \int \nabla \Delta v^\kappa \cdot \nabla \Delta \left(v^\kappa \cdot \nabla v^\kappa\right) + (k_B T_K) \int \nabla \Delta v^\kappa \cdot \nabla \Delta \left(\vr^\kappa \nabla \Psi^\kappa\right),\\
J_{2, 2} & = - \int \nabla \Delta (u^\kappa \cdot \nabla v^\kappa) : \nabla \Delta v^\kappa,\\
J_{2, 3} & = - \int \nabla \Delta (v^\kappa \cdot \nabla u^\kappa) : \nabla \Delta v^\kappa - (k_B T_K) \int \nabla \Delta \left(\rho^\kappa \nabla \Psi^\kappa + \vr^\kappa \nabla \Phi^\kappa\right) : \nabla \Delta v^\kappa.
\end{align*}
Notice that the term $J_{2, 1}$ comes from the left hand side of \eqref{vkt}, which is the same as the right hand side of \eqref{euler}. Then if we repeat the estimates \eqref{Dalphau} and use \eqref{ksolest}, we have that
\begin{align}
\label{J21}
\begin{split}
J_{2, 1} & \leq C \|\nabla v^\kappa\|_{L^\infty} \|v^\kappa\|_{H^3}^2 + C \|v^\kappa\|_{H^3} \left(\|\vr^\kappa\|_{L^\infty} \|\nabla \Psi^\kappa\|_{H^3} + \|\nabla \Psi^\kappa\|_{L^\infty} \|\vr^\kappa\|_{H^3}\right)\\
& \leq C \left(\|v^\kappa\|_{H^3}^2 + \|\vr^\kappa\|_{H^3}^2\right).
\end{split}
\end{align}
Applying the Leibniz rule to $J_{2, 2}$, the most dangerous term vanishes,
\[
- \int u^\kappa \cdot (\nabla \nabla \Delta v^\kappa : \nabla \Delta v^\kappa) = - \frac{1}{2} \int u^\kappa \cdot \nabla |\nabla \Delta v^\kappa|^2 = 0,
\]
because of the incompressibility of $u^\kappa$. The lower order terms can be estimated using the calculus inequality \eqref{fgHm} and the bounds \eqref{ksolest} so that
\begin{align}
\label{J22}
J_{2, 2} \leq C \|v^\kappa\|_{H^3}^2.
\end{align}
For the integrals in $J_{2, 3}$, we use the calculus inequality \eqref{fgHm} to get
\begin{align}
\label{J23}
\begin{split}
J_{2, 3} & \leq C \big(\|\nabla u^\kappa\|_{L^\infty} \|v^\kappa\|_{H^3} + \|u^\kappa\|_{H^4} \|v^\kappa\|_{L^\infty} + \|\rho^\kappa\|_{H^3} \|\nabla \Psi^\kappa\|_{L^\infty} + \|\rho^\kappa\|_{L^\infty} \|\nabla \Psi^\kappa\|_{H^3}\\
& \qquad + \|\nabla \Phi^\kappa\|_{L^\infty} \|\vr^\kappa\|_{H^3} + \|\nabla \Phi^\kappa\|_{H^3} \|\vr^\kappa\|_{L^\infty}\big) \|v^\kappa\|_{H^3}\\
& \leq C \big(\|v^\kappa\|_{H^3} + \|u^\kappa\|_{H^4} \|v^\kappa\|_{L^\infty} + \|\vr^\kappa\|_{H^3}\big) \|v^\kappa\|_{H^3},
\end{split}
\end{align}
where in the last inequality, we utilised the estimates $\|\nabla \Psi^\kappa\|_{L^\infty} \leq C \|\nabla \Psi^\kappa\|_{H^3} \leq C \|\vr^\kappa\|_{H^3}$, the Sobolev embedding $H^2(\T^2) \hookrightarrow L^\infty(\T^2)$, and the bounds \eqref{ksolest}.

The problematic term is $\|u^\kappa\|_{H^4} \|v^\kappa\|_{L^\infty} \|v^\kappa\|_{H^3}$. We need to use the lower norm $\|v^\kappa\|_{L^\infty}$ to balance the higher norm $\|u^\kappa\|_{H^4}$. For any $1 < s < 2$, we estimate
\begin{align}
\label{vkHs1}
\frac{\diff}{\diff{t}} \|v^\kappa\|_{H^s} \leq C (\|v^\kappa\|_{H^3} + \|u^\kappa\|_{H^3}) \|v^\kappa\|_{H^s} + C (\|\rho^\kappa\|_{H^3} + \|\vr^\kappa\|_{H^3}) \|\vr^\kappa\|_{H^s},
\end{align}
and
\begin{align}
\label{vkHs2}
\begin{split}
& \frac{\diff}{\diff{t}} \left(\|\vr^\kappa\|_{H^s} + \|\vs^\kappa\|_{H^s}\right)\\
&\quad\leq C \left(\|v^\kappa\|_{H^3} + \|u^\kappa\|_{H^3} + \|\rho^\kappa\|_{H^3} + \|\vr^\kappa\|_{H^3} + \|\sigma^\kappa\|_{H^3} + \|\vs^\kappa\|_{H^3}\right) \left(\|\vr^\kappa\|_{H^s} + \|\vs^\kappa\|_{H^s}\right).
\end{split}
\end{align}
Summing \eqref{vkHs1} and \eqref{vkHs2}, applying Gr\"{o}nwall's inequality, and invoking the bounds in \eqref{initk}, we obtain
\begin{align}
\|v^\kappa\|_{L^\infty(0, T; H^s)} + \|\vr^\kappa\|_{L^\infty(0, T; H^s)} + \|\vs^\kappa\|_{L^\infty(0, T; H^s)} \leq C \kappa^{3 - s}.
\end{align}
Thus, from \eqref{H4kest} and the Sobolev embedding $\|v^\kappa\|_{L^\infty} \leq C \|v^\kappa\|_{H^s}$, we deduce
\begin{align}
\label{vkHs4}
\|u^\kappa\|_{H^4} \|v^\kappa\|_{L^\infty} \|v^\kappa\|_{H^3} \leq C \kappa^{-1} \kappa^{3 - s} \|v^\kappa\|_{H^3} = C \kappa^{2 - s} \|v^\kappa\|_{H^3}.
\end{align}
Consequently, using \eqref{vkHs4} in \eqref{J23}, we have
\begin{align}
\label{J23-2}
J_{2, 3} \leq C \big(\|v^\kappa\|_{H^3} + \|\vr^\kappa\|_{H^3}\big) \|v^\kappa\|_{H^3} + C \kappa^{2 - s} \|v^\kappa\|_{H^3}.
\end{align}
Now gathering the estimates \eqref{J21}--\eqref{J22} and \eqref{J23-2} in \eqref{vkH4rhs}, we obtain the estimate
\begin{align}
\label{vkH3}
\frac{\diff}{\diff{t}} \|v^\kappa\|_{H^3}^2 \leq  C \left(\|v^\kappa\|_{H^3}^2 + \|\vr^\kappa\|_{H^3}^2\right) + C \kappa^{2 - s} \|v^\kappa\|_{H^3},
\end{align}
where $1 < s < 2$.

Summing \eqref{vrvskH3} and \eqref{vkH3} and dividing by $(\|\vr^\kappa\|_{H^3} + \|\vs^\kappa\|_{H^3} + \|v^\kappa\|_{H^3})$, we have that
\[
\frac{\diff}{\diff{t}} \left(\|\vr^\kappa\|_{H^3} + \|\vs^\kappa\|_{H^3} + \|v^\kappa\|_{H^3}\right) \leq C \left(\|v^\kappa\|_{H^3} + \|\vr^\kappa\|_{H^3} + \|\vs^\kappa\|_{H^3}\right) + C \kappa^{2 - s},
\]
which implies $(\|\vr^\kappa\|_{L^\infty(0, T; H^3)} + \|\vs^\kappa - \bar{\vs}^\kappa\|_{L^\infty(0, T; H^3)} + \|v^\kappa\|_{L^\infty(0, T; H^3)}) \to 0$ as $\kappa \to 0$. Indeed, we have
\begin{align}
\label{H3conv1}
\begin{split}
& \|\rho^\kappa - \rho\|_{L^\infty(0, T; H^3)} + \|\sigma^\kappa - \sigma\|_{L^\infty(0, T; H^3)} + \|u^\kappa - u\|_{L^\infty(0, T; H^3)}\\
&\quad\leq C \left(\|c_1^\kappa(0) - c_1(0)\|_{H^3} + \|c_2^\kappa(0) - c_2(0)\|_{H^3} + \|u_0^\kappa - u_0\|_{H^3} + \kappa^{2 - s} T\right).
\end{split}
\end{align}

\vspace{.5ex}\noindent {\bf Step 2.}
We now derive the energy estimates for the difference between the solutions $(\rho^\nu, \sigma^\nu, u^\nu)$ to the initial value problem for the Nernst-Planck-Navier-Stokes system and the solutions $(\rho^\kappa, \sigma^\kappa, u^\kappa)$ to the Nernst-Planck-Euler system with regularized initial data $(\rho^\kappa(0), \sigma^\kappa(0), u^\kappa_0)$. We denote by $(\vr^\nk, \vs^\nk, v^\nk) = (\rho^\nu - \rho^\kappa, \sigma^\nu - \sigma^\kappa, u^\nu - u^\kappa)$ and by $\Psi^\nk$ the solutions to the Poisson equation $- \ve \Delta \Psi^\nk = \vr^\nk$. Note that by the triangle inequality, $(\vr^\nk, \vs^\nk, v^\nk)$ also satisfy the bounds \eqref{ksolest}. Similar to \eqref{vrvskH3}, we have that
\begin{align}
\label{rsnkH3}
\begin{split}
& \frac{\diff}{\diff{t}} \left(\|\nabla \Delta \vr^\nk\|_{L^2}^2 + \|\nabla \Delta \vs^\nk\|_{L^2}^2\right) \leq C \left(\|v^\nk\|_{H^3}^2 + \|\vr^\nk\|_{H^3}^2 + \|\vs^\nk\|_{H^3}^2\right).
\end{split}
\end{align}

The $H^3$-estimates for $v^\nk$ requires special attentions. The equation satisfied by $v^\nk$ is
\begin{align*}
& \pt v^\nk + v^\nk \cdot \nabla v^\nk - \nu \Delta v^\nk + \nabla (p^\nu - p^\kappa)\\
&\quad= \nu \Delta u^\kappa - u^\kappa \cdot \nabla v^\nk - v^\nk \cdot \nabla u^\kappa - (k_B T_K) \left(\vr^\nk \nabla \Psi^\nk + \vr^\nk \nabla \Phi^\kappa + \rho^\kappa \nabla \Psi^\nk\right).
\end{align*}
Testing with $- \Delta^3 v^\nk$ and integrating by parts, we obtain
\begin{align}
\label{vnkH3}
\frac{1}{2} \frac{\diff}{\diff{t}} \|\nabla \Delta v^\nk\|_{L^2}^2 + \nu \|\Delta^2 v^\nk\|_{L^2}^2 = J_{3, 1} + J_{3, 2} + J_{3, 3} + J_{3, 4},
\end{align}
where
\begin{align*}
J_{3, 1} & = - \nu \int \Delta^2 u^\kappa : \Delta^2 v^\nk,\\
J_{3, 2} & = - \int \nabla \Delta (v^\nk \cdot \nabla v^\nk) : \nabla \Delta v^\nk - \int \nabla \Delta (u^\kappa \cdot \nabla v^\nk) : \nabla \Delta v^\nk,\\
J_{3, 3} & =  - \int \nabla \Delta (v^\nk \cdot \nabla u^\kappa) : \nabla \Delta v^\nk,\\
J_{3, 4} & = - (k_B T_K) \int \nabla \Delta (\vr^\nk \nabla \Psi^\nk + \vr^\nk \nabla \Phi^\kappa + \rho^\kappa \nabla \Psi^\nk) : \nabla \Delta v^\nk.
\end{align*}
To estimate the term $J_{3, 1}$, we use the Cauchy-Schwarz inequality and Young's inequality to obtain
\begin{align}
\label{J31}
J_{3, 1} \leq \nu \|\Delta^2 u^\kappa\|_{L^2} \|\Delta^2  v^\nk\|_{L^2} \leq \frac{\nu}{2} \|\Delta^2 v^\nk\|_{L^2}^2 + \frac{\nu}{2} \|\Delta^2 u^\kappa\|_{L^2}^2.
\end{align}
After using the Leibniz rule, we write $J_{3, 2}$
\begin{align}
J_{3, 2} = J_{3, 2, 1} + \Rc'_1,
\end{align}
where
\[
J_{3, 2, 1} = - \int v^\nk \cdot (\nabla \nabla \Delta v^\nk : \nabla \Delta v^\nk) - \int u^\kappa \cdot (\nabla \nabla \Delta v^\nk : \nabla \Delta v^\nk),
\]
and $\Rc_1$ represents the lower order terms, which we estimate using the Sobolev calculus inequality \eqref{fgHm}, Sobolev embeddings and \eqref{ksolest}
\begin{align}
\begin{split}
|\Rc_1| & \leq C \|\nabla v^\nk\|_{L^\infty} \|v^\nk\|_{H^3}^2 + C \|\nabla v^\nk\|_{L^\infty} \|u^\kappa\|_{H^3} \|v^\nk\|_{H^3} + C \|\nabla u^\kappa\|_{L^\infty} \|v^\nk\|_{H^3} \|v^\nk\|_{H^3}\\
& \leq C \|v^\nk\|_{H^3}^3 + C \|v^\nk\|_{H^3}^2.
\end{split}
\end{align}
The highest order term $J_{3, 2, 1}$ vanishes by integration by parts and the incompressibility of the velocity fields
\begin{align}
J_{3, 2, 1} = - \frac{1}{2} \int (v^\nk + u^\kappa) \cdot \nabla |\nabla \Delta v^\kappa|^2 = 0.
\end{align}
For $J_{3, 3}$, we have
\begin{align}
J_{3, 3} = J_{3, 3, 1} + \Rc'_2,
\end{align}
where
\begin{align}
J_{3, 3, 1} = \int v^\nk \cdot (\nabla \nabla \Delta u^\kappa : \nabla \Delta v^\nk) \leq C \|v^\nk\|_{L^\infty} \|u^\kappa\|_{H^4} \|v^\nk\|_{H^3},
\end{align}
and the lower order term $\Rc'_2$ satisfies the same bounds as $\Rc'_1$
\begin{align}
|\Rc'_2| \leq C \|v^\nk\|_{H^3}^3 + C \|v^\nk\|_{H^3}^2.
\end{align}
For the integrals involving the electric volume force, we apply \eqref{fgHm} to get
\begin{align}
\label{J34}
\begin{split}
J_{3, 4} & \leq C \big(\|\vr^\nk\|_{H^3} \|\nabla \Psi^\nk\|_{L^\infty} + \|\vr^\nk\|_{L^\infty} \|\nabla \Psi^\nk\|_{H^3} + \|\vr^\nk\|_{H^3} \|\nabla \Phi^\kappa\|_{L^\infty}\\
& \qquad + \|\vr^\nk\|_{L^\infty} \|\nabla \Phi^\kappa\|_{H^3}+ \|\rho^\kappa\|_{H^3} \|\nabla \Psi^\nk\|_{L^\infty} + \|\rho^\kappa\|_{L^\infty} \|\nabla \Psi^\nk\|_{H^3}\big) \|v^\nk\|_{H^3}\\
& \leq C \left(\|\vr^\nk\|_{H^3}^2 + \|\vr^\nk\|_{H^3}\right) \|v^\nk\|_{H^3},
\end{split}
\end{align}
where the last inequality follows from \eqref{ksolest} and the bounds $\|\nabla \Psi^\nk\|_{L^\infty} \leq \|\vr^\nk\|_{H^1}$.
Putting the estimates \eqref{J31}--\eqref{J34} into \eqref{vnkH3}, we arrive at
\begin{align}
\label{vnkL2}
\begin{split}
\frac{\diff}{\diff{t}} \|\nabla \Delta v^\nk\|_{L^2}^2 + \nu \|\Delta^2 v^\nk\|_{L^2}^2 & \leq \nu \|u^\kappa\|_{H^4}^2 + C \|v^\nk\|_{L^\infty} \|u^\kappa\|_{H^4} \|v^\nk\|_{H^3} + C \|v^\nk\|_{H^3}^3\\
& \qquad + C \|v^\nk\|_{H^3}^2 + C \left(\|\vr^\nk\|_{H^3}^2 + \|\vr^\nk\|_{H^3}\right) \|v^\nk\|_{H^3}.
\end{split}
\end{align}
Similar to \eqref{vkHs1}--\eqref{vkHs4}, we have
\[
\|v^\nk\|_{L^\infty} \|u^\kappa\|_{H^4} \|v^\nk\|_{H^3} \leq C \kappa^{2 - s} \|v^\nk\|_{H^3},
\]
for some $1 < s < 2$, and by \eqref{H4kest}, we also have
\[
\nu \|u^\kappa\|_{H^4}^2 \leq \frac{C \nu}{\kappa^2}.
\]
Hence, summing \eqref{vnkL2} and \eqref{rsnkH3}, discarding the dissipation term, and using the uniform $L^\infty(0, T; H^3(\T^2))$ bounds for $v^\nk$, we conclude that
\[
\frac{\diff}{\diff{t}} \left(\|v^\nk\|_{H^3}^2 + \|\vr^\nk\|_{H^3}^2 + \|\vs^\nk\|_{H^3}^2\right) \leq \frac{C \nu}{\kappa^2} + C \kappa^{2 - s} + C \left(\|v^\nk\|_{H^3}^2 + \|\vr^\nk\|_{H^3}^2 + \|\vs^\nk\|_{H^3}^2\right).
\]
Then by Gr\"{o}nwall's lemma and the triangle inequality, we deduce
\begin{align*}
& \|v^\nk\|_{L^\infty(0, T; H^3)}^2 + \|\vr^\nk\|_{L^\infty(0, T; H^3)}^2 + \|\vs^\nk - \bar{\vs}^\nk\|_{L^\infty(0, T; H^3)}^2\\
&\quad\leq C \left(\|u_0^\nu - u_0^\kappa\|_{H^3}^2 + \|c_1^\nu(0) - c_1^\kappa(0)\|_{H^3}^2 + \|c_2^\nu(0) - c_2^\kappa(0)\|_{H^3}^2\right) + C T \bigg(\frac{\nu}{\kappa^2} + \kappa^{2 - s}\bigg)\\
&\quad\leq C \left(\|u_0^\nu - u_0\|_{H^3}^2 + \|c_1^\nu(0) - c_1(0)\|_{H^3}^2 + \|c_2^\nu(0) - c_2(0)\|_{H^3}^2\right)\\
&\qquad + C \left(\|u_0 - u_0^\kappa\|_{H^3}^2 + \|c_1(0) - c_1^\kappa(0)\|_{H^3}^2 + \|c_2(0) - c_2^\kappa(0)\|_{H^3}^2\right) + C T \bigg(\frac{\nu}{\kappa^2} + \kappa^{2 - s}\bigg).
\end{align*}
We can choose, for instance,
\[
\kappa = \nu^{\frac{1}{3}},
\]
to obtain
\[
\lim_{\nu \to 0^+} \left(\|u_0 - u_0^\kappa\|_{H^3}^2 + \|c_1(0) - c_1^\kappa(0)\|_{H^3}^2 + \|c_2(0) - c_2^\kappa(0)\|_{H^3}^2\right) = 0
\]
by \eqref{initk}, and
\[
\lim_{\nu \to 0^+} \bigg(\frac{\nu}{\kappa^2} + \kappa^{2 - s}\bigg) = 0,
\]
because $1 < s < 2$. Finally, invoking the asymptotically matching initial condition \eqref{initconv} and the estimate \eqref{H3conv1} and applying the triangle inequality complete the proof of \eqref{H3conv}.
\end{proof}

We remark that a rate of convergence of the $H^3$-norms in Theorem~\ref{thm:vanish} can be obtained using the method in \cite{Mas07}, but we do not pursue it here.

\appendix

\section{Global well-posedness and uniqueness of the regularized system}
\label{sec:lwpreg}

In this appendix, we fix $\ell > 0$ in \eqref{rhol}--\eqref{approxinitell} and drop the parameter $\ell$ for simplicity of notation. We consider the system 
\begin{align}
& \pt \rho = - [u] \cdot \nabla \rho + D (\Delta \rho + \nabla \sigma \cdot \nabla \Phi + \sigma \Delta \Phi), \label{rhoa}\\
& \pt \sigma = - [u] \cdot \nabla \sigma + D (\Delta \sigma + \nabla \rho \cdot \nabla \Phi + \rho \Delta \Phi), \label{sigmaa}\\
& - \ve \Delta \Phi = \rho, \la{poia}\\
& \pt u + [u] \cdot \nabla u + [\na u]^* u + \nabla p = - (k_B T_K) \rho \nabla \Phi, \label{ua}\\
& \nabla \cdot u = 0, \label{incompapprox}
\end{align}
with initial data
\begin{align}
\label{approxinit}
\begin{split}
& \rho(\cdot, 0) = \rho(0) \in W^{1, r}(\T^2),\\
& \sigma(\cdot, 0) = \sigma(0) \in W^{1, r}(\T^2),\\
&\sigma(0) \ge |\rho(0)|,\\
& u(\cdot, 0) = u_0 \in W^{1,r}(\T^2) \quad \text{and} \quad \div u_0 = 0
\end{split}
\end{align}
with $r\ge 2$.
We prove first the local existence and uniqueness of solutions of the system \eqref{rhoa}--\eqref{incompapprox} with initial data \eqref{approxinit}.

\begin{proposition}
\label{loca}
Given $r \ge 2$ and $\ell > 0$, there exists a time $T_0 > 0$ depending only on $r$, $M_r$, and $\ell$ such that for any initial data with
\[
\|u_0\|_{W^{1,r}(\T^2)} + \|\rho(0)\|_{W^{1, r}(\T^2)} + \|\sigma(0)\|_{W^{1, r}(\T^2)} \le  M_r,
\]
there exists a unique solution of \eqref{rhoa}--\eqref{incompapprox} with initial data \eqref{approxinit}. The solution satisfies
\be
\sup_{t \in [0, T_0]} \|u(t)\|_{W^{1,r}(\T^2)} + \|\rho(t)\|_{W^{1, r}(\T^2)} + \|\sigma(t)\|_{W^{1, r}(\T^2)} \le 6M_r.
\la{mrbound}
\ee
\end{proposition}

\begin{proof}
We use an iteration to construct a sequence of approximate solutions for the system \eqref{rhoa}--\eqref{incompapprox}. The sequence of approximate solutions is denoted  $\{(\rho^n, \sigma^n, u^n, p^n)\}_{n \in \N}$. The iteration we consider is
\begin{align}
& \pt \rho^{n + 1} = - [u^n] \cdot \nabla \rho^{n + 1} + D (\Delta \rho^{n + 1} + \nabla \sigma^{n + 1} \cdot \nabla \Phi^n + \sigma^{n  + 1} \Delta \Phi^n), \label{rhoapprox1}\\
& \pt \sigma^{n + 1} = - [u^n] \cdot \nabla \sigma^{n + 1} + D (\Delta \sigma^{n + 1} + \nabla \rho^{n + 1} \cdot \nabla \Phi^n + \rho^{n + 1} \Delta \Phi^n), \label{sigmaapprox1}\\
& - \ve \Delta \Phi^{n + 1} = \rho^{n + 1}, \label{poissonapprox1}\\
& \pt u^{n + 1} + [u^n] \cdot \nabla u^{n + 1} + [\na u^n]^* u^{n+1} + \nabla p^{n + 1} = - (k_B T_K) \rho^{n + 1} \nabla \Phi^n, \label{eulerapprox1}\\
& \nabla \cdot u^{n + 1} = 0. \label{incompapprox1}
\end{align}
Given $\rho^n, \sigma^n, \Phi^n, u^n, p^n$, the system \eqref{rhoa}--\eqref{incompapprox1} is a linear system for $\rho^{n+1}, \sigma^{n+1}, \Phi^{n+1}, u^{n+1}, p^{n+1}$ and preserves the positivity of
\[
c_1^{n+1} = \frac{\sigma^{n+1} + \rho^{n+1}}{2} \quad \text{and} \quad c_2^{n+1} = \frac{\sigma^{n+1} - \rho^{n+1}}{2}.
\]
We prove bounds by induction. We assume that
\be
\sup_{0\le t\le T_0}\left(\|\rho^{n}(t)\|_{W^{1,r}} + \|\sigma^n(t)\|_{W^{1,r}} + \|u^{n}(t)\|_{W^{1,r}}\right)\le 6M_r.
\la{indu}
\ee
To simplify the notation, we denote
\begin{align}
\label{prime}
(\rho', \sigma', u', \Phi') = (\rho^{n + 1}, \sigma^{n + 1}, u^{n + 1}, \Phi^{n + 1}) \quad \text{and} \quad (\rho, \sigma, u, \Phi) = (\rho^n, \sigma^n, u^n, \Phi^n).
\end{align}
We start by estimating the $L^p$-norms of $\rho'$ and $\sigma'$ for any $p\ge 2$. Testing \eqref{rhoa} and \eqref{sigmaa} with $\frac{1}{p - 1} |\rho'|^{p - 2} \rho'$ and $\frac{1}{p - 1} |\sigma'|^{p - 2} \rho'$ respectively and using \eqref{poia}, we have 
\begin{align*}
& \frac{1}{p (p - 1)} \frac{\diff}{\diff{t}} \left(\|\rho'\|_{L^p}^p + \|\sigma'\|_{L^p}^p\right) + D \int |\rho'|^{p - 2} |\nabla \rho'|^2 + D \int |\sigma'|^{p - 2} |\nabla \sigma'|^2\\
&\quad= - D \int |\rho'|^{p - 2} \sigma' \nabla \rho' \cdot \nabla \Phi - D \int |\sigma'|^{p - 2} \rho' \nabla \sigma' \cdot \nabla \Phi.
\end{align*}
By H\"{o}lder's inequailty and Young's inequality, we obtain
\begin{align}
\label{rhosigmaLp00}
& \frac{1}{p (p - 1)} \frac{\diff}{\diff{t}} \left(\|\rho'\|_{L^p}^p + \|\sigma'\|_{L^p}^p\right) + D \int |\rho'|^{p - 2} |\nabla \rho'|^2 + D \int |\sigma'|^{p - 2} |\nabla \sigma'|^2\notag\\
&\quad\leq \frac{D}{2} \int |\rho'|^{p - 2} |\nabla \rho'|^2 + \frac{D}{2} \int |\sigma'|^{p - 2} |\nabla \sigma'|^2 + \frac{D}{2} \int |\rho'|^{p - 2} |\sigma'|^2 |\nabla \Phi|^2 + \frac{D}{2} \int |\sigma'|^{p - 2} |\rho'|^2 |\nabla \Phi|^2.
\end{align}
Therefore, we obtain
\begin{align}
\label{rhosigmaLp0}
\frac{\diff}{\diff{t}} \left(\|\rho'\|_{L^p}^p + \|\sigma'\|_{L^p}^p\right) \leq \frac{D p (p - 1)}{2} \|\nabla \Phi\|_{L^\infty}^2 \left(\|\rho'\|_{L^p}^p + \|\sigma'\|_{L^p}^p\right).
\end{align}
In view of the elliptic estimate
\be
\|\na\Phi\|_{L^{\infty}}\le C\|\rho\|_{L^3}
\la{naphirho}
\ee
and the induction assumption \eqref{indu} we have
\be
\sup_{0\le t\le T_0}\left (\|\rho'(t)\|_{L^p} + \|\sigma'(t)\|_{L^p}\right) \le C_pM_r\exp(CT_0M_r^2)
\la{lrnorms}
\ee
for $p\ge 2$, where we used the embedding $W^{1,r}\hookrightarrow L^p$ and estimated $\rho(0)$ and $\sigma(0)$ in $L^p$ in terms of $M_r$.
Also, in view of the fact that $r\ge 2$ we have
\be
\int_0^{T_0}\|\na\rho'(t)\|_{L^2}^2 + \|\na \sigma'(t)\|_{L^2}^2 dt \le CM_r^2(1+ T_0M_r^2\exp(CT_0 M_r^2)) = N_r^2
\la{inth1}
\ee
which follows from \eqref{rhosigmaLp00} and \eqref{lrnorms} for $p=2$.
Note that the assumption
\be
CT_0M_r^2 <\log \fr{3}{2}
\la{tmr}
\ee
implies
\be
N_r \le C M_r
\la{nrmr}
\ee
and
\be
\sup_{0\le t\le T_0}\left (\|\rho'(t)\|_{L^r} + \|\sigma'(t)\|_{L^r}\right) \le \fr{3}{2} M_r.
\la{lrnormb}
\ee
Taking the curl of the equation (\ref{ua}) we have
\be
\pa_t \omega' + [u]\cdot \na \omega' = -k_BT_K\na^{\perp}\rho'\cdot\na\Phi. 
\la{omegaprime}
\ee
We have thus, in view of (\ref{indu}), (\ref{naphirho}) and (\ref{inth1}) that
\be
\sup_{0\le t\le T_0}\|\omega'(t)\|_{L^r} \le M_r(1 + CM_rN_r\sqrt{T_0}).
\la{vortr}
\ee
We remark that the estimates can be closed at the level $L^3$ for the concentrations and $H^1$ for the velocity. 

We continue by estimating gradients of concentrations.
From \eqref{rhoapprox1}--\eqref{sigmaapprox1}, we have that
\begin{align}
\label{rhosigmaLp}
\begin{split}
& \frac{1}{r} \frac{\diff}{\diff{t}} \left(\|\nabla \rho'\|_{L^r}^r + \|\nabla \sigma'\|_{L^r}^r\right) + \mathcal{D}_1'\\
&\quad= - \int |\nabla \rho'|^{r - 2} \nabla \rho' \cdot (\nabla [u])^* \nabla \rho' - \int |\nabla \sigma'|^{r - 2} \nabla \sigma' \cdot (\nabla [u])^* \nabla \sigma'\\
&\qquad - D \int |\nabla \rho'|^{r - 2} \Delta \rho' \nabla \sigma' \cdot \nabla \Phi - D \int |\nabla \rho'|^{r - 2} \Delta \rho' \sigma' \Delta \Phi - D \int \nabla \rho' \cdot \nabla |\nabla \rho'|^{r - 2} \sigma' \Delta \Phi\\
&\qquad  - D (r - 2) \int \nabla \rho' \cdot (\nabla \nabla \rho') \cdot \nabla \rho' |\nabla \rho'|^{r - 4} \nabla \sigma' \cdot \nabla \Phi\\
&\qquad - D \int |\nabla \sigma'|^{r - 2} \Delta \sigma' \nabla \rho' \cdot \nabla \Phi - D \int |\nabla \sigma'|^{r - 2} \Delta \sigma' \rho' \Delta \Phi - D \int \nabla \sigma' \cdot \nabla |\nabla \sigma'|^{r - 2} \rho' \Delta \Phi\\
&\qquad  - D (r - 2) \int \nabla \sigma' \cdot (\nabla \nabla \sigma') \cdot \nabla \sigma' |\nabla \sigma'|^{r - 4} \nabla \rho' \cdot \nabla \Phi,
\end{split}
\end{align}
where
\[
\mathcal{D}_1' =  D \int |\nabla \rho'|^{r - 2} |\nabla \nabla \rho'|^2 + D \int |\nabla \sigma'|^{r - 2} |\nabla \nabla \sigma'|^2 + \frac{4 D (r - 2)}{r^2} \int \left|\nabla |\nabla \rho'|^{\frac{r}{2}}\right|^2 + \frac{4 D (r - 2)}{r^2} \int \left|\nabla |\nabla \sigma'|^{\frac{p}{2}}\right|^2.
\]
In a manner as for estimates \eqref{DrhoLp}--\eqref{Ygrowth}, we denote
\[
Y' = \|\nabla \rho'\|_{L^r}^r + \|\nabla \sigma'\|_{L^r}^r = \|R'\|_{L^2}^2 + \|S'\|_{L^2}^2, \quad R' = |\nabla \rho'|^{\frac{r}{2}}, \quad S' = |\nabla \sigma'|^{\frac{r}{2}}.
\]
Then using H\"{o}lder's inequality and Young's inequality in \eqref{rhosigmaLp}, we obtain
\begin{align*}
\frac{\diff}{\diff{t}} Y' + \mathcal{D}_1' & \leq \int |\nabla [u]| (|R'|^2 + |S'|^2) + \frac{D}{2} \int |\nabla \rho'|^{r - 2} |\nabla \nabla \rho'|^2 + \frac{D}{2} \int |\nabla \sigma'|^{r - 2} |\nabla \nabla \sigma'|^2\\
& \quad + 2 D ((r - 2)^2 + 1) \|\nabla \Phi\|_{L^\infty}^2 \left(\|\nabla \rho'\|_{L^r}^{r - 2} \|\nabla \sigma'\|_{L^r}^2 + \|\nabla \sigma'\|_{L^r}^{r - 2} \|\nabla \rho'\|_{L^r}^{r - 2}\right)\\
& \quad + \frac{2 D}{\ve^2} ((r - 2)^2 + 1) \int \rho^2 \left(|\nabla \rho'|^{r - 2} |\sigma'|^2 + |\nabla \sigma'|^{r - 2} |\rho'|^2\right),
\end{align*}
with
\[
\mathcal{D}_1' =  D \int |\nabla \rho'|^{r - 2} |\nabla \nabla \rho'|^2 + D \int |\nabla \sigma'|^{r - 2} |\nabla \nabla \sigma'|^2 + \frac{4 D (r - 2)}{r^2} \left(\|\nabla R'\|_{L^2}^2 + \|\nabla S'\|_{L^2}^2\right).
\]
Then we have
\begin{align}
\label{Yest'1}
\begin{split}
\frac{\diff}{\diff{t}} Y' + \mathcal{D}_2' & \leq \int |\nabla [u]| (|R'|^2 + |S'|^2) + 2 D ((r - 2)^2 + 1) \|\nabla \Phi\|_{L^\infty}^2 Y'\\
& \quad + \frac{2 D}{\ve^2} ((r - 2)^2 + 1) \int \rho^2 (|\rho'|^2 + |\sigma'|^2) \left(|R'|^{\frac{2 r - 4}{r}} + |S'|^{\frac{2 r - 4}{r}}\right),
\end{split}
\end{align}
where
\begin{align*}
\mathcal{D}_2' & = \frac{D}{2} \int |\nabla \rho'|^{r - 2} |\nabla \nabla \rho'|^2 + \frac{D}{2} \int |\nabla \sigma'|^{r - 2} |\nabla \nabla \sigma'|^2 + \frac{4 D (r - 2)}{r^2} \left(\|\nabla R'\|_{L^2}^2 + \|\nabla S'\|_{L^2}^2\right)\\
& \geq \frac{4 D (r - 2)}{r^2} \left(\|R'\|_{H^1}^2 + \|S'\|_{H^1}^2 - \|R'\|_{L^2}^2 - \|S'\|_{L^2}^2\right).
\end{align*}
Using the Ladyzhenskaya inequality we have that
\begin{align}
\label{Yest'1-1}
\begin{split}
\int |\nabla [u]| (|R'|^2 + |S'|^2) & \leq \|\nabla [u]\|_{L^2} \left(\|R'\|_{L^4}^2 + \|S'\|_{L^4}^2\right)\\
& \leq \|\nabla [u]\|_{L^2} \left(\|R'\|_{L^2} \|\nabla R'\|_{L^2} + \|R'\|_{L^2}^2 + \|S'\|_{L^2} \|\nabla S'\|_{L^2} + \|S'\|_{L^2}^2\right)\\
& \leq \frac{D (r - 2)}{r^2} \left(\|R'\|_{H^1}^2 + \|S'\|_{H^1}^2\right) + C \|\nabla [u]\|_{L^2}^2 \left(\|R'\|_{L^2}^2 + \|S'\|_{L^2}^2\right).
\end{split}
\end{align}
By H\"{o}lder's inequlaity we obtain
\begin{align}
\label{Yest'1-2}
\begin{split}
& \frac{2 D r}{\ve^2} ((r - 2)^2 + 1) \int \rho^2 (\rho^2 + \sigma^2) \left(|R'|^{\frac{2 r - 4}{r}} + |S'|^{\frac{2 r - 4}{r}}\right)\\
&\quad\leq C \Big(\|R'\|_{L^2}^{2} + \|S'\|_{L^2}^{2}\Big)^{\frac{r - 2}{r}} \Big(\|\rho\|_{L^{2r}} + \|\sigma\|_{L^{2r}}\Big)^{4}\\
&\quad\leq \|R'\|_{L^2}^2 + \|S'\|_{L^2}^2 + C \Big(\|\rho\|_{L^{2r}} + \|\sigma\|_{L^{2r}}\Big)^{2r}.
\end{split}
\end{align}
Using the inequalities \eqref{Yest'1-1}--\eqref{Yest'1-2} in \eqref{Yest'1}, we have that
\be
\frac{\diff}{\diff{t}} Y' \leq C \left(1 + \|\nabla \Phi\|_{L^\infty}^2   + \|\omega\|_{L^2}^2\right) Y' + \Big(\|\rho\|_{L^{2r}} + \|\sigma\|_{L^{2r}}\Big)^{2r}.
\la{odeY'}
\ee
By the induction assumption (\ref{indu}) it follows 
\be
\sup_{0\le t\le T_0}Y'(t) \le M_r^r(1 + CM_r^{r} T_0)\exp(C(1+M_r^2)T_0)
\la{yprimei}
\ee
and taking the $r$-th root,
\be
\sup_{0\le t\le T_0}(\|\na\sigma'(t)\|_{L^r} + \|\na\rho'(t)\|_{L^r})
\le M_r(1+ CM_rT_0^{\fr{1}{r}})\exp\left({\frac{C(1+M_r^2)T_0}{r}}\right).
\la{normr}
\ee 
Choosing thus
\be
T_0 \le \delta M_r^{-r}
\la{finalchoice}
\ee
with $\delta>0$ sufficiently small and independent of $n$ we see from (\ref{lrnormb}) 
and (\ref{normr}) that
\be
\sup_{0\le t\le T}(\|\rho'\|_{W^{1,r}} + \|\sigma'(t)\|_{W^{1,r}}) \le 4M_r
\la{norms}
\ee
holds.
Returning to the vorticity equation \eqref{omegaprime} we can improve the inequality \eqref{vortr} by using \eqref{norms} (which was proved independently of it, using only induction), and obtain
\be
\sup_{0\le t\le r}\|\omega'(t)\|_{L^r}\le M_r(1+ CM_r^2T_0) \le 2M_r
\la{newvortr}
\ee
if $T_0$ satisfies (\ref{finalchoice}).  From (\ref{norms}) and (\ref{newvortr}) we have
\be
\sup_{0\le t\le T_0}\left(\|\rho^{n+1}(t)\|_{W^{1,r}} + \|\sigma^{n+1}(t)\|_{W^{1,r}} + \|u^{n+1}(t)\|_{W^{1,r}}\right)\le 6M_r
\la{induc}
\ee
and so the induction is complete. The solutions remain bounded in $W^{1,r}$ indpendently of $n$. We can pass to weakly convergent subsequences.

In order to prove convergence, we estimate differences of successive terms in a weaker norm. Going back to the notation \eqref{prime} as well as
\[
(\rho'', \sigma'', u'', \Phi'') = (\rho^{n + 2}, \sigma^{n + 2}, u^{n + 2}, \Phi^{n + 2}),
\]
we estimate $\|\rho'' - \rho'\|_{L^2}$, $\|\sigma'' - \sigma'\|_{L^2}$, and $\|u'' - u'\|_{L^2}$. Taking the differences of $(n + 2)$ equations and the $(n + 1)$ equations in the iteration, we deduce
\be
\begin{aligned}
& \frac{\diff}{\diff{t}} \left(\|\rho'' - \rho'\|_{L^2}^2 + \|\sigma'' - \sigma'\|_{L^2}^2\right) + D \left(\|\nabla (\rho'' - \rho')\|_{L^2}^2 + \|\nabla (\sigma'' - \sigma')\|_{L^2}^2\right)\\
&\quad\leq C \|[u'] - [u]\|_{L^4}^2 \left(\|\rho'\|_{L^4}^2 + \|\sigma'\|_{L^4}^2\right) + C \|\nabla \Phi'\|_{L^\infty}^2 \left(\|\rho'' - \rho'\|_{L^2}^2 + \|\sigma'' - \sigma'\|_{L^2}^2\right)\\
& \qquad + C \|\nabla (\Phi' - \Phi)\|_{L^4}^2 \left(\|\rho'\|_{L^4}^2 + \|\sigma'\|_{L^4}^2\right).
\la{difrhos}
\end{aligned}
\ee
Denoting 
\be
\delta_{n}^2 = \|\rho' -\rho\|^2_{L^2} + \|\sigma'-\sigma\|_{L^2}^2
\la{deltan}
\ee
and
\be
\upsilon_n = \|u'-u\|_{L^2}, 
\la{upsi}
\ee
we have from \eqref{difrhos} that
\be
\fr{d}{dt}\delta_{n+1}^2 \le CM_r^2\left( \delta_{n+1}^2 + \delta_n^2 + C\ell^{-2}\upsilon_n^2\right).
\la{diffrhosb}
\ee
Regarding the velocity, from \eqref{eulerapprox1}, we obtain 
\be
\fr{d}{dt} \upsilon_{n+1}^2 \le CM_r^2[C_{\ell}(\upsilon_n^2 + \upsilon_{n+1}^2) + \delta_{n+1}^2 + \delta_n^2]
\la{upsilons}
\ee
where we have to use $\ell$ to estimate $[\na u]$ in $L^{\infty}$ in terms of only the $L^2$ norm.
By induction, it follows that 
\be
\sup_{0\le t\le T_0} (\delta_{n+1}(t)^2 + \upsilon_{n+1}(t)^2) \le CM^n\fr{T_0^n}{n!}\Gamma_1
\la{ind}
\ee
with $M$ depending on $M_r$ and $\ell$ and $\Gamma_1$ depending on the first difference of solutions.
This implies that the sequence is Cauchy in $L^2$. Passing to weak limit in $W^{1,r}$ on a subsequence, and using the fact that the weak limit of the shifted subsequence is the same, we conclude the existence of solutions of
 \eqref{rhoa}--\eqref{incompapprox} with initial data \eqref{approxinit} which satisfy (\ref{mrbound}).
Their uniqueness follows by estimating the difference of two solutions in $L^2$, using the a~priori bounds in terms of $M_r$ and $\ell$, in a manner entirely similar to the estimates leading to (\ref{diffrhosb}) and (\ref{upsilons}).

\end{proof}

Now we prove the global existence of solutions.
\beg{proposition}Let $T>0$, $r\ge 2$, $\ell>0$. For any initial data with
\[
\|u_0\|_{W^{1,r}(\T^2)} + \|\rho(0)\|_{W^{1, r}(\T^2)} + \|\sigma(0)\|_{W^{1, r}(\T^2)} <\infty,
\]
there exists a unique solution of \eqref{rhoa}--\eqref{incompapprox} with initial data \eqref{approxinit}. The solution satisfies
\[
\sup_{t \in [0, T]} \|u(t)\|_{W^{1,r}(\T^2)} + \|\rho(t)\|_{W^{1, r}(\T^2)} +\
 \|\sigma(t)\|_{W^{1, r}(\T^2)} <\infty .
\]
\end{proposition}
\beg{proof}
The proof follows by contradiction from the a~priori bound results in Theorem \ref{unifellb} and the local existence result Proposition \ref{loca}. Indeed, the a~priori bounds \eqref{solest1} and \eqref{solest2} imply that the quantity
\begin{align}
\label{BKMapprox}
Q(T_2, T_1) = \int_{T_1}^{T_2} \|\nabla [u](\tau)\|_{L^\infty(\T^2)} + \|\rho(\tau)\|_{L^p(\T^2)}^p \diff{\tau} 
\end{align}
is bounded in terms of only the $W^{1,r}$ norms at $T_1$ and independently of time. This quantity can be used to bound the $W^{1,r}$ norms at time $T_2$. By the local existence Proposition \ref{loca} the supremum of times where the $W^{1,r}$ norms are finite cannot be strictly less than $T$.

\end{proof}

\end{document}